\documentclass[12pt]{article}
\usepackage{graphicx,color}
\usepackage{amsmath,amsfonts,amssymb}
\usepackage{hyperref}

\usepackage[numbers]{natbib}

\newcommand{\rs}[1]{{\mbox{\scriptsize \sc #1}}}
\newcommand{\vc}[1]{{\boldsymbol #1}}

\newcommand{\sr}[1]{{\cal #1}}
\newcommand{\dd}[1]{\mathbb{#1}}

\newcommand{\cp}{{\Gamma}}
\newcommand{\rmn}[1]{\if#11I\else {\if#12I\hspace{-0.12ex}I\hspace{-0.85ex}\else {\if #13I\hspace{-0.16ex}I\hspace{-0.16ex}I\hspace{-1.6ex}\else I\hspace{-1.2ex}V \fi} \fi} \fi}

\newcommand{\rt}[1]{({\mbox{\scriptsize \sc #1}})}

\newcommand{\eqn}[1]{(\ref{eqn:#1})}
\newcommand{\lem}[1]{Lemma~\ref{lem:#1}}

\newcommand{\thr}[1]{Theorem~\ref{thr:#1}}

\newcommand{\tab}[1]{Table~\ref{tab:#1}}

\newcommand{\ass}[1]{Assumption~\ref{ass:#1}}
\newcommand{\rem}[1]{Remark~\ref{rem:#1}}
\newcommand{\fig}[1]{Figure~\ref{fig:#1}}
\newcommand{\app}[1]{Appendix~\ref{app:#1}}
\newcommand{\sectn}[1]{Section~\ref{sect:#1}}

\newcommand{\thrt}[1]{\ref{thr:#1}}

\newcommand{\remt}[1]{\ref{rem:#1}}
\newcommand{\sect}[1]{\ref{sect:#1}}

\newcommand{\br}[1]{\langle #1 \rangle}

\newcommand{\ol}{\overline}
\newcommand{\ul}{\underline}
\newcommand{\pend}{\hfill \thicklines \framebox(6.6,6.6)[l]{}}

\newenvironment{proof}{\noindent {\sc  Proof.} \rm}{\pend}
\newenvironment{proof*}[1]{\noindent {\sc  #1} \rm}{\pend}

\newtheorem{theorem}{Theorem}[section]
\newtheorem{lemma}{Lemma}[section]
\newtheorem{assumption}{Assumption}[section]

\newtheorem{remark}{Remark}[section]

\newtheorem{corollary}{Corollary}[section]

\newcommand{\setsection}[2] {
\setcounter{section}{#1}
\setcounter{subsection}{0}
\setcounter{equation}{0}
\setcounter{conjecture}{0}
\setcounter{assumption}{0}
\setcounter{question}{0}
\setcounter{definition}{0}
\setcounter{theorem}{0}
\setcounter{corollary}{0}
\setcounter{lemma}{0}
\setcounter{proposition}{0}
\setcounter{remark}{0}
\setcounter{appen}{0}
\setsection*{\large \bf \thesection. #2}}

\newenvironment{mylist}[1]{\begin{list}{}
{\setlength{\itemindent}{#1mm}}
{\setlength{\itemsep}{0ex plus 0.2ex}}
{\setlength{\parsep}{0.5ex plus 0.2ex}}
{\setlength{\labelwidth}{10mm}}
}{\end{list}}

\newcommand{\setnewcounter} {
\setcounter{subsection}{0}
\setcounter{equation}{0}
\setcounter{conjecture}{0}
\setcounter{assumption}{0}
\setcounter{question}{0}
\setcounter{definition}{0}
\setcounter{theorem}{0}
\setcounter{corollary}{0}
\setcounter{lemma}{0}
\setcounter{proposition}{0}
\setcounter{remark}{0}
}

\begin{document}
\title{\bf \Large A superharmonic vector for a nonnegative matrix with QBD block structure and its application to a Markov modulated two dimensional reflecting process}

\author{Masakiyo Miyazawa}
\date{Revised, July 17, 2015}

\maketitle

\begin{abstract}
  Markov modulation is versatile in generalization for making a simple stochastic model which is often analytically tractable to be more flexible in application. In this spirit, we modulate a two dimensional reflecting skip-free random walk in such a way that its state transitions in the boundary faces and interior of a nonnegative integer quadrant are controlled by Markov chains. This Markov modulated model is referred to as a 2d-QBD process according to \citet{Ozaw2013}. We are interested in the tail asymptotics of its stationary distribution, which has been well studied when there is no Markov modulation. 
  
  Ozawa studied this tail asymptotics problem, but his answer is not analytically tractable. We think this is because Markov modulation is so free to change a model even if the state space for Markov modulation is finite. Thus, some structure, say, extra conditions, would be needed to make the Markov modulation analytically tractable while minimizing its limitation in application. 
  
  The aim of this paper is to investigate such structure for the tail asymptotic problem. For this, we study the existence of a right subinvariant positive vector, called a superharmonic vector, of a nonnegative matrix with QBD block structure, where each block matrix is finite dimensional. We characterize this existence under a certain extra assumption. We apply this characterization to the 2d-QBD process, and derive the tail decay rates of its marginal stationary distribution in an arbitrary direction. This solves the tail decay rate problem for a two node generalized Jackson network, which has been open for many years.
\end{abstract}
  
\begin{quotation}
\end{quotation}

\section{Introduction}
\label{sect:Introduction}

Our primary interest is in methodology for deriving the tail asymptotics of the stationary distribution of a Markov modulated two dimensional reflecting random walk for queueing network applications, provided it is stable. This process has two components, front and background processes. We assume that the front process is a skip-free reflecting random walk on a nonnegative quarter plane of lattice, and the background process has finitely many states. We are particularly interested in a two node generalized Jackson network for its application.

According to \citet{Ozaw2013}, we assume the following transition structure. The state space of the front process is composed of the inside of the quarter plane and three boundary faces, the origin and the two half coordinate axes. Within each region, state-transitions are homogeneous, that is, subject to a Markov modulated random walk, but different regions may have different state-transitions. Between pairs of the four regions, state-transitions may also be different. See \fig{RRW1} in \sectn{two} for their details. This Markov modulated two dimensional random walk is called a discrete-time 2d-QBD process, 2d-QBD process for short, in \cite{Ozaw2013}. We adopt the same terminology. This process is flexible enough to handle many two node queueing networks in continuous time through uniformization. The generalized Jackson network is such an example.

For the 2d-QBD process, we assume that it has a stationary distribution, and denote a random vector subject to it by $(\vc{L}, J)$, where $\vc{L}$ represents a random walk component taking values in $\dd{R}_{+}^{2}$ while $J$ represents a background state. For $i=1,2$, we consider the tail asymptotics by logarithmic ratios of the stationary tail probabilities in the $i$-th coordinate directions:
\begin{align}
\label{eqn:decay coordinate}
  \frac 1n \log \dd{P}( L_{i} > n, L_{3-i} = \ell, J = k), \qquad n \to \infty,
\end{align}
for each fixed $\ell \ge 0$ and background state $k$, and those of the marginal stationary distribution in an arbitrary direction $\vc{c} \equiv (c_{1}, c_{2})$:
\begin{align}
\label{eqn:decay arbitrary}
  \frac 1x \log \dd{P}( c_{1} L_{1} + c_{2} L_{2} > x), \qquad x \to \infty.
\end{align}
It will be shown that those ratios converges to constants (Theorems \thrt{lower bound 1} and \thrt{decay rate 1}). They are negative, and their absolute values are called exponential decay rates. We demonstrate those tail asymptotic results for a two node generalized Jackson network with Markov modulated arrivals and phase type service time distributions. This solves the problem which has been open for many years (see \sectn{Jackson} for details).

\citet{Ozaw2013} studied the tail asymptotics in the coordinate directions including \eqn{decay coordinate}. He showed that the method for a two-dimensional reflecting random walk studied by \citet{Miya2009} is applicable with help of invariant vectors obtained by \citet{LiZhao2003}. We refer to this method as a QBD approach, which is composed of the following three key steps.
\begin{itemize}
\item [1)] Formulate the 2d-QBD process as a one dimensional QBD process with infinitely many background states, where one of the coordinate axes is taken as a level.
\item [2)] Find right and left invariant vectors of a nonnegative matrix with QBD block structure, which will be introduced shortly, and get upper and lower bounds of the tail decay rates.
\item [3)] Derive the tail decay rates, combining those results in the two directions.
\end{itemize}
Here, an infinite dimensional square matrix is said to have QBD block structure if it is partitioned into blocks in such a way that each block is a square matrix of the same size except for the first row and first column blocks, the whole matrix is block tridiagonal and each row of blocks is repeated and shifted except for the first two rows  (see \eqn{QBD structure} for its definite form). In step 1), the blocks for the one dimensional QBD are infinite dimensional, while, in step 2), those for the nonnegative matrix are finite dimensional.

A hard part of this QBD approach is in step 2). In \cite{Ozaw2013}, the invariant vectors are only obtained by numerically solving certain parametrized equations over a certain region of parameters. This much degrades applicability of the tail asymptotic results. For example, it is hard to get useful information from them for the tail asymptotics in the two node generalized Jackson network (see, e.g, \cite{FujiTakaMaki1998,KatoMakiTaka2004}). We think this analytic intractability can not be avoided because no structural condition is assumed for the Markov modulation. In applications, it may have certain structure. Thus, it is interesting to find conditions for the invariant vectors to be analytically tractable while minimizing limitations in application.

Another problem in \cite{Ozaw2013} is complicated descriptions. They can not be avoided because of the complicated modeling structure, but we easily get lost in computations. We think here simplification or certain abstraction is needed.

In addition to those two problems, the QBD approach is not so useful to study the tail asymptotics in an arbitrary direction. For this, it is known that the stationary inequalities in terms of moment generating functions are useful in the case that there is no Markov modulation (e.g., see \cite{KobaMiya2014,Miya2011}). So far, it is interesting to see whether this moment generating function approach still works under Markov modulation.

We attack those three problems in this paper. We first consider the description problem, and find a simpler matrix representation for a nonnegative matrix with QBD block structure. This representation is referred to as a canonical form. We then consider the problem in step 2). 

For this, we relax the problem by considering a right subinvariant positive vector, which is said to be superharmonic, instead of a right invariant positive vector, which is said to be harmonic. It is known that the existence of a right subinvariant vector is equivalent to that of a left subinvariant nonnegative vector (e.g., see \cite{Vere1967}). When a nonnegative matrix is stochastic, a right subinvariant vector can be viewed as a superharmonic function. Because of this fact, we use the terminology superharmonic vector. In the stochastic case, it obviously exists. When the matrix is substochastic and does not have the boundary blocks, this problem has been considered in studying a quasi-stationary distribution for QBD processes (see, e.g., \cite{Kiji1993,LiZhao2002,LiZhao2003}).

In step 2), we do not assume any stochastic or substochastic condition for a nonnegative matrix with QBD block structure, which is crucial in our applications. As we will see, we can find necessary and sufficient conditions for such a matrix to have a superharmonic vector (see \thr{characterization 1}). The sufficiency is essentially due to \citet{LiZhao2003} and related to \citet{BeanPollTayl2000} (see Remarks \remt{super-h 2} and \remt{characterization 1}). However, this characterization is not useful in application as we already discussed. So, we will assume a certain extra condition to make an answer to be tractable. Under this extra assumption, we characterize the existence of a superharmonic vector using primitive data on the block matrices (\thr{super-h 2}).

This characterization enables us to derive the tail asymptotics of the stationary distribution in the coordinate directions for the 2d-QBD process. For the problem of the tail asymptotics in an arbitrary direction, we show that the moment generating function approach can be extended for the Markov modulated case. For this, we introduce a canonical form for the Markov modulated two dimensional random walk, which is similar to that for a nonnegative matrix with QBD block structure.

There has been a lot work on tail asymptotic problems in queueing networks (see, e.g., \cite{Miya2011} and references therein). Most of studies focus on two dimensional reflecting processes or two node queueing networks. The 2d-QBD process belongs to this class of models, but allows them to have background processes with finitely many states. There is a huge gap between finite and infinite numbers of background states, but we hope the present results will stimulate to study higher dimensional tail asymptotic problems.

This paper is made up by five sections and appendices. \sectn{nonnegative} drives necessary and sufficient conditions for a nonnegative matrix with QBD block structure to have a right sub-invariant vector with and without extra assumptions. This result is applied to the 2d-QBD process, and the tail decay rates of its stationary distribution are derived in \sectn{application}. The tail decay rates of the marginal stationary distribution in an arbitrary direction are obtained for the generalized Jackson network in \sectn{two node}. We finally give some concluding remarks in \sectn{concluding}.

We summarize basic notation which will be used in this paper (see Tables 1 and 2).
\begin{table}[htdp]
\begin{center}
\begin{tabular}{llll}
  $\dd{Z}$ & the set of all integers, & $\dd{Z}_{+}$ & the set of all nonnegative integers,\\
  $\dd{R}$ & the set of all real numbers, $\,$ & $\dd{R}_{+}$ & the set of all nonnegative real numbers,\\
  $\dd{H}$ & $\{-1, 0, 1\}$, & $\dd{H}_{+}$ & $\{0, 1\}$,\\
  $\br{\vc{x}, \vc{y}}$ & $x_{1}y_{1} + x_{2}y_{2}$ for $\vc{x}, \vc{y} \in \dd{R}^{2}$, & $\vc{1}$ & the column vector whose entries are all units.
\end{tabular}
\caption{Notation for sets of numbers and vectors}
\label{tab:set}
\end{center}\vspace{-4ex}
\end{table}

 For nonnegative square matrices $T, T_{i}, T_{ij}$ with indices $i, j \in \dd{Z}$ such that $T_{i}$ and $T_{ij}$ are null matrices except for finitely many $i$ and $j$, we will use the following conventions.
 
\begin{table}[htdp]
\begin{center}
\begin{tabular}{ll}
  $c_{p}(T)$ & $\sup\{u \ge 0; \sum_{n=0}^{\infty} u^{n} T^{n} < \infty \}$: the convergence parameter of $T$,\\
  $\gamma_{\rs{pf}}(T)$ & the Perron-Frobenius eigenvalue of $T$ if $T$ is finite dimensional, \\
  & while it equals $c_{p}(T)^{-1}$ if $T$ is infinite dimensional.\\
  $T_{*}(\theta)$ & $\sum_{i \in \dd{Z}} e^{i \theta} T_{i}$ for $\theta \in \dd{R}$: the matrix MGF of $\{T_{i}\}$,\\
  & where MGF is for moment generating function, \\
  $T_{**}(\vc{\theta})$ & $\sum_{i,j \in \dd{Z}} e^{(i,j) \vc{\theta}} T_{ij}$ for $\vc{\theta} \in \dd{R}^{2}$: the matrix MGF of $\{T_{ij}\}$,\\
   $\gamma^{(\cdot)}(\cdot)$ & $\gamma^{(T_{*})}(\theta) = \gamma_{\rs{pf}}(T_{*}(\theta))$, $\quad \gamma^{(T_{**})}(\vc{\theta}) = \gamma_{\rs{pf}}(T_{**}(\vc{\theta}))$,\\
   $(I-T)^{-1}$ & $\sum_{n=0}^{\infty} T^{n}$ (this is infinite if $c_{p}(T) < 1$),\\
   $\Delta_{\vc{x}}$ & the diagonal matrix whose $i$-th diagonal entry is $x_{i}$,\\
   & where $x_{i}$ is the $i$-th entry of vector $\vc{x}$.
\end{tabular}
\caption{Conventions for matrices and their MGF}
\label{tab:matrix}
\end{center}\vspace{-2ex}
\end{table}
Here, the sizes of those matrices must be the same among those with the same type of indices, but they may be infinite. We also will use those matrices and related notation when the off-diagonal entries of $T, T_{i}, T_{ij}$ are nonnegative.

\section{Nonnegative matrices and QBD block structure}
\setnewcounter
\label{sect:nonnegative}

Let $K$ be a nonnegative square matrix with infinite dimension. Throughout this section, we assume the following regularity condition.
\begin{itemize}
\item [(\sect{nonnegative}a)] $K$ is irreducible, that is, for each entry $(i, j)$ of $K$, there is some $n \ge1$ such that the $(i,j)$ entry of $K^{n}$ is positive.
\end{itemize}

\subsection{Superharmonic vector}
\label{sect:super}

In this subsection, we do not assume any assumption other than (\sect{nonnegative}a), and introduce some basic notions. A positive column vector $\vc{y}$ satisfying
\begin{align}
\label{eqn:super-h 1}
  K \vc{y} \le \vc{y}
\end{align}
is called a superharmonic vector of $K$, where the inequality of vectors is entry-wise. The condition \eqn{super-h 1} is equivalent to that there exists a positive row vector $\vc{x}$ satisfying $\vc{x} K \le \vc{x}$. This $\vc{x}$ is called a sub-invariant vector. Instead of \eqn{super-h 1}, if, for $u > 0$,
\begin{align}
\label{eqn:super-h t}
  u K \vc{y} \le \vc{y},
\end{align}
then $\vc{y}$ is called a $u$-superharmonic vector. We will not consider this vector, but most of our arguments are parallel to those for a superharmonic vector because $\vc{y}$ of \eqn{super-h t} is superharmonic for $uK$.

These conditions can be given in terms of the convergence parameter $c_{p}(K)$ of $K$ (see \tab{matrix} for its definition). As shown in Chapter 5 of the book of Nummelin \cite{Numm1984} (see also Chapter 6 of \cite{Sene1981}), 
\begin{align}
\label{eqn:cp harmonic 1}
  c_{p}(K) = \sup \{ u \ge 0; u K \vc{y} \le \vc{y} \mbox{ for some } \vc{y} > 0 \},
\end{align}
or equivalently $c_{p}(K) = \sup \{ u\ge 0; u \vc{x} K \le \vc{x} \mbox{ for some } \vc{x} > 0 \}$. Applying this fact to $K$, we have the following lemma. For completeness, we give its proof.
\begin{lemma}
\label{lem:super-h 1}
  A nonnegative matrix $K$ satisfying (\sect{nonnegative}a) has a superharmonic vector if and only if $c_{p}(K) \ge 1$.
\end{lemma}

\begin{proof}
  If $K$ has a superharmonic vector, then we obviously have $c_{p}(K) \ge 1$ by \eqn{cp harmonic 1}. Conversely, suppose $c_{p}(K) \ge 1$. Then, by \eqn{cp harmonic 1}, for any positive $u < 1$, we can find a positive vector $\vc{z}(u)$ such that
\begin{align}
\label{eqn:uK}
  u K \vc{z}(u) \le \vc{z}(u).
\end{align}
Denote the $i$-th entry of $\vc{z}(u)$ by $z_{i}(u)$, and define vector $\vc{y}(u)$ whose $i$-th entry $y_{i}(u)$ is given by
\begin{align*}
  y_{i}(u) = \min(1, z_{i}(u)/z_{0}(u)).
\end{align*}
Then, it follows from \eqn{uK} that $y_{0}(u) = 1$, $y_{i}(u) \le 1$ for all $i$, and
\begin{align*}
  u K \vc{y}(u) \le \vc{y}(u).
\end{align*}
Taking the limit infimum of both sides of the above inequality as $u \uparrow 1$, and letting $\vc{y} = \liminf_{u \uparrow 1} \vc{y}(u)$, we have \eqn{super-h 1}. Thus, $K$ indeed has a superharmonic vector $\vc{y}$, which must be positive by the irreducible assumption (\sect{nonnegative}a).
\end{proof}

The importance of Condition \eqn{super-h 1} lies in the fact that $\Delta_{\vc{y}}^{-1} K \Delta_{\vc{y}}$ is substochastic, that is, $K$ can be essentially considered as a substochastic matrix. This enables us to use probabilistic arguments for manipulating $K$ in computations.

\subsection{QBD block structure and its canonical form}
\label{sect:QBD}

We now assume further structure for $K$. Let $m_{0}$ and $m$ be arbitrarily given positive integers. Let $A_{i}$ and $B_{i}$ for $i=-1,0,1$ be nonnegative matrices such that $A_{i}$ for $i=-1,0,1$ are $m \times m$ matrices, $B_{-1}$ is an $m_{0} \times m$ matrix, $B_{0}$ is an $m_{0} \times m_{0}$ matrix and $B_{1}$ is $m \times m_{0}$ matrix. We assume that $K$ has the following form:
\begin{align}
\label{eqn:QBD structure}
  K = \left (\begin{matrix}
B_{0} & B_{1} & 0 & 0 & 0 & \cdots \cr
 B_{-1} & A_{0} & A_{1} & 0 & 0 & \cdots  \cr
 0 & A_{-1} & A_{0} & A_{1} & 0 & \cdots \cr
 0 & 0 & A_{-1} & A_{0} & A_{1} & \cdots \cr
 & \ddots & \ddots & \ddots & \ddots & \cdots \cr
 \end{matrix} \right ).
\end{align}

If $K$ is stochastic, then it is the transition matrix of a discrete-time QBD process. Thus, we refer to $K$ of \eqn{QBD structure} as a nonnegative matrix with QBD block structure.

As we discussed in \sectn{Introduction}, we are primarily interested in tractable conditions for $K$ to have a superharmonic vector. Denote this vector by $\vc{y} \equiv (\vc{y}_{0}, \vc{y}_{1}, \ldots)^{\rs{t}}$. That is, $\vc{y}$ is positive and satisfies the following inequalities.
\begin{align}
\label{eqn:y 0}
 & B_{0}\vc{y}_{0} + B_{1}\vc{y}_{1} \le \vc{y}_{0},\\
\label{eqn:y 1}
 & B_{-1}\vc{y}_{0} + A_{0}\vc{y}_{1} + A_{1}\vc{y}_{2} \le \vc{y}_{1},\\
\label{eqn:y +}
 & A_{-1}\vc{y}_{n-1} + A_{0}\vc{y}_{n} + A_{1}\vc{y}_{n+1} \le \vc{y}_{n}, \qquad n=2, \ldots.
\end{align}

Although the QBD block structure is natural in applications, there are two extra equations \eqn{y 0} and \eqn{y 1} which involve the boundary blocks $B_{i}$. Let us consider how to reduce them to one equation. From \eqn{y 0}, $B_{0}\vc{y}_{0} \le \vc{y}_{0}$ and $B_{0}\vc{y}_{0} \ne \vc{y}_{0}$. Hence, if $c_{p}(B_{0}) = 1$, then $\vc{y}_{0}$ must be the left invariant vector of $B_{0}$ (see Theorem 6.2 of \cite{Sene1981}), but this is impossible because $\vc{y}_{1} > \vc{0}$. Thus, we must have $c_{p}(B_{0}) > 1$, and therefore $(I - B_{0})^{-1}$ is finite (see our convention (\tab{matrix}) for this inverse). Let
\begin{align}
\label{eqn:c 0}
  C_{0} = B_{-1} (I - B_{0})^{-1} B_{1} + A_{0},
\end{align}
then \eqn{y 0} and \eqn{y 1} yield
\begin{align}   
\label{eqn:c 1}
  C_{0} \vc{y}_{1} + A_{1}\vc{y}_{2} \le \vc{y}_{1}.
\end{align}
This suggests that we should define a matrix $\ol{K}$ as
\begin{align}
\label{eqn:ol K}
  \ol{K} = \left (\begin{matrix}
 C_{0} & A_{1} & 0 & 0 & 0 & \cdots \cr
 A_{-1} & A_{0} & A_{1} & 0 & 0 & \cdots  \cr
 0 & A_{-1} & A_{0} & A_{1} & 0 & \cdots \cr
 0 & 0 & A_{-1} & A_{0} & A_{1} & \cdots \cr
 & \ddots & \ddots & \ddots & \ddots & \ddots \cr
 \end{matrix} \right ),
\end{align}
where $C_{0}$ is defined by \eqn{c 0}. Denote the principal matrix of $K$ (equivalently, $\ol{K}$) obtained by removing the first row and column blocks by $K_{+}$. Namely,
\begin{align}
\label{eqn:+ K}
  K_{+} = \left (\begin{matrix}
 A_{0} & A_{1} & 0 & 0 &  \cdots \cr
 A_{-1} & A_{0} & A_{1} & 0 &  \cdots  \cr
 0 & A_{-1} & A_{0} & A_{1} &  \cdots \cr
 & \ddots & \ddots & \ddots &  \ddots \cr
 \end{matrix} \right ).
\end{align}

\begin{lemma}
\label{lem:super-h 2}
  (a) $K$ has a superharmonic vector if and only if $\ol{K}$ has a superharmonic vector. (b) $\max(c_{p}(K), c_{p}(\ol{K})) \le c_{p}(K_{+})$. (c) If $c_{p}(K) \ge 1$, then $c_{p}(K) \le c_{p}(\ol{K}) \le c_{p}(K_{+})$.
\end{lemma}

\begin{remark}
\label{rem:super-h 2}
A similar result for $K$ and $K_{+}$ is obtained in \citet{BeanPollTayl2000}.
\end{remark}

\begin{proof}
  Assume that $K$ has a superharmonic vector $\vc{y} \equiv (\vc{y}_{0}, \vc{y}_{1}, \ldots)^{\rs{t}}$. Then, we have seen that $c_{p}(B_{0}) > 1$, and therefore $(I-B_{0})^{-1} < \infty$. Define $\ol{\vc{y}} \equiv (\ol{\vc{y}}_{0}, \ol{\vc{y}}_{1}, \ldots)^{\rs{t}}$ by $\ol{\vc{y}}_{n} = \vc{y}_{n+1}$ for $n \ge 0$, and define $C_{0}$ by \eqn{c 0}. Then, from \eqn{c 1} , we have
\begin{align*}
  C_{0} \ol{\vc{y}}_{0} + A_{1} \ol{\vc{y}}_{1} \le \ol{\vc{y}}_{0}.
\end{align*}
This and \eqn{y +} verify that $\ol{\vc{y}}$ is superharmonic for $\ol{K}$. On the contrary, assume that $\ol{K}$ is well defined and $\ol{\vc{y}} \equiv (\ol{\vc{y}}_{0}, \ol{\vc{y}}_{1}, \ldots)^{\rs{t}}$ is superharmonic for $\ol{K}$. Obviously, the finiteness of $\ol{K}$ implies that $B_{-1} (I - B_{0})^{-1} B_{1}$ is finite. Suppose that $c_{p}(B_{0}) \le 1$, then some principal submatrix of $(I - B_{0})^{-1}$ has divergent entries in every row and column of this submatrix. Denote a collection of all such principal matrices which are maximal in their size by $\sr{P}_{0}$. Then, all entries $(i,j)$ of submatrices in $\sr{P}_{0}$, we must have, for all $n \ge 0$,
\begin{align*}
  [B_{-1}]_{ki} \left[\sum_{s=0}^{n} B_{0}^{s}\right]_{ij} [B_{1}]_{j\ell} = 0, \qquad \forall k, \ell \in \{1,2,\ldots, m\},
\end{align*}
because of the finiteness of $B_{-1} (I - B_{0})^{-1} B_{1}$. This contradicts the irreducibility (\sect{nonnegative}a) of $K$. Hence, we have $c_{p}(B_{0}) > 1$. Define $\vc{y} \equiv (\vc{y}_{0}, \vc{y}_{1}, \ldots)^{\rs{t}}$ as
\begin{align*}
  \vc{y}_{0} = (I - B_{0})^{-1} B_{1} \ol{\vc{y}}_{0}, \qquad \vc{y}_{n} = \ol{\vc{y}}_{n-1}, \quad n \ge 1,
\end{align*}
where $(I - B_{0})^{-1} < \infty$ because of $c_{p}(B_{0}) > 1$. Then, from \eqn{c 0}, we have
\begin{align*}
  C_{0} \vc{y}_{1} = C_{0} \ol{\vc{y}}_{0} = B_{-1} \vc{y}_{0} + A_{0} \vc{y}_{1},
\end{align*}
and therefore the fact that $C_{0} \ol{\vc{y}}_{0} + A_{1} \ol{\vc{y_{1}}} \le \ol{\vc{y}}_{0}$ implies \eqn{y 1}. Finally, the definition of $\vc{y}_{0}$ implies \eqn{y 0} with equality, while the definition of $\vc{y}_{n}$ for $n \ge 1$ implies \eqn{y +}. Hence, $\vc{y}$ is superharmonic for $K$. This proves (a). (b) is immediate from \eqn{cp harmonic 1} since $(\vc{y}_{1}, \ldots)^{\rs{t}}$ is superharmonic for $K_{+}$ if $(\vc{y}_{0}, \vc{y}_{1}, \ldots)^{\rs{t}}$ is superharmonic for $K$ (or $\ol{K}$). For (c), recall that the canonical form of $uK$ is denoted by $\ol{uK}$ for $u > 0$. If $u \ge 1$, we can see that $u \ol{K} \le \ol{uK}$, and therefore $c_{p}(K) \le c_{p}(\ol{K})$. This and (b) conclude (c).
\end{proof}
  
  By this lemma, we can work on $\ol{K}$ instead of $K$ so as to find a superharmonic vector. It is notable that all block matrices of $\ol{K}$ are $m \times m$ square matrices and it has repeated row and column structure except for the first row and first column blocks. This greatly simplifies computations. So far, we refer to $\ol{K}$ as the canonical form of $K$.
  
  In what follows, we will mainly work on the canonical form $\ol{K}$ of $K$. For simplicity, we will use $\vc{y} \equiv (\vc{y}_{0}, \vc{y}_{1}, \ldots)^{\rs{t}}$ for a superharmonic vector of $\ol{K}$.
  
\subsection{Existence of a superharmonic vector}
\label{sect:existence superh}

Suppose that $\ol{K}$ of \eqn{ol K} has a superharmonic vector $\vc{y} \equiv (\vc{y}_{0}, \vc{y}_{1}, \ldots)^{\rs{t}}$. That is,
\begin{align}
\label{eqn:oy 0}
 & C_{0} \vc{y}_{0} + A_{1} \vc{y}_{1} \le \vc{y}_{0}, \\
\label{eqn:oy +}
 & A_{-1} \vc{y}_{n-1} + A_{0} \vc{y}_{n} + A_{1} \vc{y}_{n+1} \le \vc{y}_{n}, \qquad n \ge 1.
\end{align}
In this section, we consider conditions for the existence of a superharmonic vector.

Letting $C_{1} = A_{1}$, we recall matrix moment generating functions for $\{A_{i}\}$ and $\{C_{i}\}$ (see \tab{matrix}):
\begin{align*}  
  A_{*}(\theta) = e^{-\theta} A_{-1} + A_{0} + e^{\theta} A_{1}, \qquad
  C_{*}(\theta) = C_{0} + e^{\theta} C_{1}, \qquad \theta \in \dd{R}.
\end{align*}
From now on, we always assume a further irreducibility in addition to (\sect{nonnegative}a).
\begin{itemize}
\item [(\sect{nonnegative}b)] $A_{*}(0)$ is irreducible.
\end{itemize}
Since $A_{*}(\theta)$ and $C_{*}(\theta)$ are nonnegative and finite dimensional square matrices, they have Perron-Frobenius eigenvalues $\gamma^{(A_{*})}(\theta) (\equiv \gamma_{\rs{pf}}(A_{*}(\theta)))$ and $\gamma^{(C_{*})}(\theta) (\equiv \gamma_{\rs{pf}}(C_{*}(\theta)))$, respectively, and their right eigenvectors $\vc{h}^{(A_{*})}(\theta)$ and $\vc{h}^{(C_{*})}(\theta)$, respectively. That is,
\begin{align}
\label{eqn:A 0}
 &  A_{*}(\theta) \vc{h}^{(A_{*})}(\theta) = \gamma^{(A_{*})}(\theta) \vc{h}^{(A_{*})}(\theta),\\
\label{eqn:C 0}
 & C_{*}(\theta) \vc{h}^{(C_{*})}(\theta) = \gamma^{(C_{*})}(\theta) \vc{h}^{(C_{*})}(\theta),
\end{align}
where $C_{*}(\theta)$ may not be irreducible, so we take a maximal eigenvalue among those which have positive right invariant vectors. Thus, $\vc{h}^{(A_{*})}(\theta)$ is positive, but $\vc{h}^{(C_{*})}(\theta)$ is nonnegative with possibly zero entries. These eigenvectors are unique up to constant multipliers.

It is well known the $\gamma^{(A_{*})}(\theta)$ and $\gamma^{(C_{*})}(\theta)$ are convex functions of $\theta$ (see, e.g., Lemma 3.7 of \cite{MiyaZwar2012}). Furthermore, their reciprocals are the convergence parameters of $A_{*}(\theta)$ and $C_{*}(\theta)$, respectively. It follows from the convexity of $\gamma^{(A_{*})}(\theta)$ and the fact that some entries of $A_{*}(\theta)$ diverge as $|\theta| \to \infty$ that
\begin{align}
\label{eqn:phi infty}
  \lim_{\theta \to -\infty} \gamma^{(A_{*})}(\theta) = \lim_{\theta \to +\infty} \gamma^{(A_{*})}(\theta) =
  +\infty .
\end{align}
We introduce the following notation.
\begin{align*}
  \Gamma^{(1d)}_{+} = \{\theta \in \dd{R}; \gamma^{(A_{*})}(\theta) \le 1\}, \qquad \Gamma^{(1d)}_{0} = \{\theta \in \dd{R}; \gamma^{(C_{*})}(\theta) \le 1\},
\end{align*}
where $\Gamma^{(1d)}_{0} \not= \emptyset$ implies that $C_{0}$ is finite, that is, $c_{p}(B_{0}) > 1$. By \eqn{phi infty}, $\Gamma^{(1d)}_{+}$ is a bounded interval or the empty set.

In our arguments, we often change the repeated row of blocks of $K$ and $\ol{K}$ so that they are substochastic. For this, we introduce the following notation. For each $\theta \in \dd{R}$ and $\vc{h}^{(A_{*})}(\theta)$ determined by \eqn{A 0}, let
\begin{align*}
  \widehat{A}^{(\theta)}_{\ell} = e^{\theta \ell} \Delta_{\vc{h}^{(A_{*})}(\theta)}^{-1} A_{\ell} \Delta_{\vc{h}^{(A_{*})}(\theta)}, \qquad \ell = 0, \pm 1,
\end{align*}
where we recall that $\Delta_{\vc{a}}$ is the diagonal matrix whose diagonal entry is given by the same dimensional vector $\vc{a}$. Let
\begin{align*}
  \widehat{A}^{(\theta)} = \widehat{A}^{(\theta)}_{-1} + \widehat{A}^{(\theta)}_{0} + \widehat{A}^{(\theta)}_{1}.
\end{align*}
Note that $\widehat{A}^{(\theta)} = \Delta_{\vc{h}^{(A_{*})}(\theta)}^{-1} A_{*}(\theta) \Delta_{\vc{h}^{(A_{*})}(\theta)}$, and therefore $\widehat{A}^{(\theta)} \vc{1} = \gamma^{(A_{*})}(\theta) \vc{1}$.

The following lemma is the first step in characterizing $c_{p}(K) \ge 1$.
\begin{lemma}
\label{lem:Gamma+}
  (a) $c_{p}(K_{+}) \ge 1$ if and only if $\Gamma^{(1d)}_{+} \ne \emptyset$, and therefore $c_{p}(K) \ge 1$ implies $\Gamma^{(1d)}_{+} \ne \emptyset$. (b) $c_{p}(K_{+}) = (\min \{\gamma^{A_{*}}(\theta); \theta \in \dd{R}\})^{-1}$.
\end{lemma}

  This lemma may be considered to be a straightforward extension of Theorem 2.1 of \citet{Kiji1993} from a substochastic to a nonnegative matrix. So, it may be proved similarly, but we give a different proof in \app{Gamma+}. There are two reasons for this. First, it makes this paper selfcontained. Second, we wish to demonstrate that it is hard to remove the finiteness of $m$ on block matrices.
  
  We now present necessary and sufficient conditions for $\ol{K}$, equivalently $K$, to have a superharmonic vector.

\begin{theorem} 
\label{thr:characterization 1}
(a) If $\Gamma^{(1d)}_{+} \ne \emptyset$, then $N \equiv (I-K_{+})^{-1}$ is finite, and therefore $G_{-} \equiv N_{11} A_{-1}$ is well defined and finite, where $N_{11}$ is the $(1,1)$-entry of $N$. (b) $c_{p}(\ol{K}) \ge 1$ holds if and only if $\Gamma^{(1d)}_{+} \ne \emptyset$ and
\begin{align}
\label{eqn:superH NS 1}
  \gamma_{\rs{pf}}(C_{0} + A_{1} G_{-}) \le 1.
\end{align}
If the equality holds in \eqn{superH NS 1}, then $c_{p}(\ol{K}) = 1$.
\end{theorem}

\begin{remark} {\rm
\label{rem:characterization 1}
  The sufficiency in (b) is essentially obtained in Theorem 6 of \cite{LiZhao2003}, which is used in Sections 3.4 and 3.5 of \cite{Ozaw2013}, where the eigenvalue $\gamma_{\rs{pf}}(C_{0} + A_{1} G_{-})$ corresponds to $u_{0}(1)$ in \cite{LiZhao2003}.  We do not need the function $u_{0}(\beta)$ there because we work on a nonnegative matrix while substochasticity is assumed in \cite{LiZhao2003}.
}\end{remark}

\begin{proof}
 (a)  Assume that $\theta \in \Gamma^{(1d)}_{+} \ne \emptyset$, then we can find a $\theta_{1}$ such that $\theta_{1} = \arg_{\theta} \min \{\theta \in \dd{R}; \gamma^{A_{*}}(\theta) = 1\}$. For this $\theta_{1}$, let $\vc{y}(\theta_{1}) = (\vc{h}^{(A_{*})}(\theta_{1}), e^{\theta_{1}} \vc{h}^{(A_{*})}(\theta_{1}), \ldots)^{\rs{t}}$, and let
\begin{align*}
  \breve{K}_{+}^{(\theta_{1})} = \Delta_{\vc{y}(\theta_{1})}^{-1} K_{+} \Delta_{\vc{y}(\theta_{1})}.
\end{align*}
It is easy to see that $\breve{K}_{+}^{(\theta_{1})}$ is strictly substochastic because the first row of blocks is defective. Hence, $(1-\breve{K}_{+}^{(\theta_{1})})^{-1}$ must be finite, and therefore $(I-K_{+})^{-1}$ is also finite. This proves (a).\\
 (b) Assume that $c_{p}(\ol{K}) \ge 1$, then $\Gamma^{(1d)}_{+} \ne \emptyset$ by \lem{Gamma+}. Hence, $(I-K_{+})^{-1}$ is finite by (a). Because of $c_{p}(\ol{K}) \ge 1$, $\ol{K}$ has a superharmonic vector. We denote this vector by $\vc{y} = \{\vc{y}_{n}; n=0,1,\ldots\}^{\rs{t}}$. Let $\vc{z} = \{\vc{y}_{n}; n=1,2,\ldots\}^{\rs{t}}$, then we have
\begin{align}
\label{eqn:superH y 0}
 & C_{0} \vc{y}_{0} + A_{1} \vc{y}_{1} \le \vc{y}_{0},\\
\label{eqn:superH y 1}
 & (A_{-1} \vc{y}_{0}, \vc{0}, \vc{0}, \ldots)^{\rs{t}} + K_{+} \vc{z} \le \vc{z}.
\end{align}
It follows from the second equation that $ [(I - K_{+})^{-1}]_{11} A_{-1} \vc{y}_{0}  \le \vc{y}_{1}$. Hence, substituting this into \eqn{superH y 0}, we have
\begin{align}
\label{eqn:superH y 00}
  (C_{0} + A_{1} G_{-}) \vc{y}_{0} \le \vc{y}_{0},
\end{align}
which is equivalent to \eqn{superH NS 1}. Conversely, assume \eqn{superH NS 1} and $\Gamma^{(1d)}_{+} \ne \emptyset$, then we have \eqn{superH y 00} for some $\vc{y}_{0} > \vc{0}$. Define $\vc{z}$ as
\begin{align*}
  \vc{z} = (I - K_{+})^{-1} (A_{-1} \vc{y}_{0}, \vc{0}, \vc{0}, \ldots)^{\rs{t}},
\end{align*}
then we get \eqn{superH y 1} with equality, and \eqn{superH y 00} yield \eqn{superH y 0}. Thus, we get the superharmonic vector $(\vc{y}_{0}, \vc{z})^{\rs{t}}$ for $\ol{K}$. This completes the proof.
\end{proof}

Using the notation in the above proof, let $\breve{N}^{(\theta_{1})} = (1-\breve{K}_{+}^{(\theta_{1})})^{-1}$, and let
\begin{align*}
  \widehat{G}^{(\theta_{1})}_{-} = \breve{N}^{(\theta_{1})}_{11} \widehat{A}^{(\theta_{1})}_{-1},
\end{align*}
then $\widehat{G}^{(\theta_{1})}_{-}$ must be stochastic because it is a transition matrix for the background state when the random walk component hits one level down. Furthermore,
\begin{align*}
  \widehat{G}^{(\theta_{1})}_{-} = \Delta_{\vc{h}^{(A_{*})}(\theta_{1})}^{-1} (e^{-\theta_{1}} G_{-}) \Delta_{\vc{h}^{(A_{*})}(\theta_{1})}.
\end{align*}
Hence, for $m=1$, $e^{-\theta_{1}} G_{-} = 1$, and therefore \eqn{superH NS 1} is identical with
\begin{align}
\label{eqn:superH NS 2}
  \gamma_{\rs{pf}}\left(C_{0} + e^{\theta_{1}} A_{1} \right) \le 1,
\end{align}
which agrees with $\gamma^{(C_{*})}(\theta_{1}) \le 1$. Hence, we have the following result.

\begin{corollary}
\label{cor:super-h 1}
For $m=1$, $c_{p}(K) \ge 1$ if and only if $\Gamma^{(1d)}_{+} \cap \Gamma^{(1d)}_{0} \ne \emptyset$.
\end{corollary}

  This corollary is essentially the same as Theorem 3.1 of \cite{Miya2009}, so nothing is new technically. Here, we have an alternative proof. However, it is notable that $K$ may have boundary blocks whose size is $m_{0} \ge 1$ while $m=1$. 

For $m \ge 2$, \thr{characterization 1} is not so useful in application because it is hard to evaluate $G_{-}$ and therefore it is hard to verify \eqn{superH NS 1}. \citet{Ozaw2013} proposes to compute the corresponding characteristics numerically. However, in its application for the 2d-QBD process, $G_{-}$ is parametrized, and we need to compute it for some range of parameters. Thus, even numerical computations are intractable.

One may wonder how to replace \eqn{superH NS 1} by a tractable condition. In the view of the case of $m=1$, one possible condition is that $\gamma^{(C_{*})}(\theta) \le 1$ for some $\theta \in \Gamma^{(1d)}_{+}$, which is equivalent to \eqn{superH NS 2} for general $m$. However, $\gamma_{\rs{pf}}\left(C_{0} + e^{\theta_{1}} A_{1}\right)$, which equals $\gamma^{(C_{*})}(\theta)$, is generally not identical with $\gamma_{\rs{pf}}\left(C_{0} + A_{1} G_{-} \right)$ (see \app{counter example}). So far, we will not pursue the use of \thr{characterization 1}.

\subsection{A tractable condition for application}
\label{sect:tractable condition}

We have considered conditions for $c_{p}(K) \ge 1$, equivalently, $c_{p}(\ol{K}) \ge 1$. For this problem, we here consider a specific superharmonic vector for $\ol{K}$. For each $\theta \in \dd{R}$ and $\vc{h} \ge \vc{0}$, define $\vc{y}(\theta) \equiv (\vc{y}_{0}(\theta), \vc{y}_{1}(\theta), \ldots)^{\rs{t}}$ by
\begin{align}
\label{eqn:y n}
  \vc{y}_{n}(\theta) = e^{\theta n}  \vc{h}, \qquad n \ge 0.
\end{align}
Then, $\ol{K} \vc{y}(\theta) \le \vc{y}(\theta)$ holds if and only if
\begin{align}
\label{eqn:A 1}
 & A_{*}(\theta) \vc{h} \le \vc{h},\\
\label{eqn:C 1}
 & C_{*}(\theta) \vc{h} \le \vc{h}.
\end{align}
These conditions hold for $\vc{y}(\theta)$ of \eqn{y n}, so we only know that they are sufficient but may not be necessary. To fill this gap, we go back to $K$ and consider its superharmonic vector, using \eqn{y n} for off-boundary blocks. This suggests that we should replace \eqn{C 1} by the following assumption.
\begin{assumption} {\rm
\label{ass:key assumption 1}
  For each $\theta \in \Gamma^{(1d)}_{+}$, there is an $m_{0}$-dimensional positive vector $\vc{h}^{(0)}(\theta)$ and real numbers $c_{0}(\theta), c_{1}(\theta)$ such that either one of $c_{0}(\theta)$ or $c_{1}(\theta)$ equals one, and
\begin{align}
\label{eqn:key assumption 1a}
 & B_{0} \vc{h}^{(0)}(\theta) + e^{\theta} B_{1} \vc{h}^{(A_{*})}(\theta) = c_{0}(\theta) \vc{h}^{(0)}(\theta),\\
\label{eqn:key assumption 1b}
 & e^{-\theta} B_{-1} \vc{h}^{(0)} (\theta)+ A_{0} \vc{h}^{(A_{*})}(\theta) + e^{\theta} A_{1} \vc{h}^{(A_{*})}(\theta) = c_{1}(\theta) \vc{h}^{(A_{*})}(\theta). \qquad
\end{align}
}\end{assumption}

\begin{remark} {\rm
\label{rem:key assumption 1}
If $c_{0}(\theta) \le 1$ and $c_{1}(\theta) \le 1$, then \eqn{C 1} is equivalent to \eqn{key assumption 1a} and \eqn{key assumption 1b}. However, it is unclear whether or not $c_{0}(\theta) = 1$ or $c_{1}(\theta) = 1$ implies \eqn{C 1}. In particular, $c_{1}(\theta) = 1$ is the case that we need in our application to the generalized Jackson network. This will be affirmatively answered in \thr{super-h 2}.
}\end{remark}

Let
\begin{align*}
 & \Gamma^{(1d)}_{0+} = \{\theta \in \dd{R}; \exists \vc{h} > \vc{0}, A_{*}(\theta) \vc{h} \le \vc{h}, C_{*}(\theta) \vc{h} \le \vc{h} \} ,\\
 & \Gamma^{(1e)}_{0+} = \{\theta \in \dd{R}; \exists \vc{h} > \vc{0}, A_{*}(\theta) \vc{h} = \vc{h}, C_{*}(\theta) \vc{h} \le \vc{h} \} .
\end{align*}
If $\Gamma^{(1d)}_{0+} \not= \emptyset$, $C_{0}$ is finite, and therefore $c_{p}(B_{0}) > 1$. $\Gamma^{(1e)}_{0+}$ is at most a two point set. Note that $\Gamma^{(1d)}_{0+} \subset \Gamma^{(1d)}_{0} \cap \Gamma^{(1d)}_{+}$, but $\Gamma^{(1d)}_{0+} = \Gamma^{(1d)}_{0} \cap \Gamma^{(1d)}_{+}$ may not be true except for $m=1$. We further note the following facts.

\begin{lemma}
\label{lem:Gamma convex 1}
  If $\Gamma^{(1d)}_{0+} \not= \emptyset$, then $\Gamma^{(1d)}_{0+}$ is a bounded convex subset of $\dd{R}$, and it can be written as the closed interval $[\theta^{(A,C)}_{\min}, \theta^{(A,C)}_{\max}]$, respectively, where
\begin{align}
\label{eqn:theta min-max}
  \theta^{(A,C)}_{\min} = \inf \{ \theta \in \Gamma^{(1d)}_{0+} \}, \qquad \theta^{(A,C)}_{\max} = \sup \{ \theta \in \Gamma^{(1d)}_{0+} \}.
\end{align}
\end{lemma}
    
  We prove this lemma in \app{Gamma convex 1} because it is just technical. Based on these observations, we claim the following fact.
\begin{theorem}
\label{thr:super-h 2}
For a nonnegative matrix $K$ with QBD block structure, assume Conditions (\sect{nonnegative}a) and (\sect{nonnegative}b). (a) If $\Gamma^{(1d)}_{0+} \ne \emptyset$, then $c_{p}(K) \ge 1$. (b) Under \ass{key assumption 1}, $c_{p}(K) \ge 1$, if and only if $\Gamma^{(1d)}_{0+} \ne \emptyset$, which can be replaced by $\Gamma^{(1e)}_{0+} \ne \emptyset$.
\end{theorem}

\begin{proof}
(a) We already know that $\vc{y}(\theta)$ of \eqn{y n} is a superharmonic vector of $\ol{K}$ for $\theta \in \Gamma^{(1d)}_{0+}$. Thus, \lem{super-h 2} implies (a).

(b) The sufficiency of $\Gamma^{(1d)}_{0+} \ne \emptyset$ is already proved in (a). To prove its necessity, we first note that $\Gamma^{(1e)}_{+}$ is not empty by \lem{Gamma+}. Hence, there is a $\theta_{1}$ such that $\theta_{1} = \min\{\theta \in \dd{R}; \gamma^{(A_{*})}(\theta) = 1\}$. For this $\theta_{1}$, we show that \eqn{C 1} holds for $\vc{h} = \vc{h}^{(A_{*})}(\theta_{1})$. To facilitate \ass{key assumption 1}, we work on $K$ rather than $\ol{K}$. Assume that a superharmonic $\vc{y} \equiv (\vc{y}_{0}, \vc{y}_{1}, \ldots)$ exists for $K$. We define the transition probability matrix $\breve{P}^{(\theta_{1})} \equiv \{\breve{P}^{(\theta_{1})}_{k\ell}; k, \ell \ge 0\}$ by
\begin{align*}
 & \breve{P}^{(\theta_{1})}_{00} = \Delta_{c_{0}(\theta_{1}) \vc{h}^{(0)}(\theta_{1})}^{-1} B_{0} \Delta_{\vc{h}^{(0)}(\theta_{1})}, \qquad \breve{P}^{(\theta_{1})}_{01} = e^{\theta_{1}} \Delta_{c_{0}(\theta_{1}) \vc{h}^{(0)}(\theta_{1})}^{-1} B_{1} \Delta_{\vc{h}^{(A_{*})}(\theta_{1})} ,\\
 & \breve{P}^{(\theta_{1})}_{10} = e^{-\theta_{1}} \Delta_{c_{1}(\theta_{1}) \vc{h}^{(A_{*})}(\theta_{1})}^{-1} B_{-1} \Delta_{\vc{h}^{(0)}(\theta_{1})}, \qquad \breve{P}^{(\theta_{1})}_{1 \ell} = c_{1}(\theta_{1})^{-1} \widehat{A}_{\ell-1}^{(\theta_{1})}, \quad \ell = 1,2 ,\\ 
 & \breve{P}^{(\theta_{1})}_{k \ell} = \widehat{A}_{\ell-k}^{(\theta_{1})}, \quad k \ge 2, |\ell - k| \le 1, 
\end{align*}
where $\breve{P}^{(\theta_{1})}_{k \ell}$ is the null matrix for $(k,\ell)$ undefined. It is easy to see that $\breve{P}^{(\theta_{1})}$ is a proper transition matrix with QBD structure by \eqn{key assumption 1a}, \eqn{key assumption 1b} and $\gamma^{(A_{*})}(\theta_{1}) = 1$. Furthermore, as shown in \app{Gamma+}, this random walk has the mean drift \eqn{drift 1} with $\theta_{1}$ instead of $\theta_{0}$. Since the definition of $\theta_{1}$ implies that $(\gamma^{\rt{A}})'(\theta_{1}) = 0$, this Markov chain is null recurrent.

We next define $\breve{\vc{y}}^{(\theta_{1})}$ as
\begin{align*}
  \breve{\vc{y}}^{(\theta_{1})}_{0} = \Delta_{\vc{h}^{(0)}(\theta_{1})}^{-1} \vc{y}_{0}, \qquad \breve{\vc{y}}^{(\theta_{1})}_{n} = e^{-\theta n} \Delta_{\vc{h}^{(A_{*})}(\theta_{1})}^{-1} \vc{y}_{n}, \qquad n \ge 1,
\end{align*}
then the 0-th row block of $ \breve{P}^{(\theta_{1})} \breve{\vc{y}}^{(\theta_{1})}$ is
\begin{align*}
  \breve{P}^{(\theta_{1})}_{00} \breve{\vc{y}}^{(\theta_{1})}_{0} + \breve{P}^{(\theta_{1})}_{01} \breve{\vc{y}}^{(\theta_{1})}_{1}  = c_{0}(\theta_{1})^{-1} \Delta_{\vc{h}^{(0)}(\theta_{1})}^{-1} \vc{y}_{0} = c_{0}(\theta_{1})^{-1} \breve{\vc{y}}^{(\theta_{1})}_{0},
\end{align*}
and, similarly,  the 1-st row block is
\begin{align*}
    \breve{P}^{(\theta_{1})}_{10} \breve{\vc{y}}^{(\theta_{1})}_{0} + \breve{P}^{(\theta_{1})}_{11} \breve{\vc{y}}^{(\theta_{1})}_{1} + \breve{P}^{(\theta_{1})}_{12} \breve{\vc{y}}^{(\theta_{1})}_{2} = c_{1}(\theta_{1})^{-1} \breve{\vc{y}}^{(\theta_{1})}_{1}.
\end{align*}
Hence, $K \vc{y} \le \vc{y}$ is equivalent to
\begin{align}
\label{eqn:breve P}
  \left(\begin{array}{ccccc} 
  c_{0}(\theta_{1}) I_{0} & 0 & 0 & 0 & \hdots \\
  0 & c_{1}(\theta_{1}) I & 0 & 0 & \hdots\\
  0 & 0 & I & 0 & \hdots\\
  0 & 0 & 0 & I & \ddots \\
  \vdots & \ddots & \ddots & \ddots & \ddots
  \end{array}\right) \breve{P}^{(\theta_{1})} \breve{\vc{y}}^{(\theta_{1})} \le \breve{\vc{y}}^{(\theta_{1})},
\end{align}
where $I_{0}$ is the $m_{0}$ dimensional identity matrix. We now prove that $c_{0}(\theta_{1}) \le 1$ and $c_{1}(\theta_{1}) \le 1$ using the assumption that either $c_{0}(\theta_{1}) = 1$ or $c_{1}(\theta_{1}) = 1$. First, we assume that $c_{1}(\theta_{1}) = 1$, and rewrite \eqn{breve P} as
\begin{align*}
   \left(\begin{array}{ccccc} 
  c_{0}(\theta_{1})  \breve{P}^{(\theta_{1})}_{00} & c_{0}(\theta_{1})  \breve{P}^{(\theta_{1})}_{01} & 0 & \hdots \\
  \breve{P}^{(\theta_{1})}_{10} & 0 & 0 & \hdots\\
  0 & 0 & 0 & \ddots \\
  \vdots & \ddots & \ddots  & \ddots
  \end{array}\right) 
  \breve{\vc{y}}^{(\theta_{1})} \le 
  \left(\begin{array}{ccccc} 
  I_{0} & 0 & 0 &\hdots \\
  0 \\
  0 &  & I_{+} - \breve{P}^{(\theta_{1})}_{+} \\
  \vdots 
  \end{array}\right) \breve{\vc{y}}^{(\theta_{1})},
\end{align*}
where $I_{+}$ and $\breve{P}^{(\theta_{1})}_{+}$ are the matrices obtained from $I$ and $\breve{P}^{(\theta_{1})}$, respectively, by deleting their first row and column blocks. Since $\breve{P}^{(\theta_{1})}_{+}$ is strictly substochastic, $I_{+} - \breve{P}^{(\theta_{1})}_{+}$ is invertible. We denote its inversion by $U$, then 
\begin{align*}
  \left(\begin{array}{ccccc} 
  I_{0} & 0 & 0 &\hdots \\
  0 \\
  0 &  & U \\
  \vdots 
  \end{array}\right)
   \left(\begin{array}{ccccc} 
  c_{0}(\theta_{1})  \breve{P}^{(\theta_{1})}_{00} & c_{0}(\theta_{1})  \breve{P}^{(\theta_{1})}_{01} & 0 & \hdots \\
  \breve{P}^{(\theta_{1})}_{10} & 0 & 0 & \hdots\\
  0 & 0 & 0 & \ddots \\
  \vdots & \ddots & \ddots  & \ddots
  \end{array}\right) 
  \breve{\vc{y}}^{(\theta_{1})} \le \breve{\vc{y}}^{(\theta_{1})},
\end{align*}
which yields that
\begin{align*}
  c_{0}(\theta_{1}) \left( \breve{P}^{(\theta_{1})}_{00} + \breve{P}^{(\theta_{1})}_{01} U_{11} \breve{P}^{(\theta_{1})}_{10} \right) \breve{\vc{y}}^{(\theta_{1})}_{0} \le \breve{\vc{y}}^{(\theta_{1})}_{0},
\end{align*}
where $U_{11}$ is $(1,1)$ block of $U$. Since $\left( \breve{P}^{(\theta_{1})}_{00} + \breve{P}^{(\theta_{1})}_{01} U_{11} \breve{P}^{(\theta_{1})}_{10} \right)$ is a stochastic matrix by the null recurrence of $\breve{P}^{(\theta_{1})}$, we must have that $c_{0}(\theta_{1}) \le 1$. The case for $c_{0}(\theta_{1}) = 1$ is similarly proved. Thus, the proof is completed in the view of \rem{key assumption 1}. 
\end{proof}

\subsection{The convergence parameter and $u$-invariant measure}
\label{sect:existence t-invriant}

We now turn to consider the invariant measure of $K$, which will be used in our application. \citet{LiZhao2003} have shown the existence of such invariant measures for $uK$ for $u > 0$ when $K$ is substochastic. We will show that their results are easily adapted for a nonnegative matrix. For this, we first classify a nonnegative irreducible matrix $T$ to be transient, null recurrent or positive recurrent. $T$ is said to be $u$-transient if
\begin{align*}
  \sum_{n=0}^{\infty} u^{n} T^{n} < \infty,
\end{align*}
while it is said to be $u$-recurrent if this sum diverges. For $u$-recurrent $T$, there always exists a $u$-invariant measure, and $T$ is said to be $u$-positive if the $u$-invariant measure has a finite total sum. Otherwise, it is said to be $u$-null. The book of \citet{Sene1981} is a standard reference for these classifications.

Suppose that $c_{p}(K) \ge 1$. We modify $K$ to be substochastic. For this, recall that $c_{p}(K) \ge 1$ is equivalent to the existence of a superharmonic vector of $K$, and that $\Delta_{\vc{a}}$ is the diagonal matrix whose diagonal entry is given by vector $\vc{a}$. Define $\widehat{K}$ for a superharmonic vector $\vc{y}$ of $K$ by
\begin{align*}
  \widehat{K} = \Delta_{\vc{y}}^{-1} K \Delta_{\vc{y}}.
\end{align*}
Then, $\widehat{K} \vc{1} \le \vc{1}$, that is, $\widehat{K}$ is substochastic. It is also easy to see that, for $0 < u \le c_{p}(K)$, $\widehat{\vc{x}}$ is a $u$ invariant measure of $\widehat{K}$ if and only if $\widehat{\vc{x}} \Delta_{\vc{y}}^{-1}$ is a $u$ invariant measure of $K$. Furthermore, the classifications for $K$ are equivalent to those for $\widehat{K}$. Thus, the results of \cite{LiZhao2003} can be stated in the following form.

\begin{lemma}[Theorem A of \citet{LiZhao2003}]
\label{lem:t-invariant m 1}
  For a nonnegative irreducible matrix $K$ with QBD block structure, let $t = c_{p}(K)$, $t_{+} = c_{p}(K_{+})$ and assume that $t \ge 1$. Then, $K$ is classified into either one of the following cases: (a) $t$-positive if $t < t_{+}$, (b) $t$-null or $t$-transient if $t = t_{+}$.
\end{lemma}
\begin{remark} {\rm
\label{rem:t-invariant m 1}
The $t$ and $t_{+}$ correspond to $\alpha$ and $\ol{\alpha}$ of \cite{LiZhao2003}, respectively. In Theorem A of \cite{LiZhao2003}, the case (b) is further classified to $t$-null and $t$-transient cases, but it requires the Perron-Frobenius eigenvalue of $t (C_{0} + R(t) A_{-1})$ to be less than 1 for $t$-transient and to equal 1 for $t$-null, where $R(t)$ is the minimal nonnegative solution $X$ of the matrix equation:
\begin{align*}
  X = t (X^{2} A_{-1} + X A_{0} + A_{1}).
\end{align*}
In general, this eigenvalue is hard to get in closed form, so we will not use this finer classification. Similar but slightly different results are obtained in Theorem 16 of \cite{BeanPollTayl2000}.
}\end{remark}

\begin{lemma}[Theorems B and C of \citet{LiZhao2003}]
\label{lem:t-invariant m 2}
  For $K$ satisfying the assumptions of \lem{t-invariant m 1}, there exist $u$-invariant measures for $0 < u \le t \equiv c_{p}(K)$. The form of these invariant measures varies according to three different types (a1) $u=t$ for $t$-recurrent, (a2) $u = t$ for $t$-transient, and (b) $u<t$.
\end{lemma}
\begin{remark} {\rm
\label{rem:t-invariant m 2}
By \lem{t-invariant m 1}, $K$ is $t$-null for (a1) if and only if $t = t_{+}$.
}\end{remark}

\section{Application to a 2d-QBD process}
\setnewcounter
\label{sect:application}

In this section, we show how \thr{super-h 2} can be applied to a tail asymptotic problem. We here consider a 2d-QBD process $\{Z_{n}\} \equiv \{(L_{1n}, L_{2n}, J_{n})\}$ introduced by \citet{Ozaw2013}, where $\vc{L}_{n} \equiv (L_{1n}, L_{2n})$ is a random walk component taking values in $\dd{Z}_{+}^{2}$ and $\{J_{n}\}$ is a background process with finitely many states. It is assumed that $\{Z_{n}\}$ is a discrete time Markov chain. The tail decay rates of the stationary distribution of the 2d-QBD process have been studied in \cite{Ozaw2013}, but there remain some crucial problems unsolved as we argued in \sectn{Introduction} and will detail in the next subsection. Furthermore, there is some ambiguity in the definition of \citet{Ozaw2013}'s 2d-QBD process. Thus, we first reconsider this definition, and show that those problems on the tail asymptotics can be well studied using \ass{key assumption 1} and \thr{super-h 2}.

\subsection{Two dimensional QBD processes}
\label{sect:two}

We will largely change the notation of \cite{Ozaw2013} to make clear assumptions. We partition the state space $\sr{S}$ of $Z_{n}$ so as to apply \lem{super-h 2}. Divide the lattice quarter plane $\dd{Z}_{+}^{2}$ into four regions.
\begin{align*}
 & \sr{U}_{0} \equiv \{(0,0)\}, \quad \sr{U}_{1} \equiv \{(\ell,0) \in \dd{Z}_{+}^{2}; \ell \ge 1\}, \quad \sr{U}_{2} \equiv \{(0,\ell) \in \dd{Z}_{+}^{2}; \ell \ge 1\}, \\
 & \sr{U}_{+} \equiv \{(\ell,\ell') \in \dd{Z}_{+}^{2}; \ell, \ell' \ge 1\},
\end{align*}
where $\sr{U}_{i}$ for $i=0,1,2$ and $\sr{U}_{+}$ are said to be a boundary face and interior, respectively. Then, the state space $\sr{S}$ for $Z_{n}$ is given by
\begin{align*}
  \sr{S} = (\sr{U}_{0} \times \sr{V}_{0}) \cup (\sr{U}_{1} \times \sr{V}_{1}) \cup (\sr{U}_{2} \times \sr{V}_{2}) \cup (\sr{U}_{+} \times \sr{V}_{+}),
\end{align*}
where $\sr{V}_{i}$ are finite sets of numbers such that their cardinality $|\sr{V}_{i}|$ is given by $m_{0} = |\sr{V}_{0}|$, $m_{1} = |\sr{V}_{1}|$, $m_{2} = |\sr{V}_{2}|$, $m = |\sr{V}_{+}|$.

To define the transition probabilities of $Z_{n}$, we further partition the state space as
\begin{align*}
 & \sr{U}_{\ell m} = \{(\ell,m)\}, \qquad \ell, m \in \dd{H}_{+} \equiv \{0,1\},\\
 & \sr{U}_{+0} = \{(n,0) \in \dd{Z}_{+}^{2}; n \ge 2\}, \qquad \sr{U}_{0+} = \{(0,n) \in \dd{Z}_{+}^{2}; n \ge 2\},\\
 & \sr{U}_{+1} = \{(n,1) \in \dd{Z}_{+}^{2}; n \ge 2\}, \qquad \sr{U}_{1+} = \{(1,n) \in \dd{Z}_{+}^{2}; n \ge 2\},\\
 & \sr{U}_{++} = \{(\ell,m) \in \dd{Z}_{+}^{2}; \ell,m \ge 2\}.
\end{align*}
On each of those sets, the transition probabilities of $Z_{n}$ are assumed to be homogeneous. Namely, for $s, s' \in \sr{A} \equiv \{0,1,+\}$, their matrices for background state transitions can be denoted by $A^{(ss')}_{ij}$ for the transition from $((\ell,m),k) \in \sr{U}_{ss'}$ to $((\ell+i,m+j),k') \in \sr{S}$. Furthermore, we assume that
\begin{align}
\label{eqn:}
 & A^{(10)}_{ij} = A^{(+0)}_{ij}, \quad A^{(01)}_{ij} = A^{(0+)}_{ij}, \quad A^{(11)}_{ij} = A^{(++)}_{ij}, \quad i,j \in \dd{H}_{+},\\
 & A^{(+1)}_{ij} = A^{(++)}_{ij}, \quad i \in \dd{H}, j \in \dd{H}_{+}, \quad A^{(1+)}_{ij} = A^{(++)}_{ij}, \quad i \in \dd{H}_{+}, j \in \dd{H}.
\end{align}
Throughout the paper, we denote $A^{(++)}_{ij}$ by $A_{ij}$. This greatly simplifies the notation. See \fig{RRW1} for those partitions of the quarter plane $\dd{Z}_{+}^{2}$ and the transition probability matrices.
 
\begin{figure}[h]
 	\centering
	\includegraphics[height=7cm]{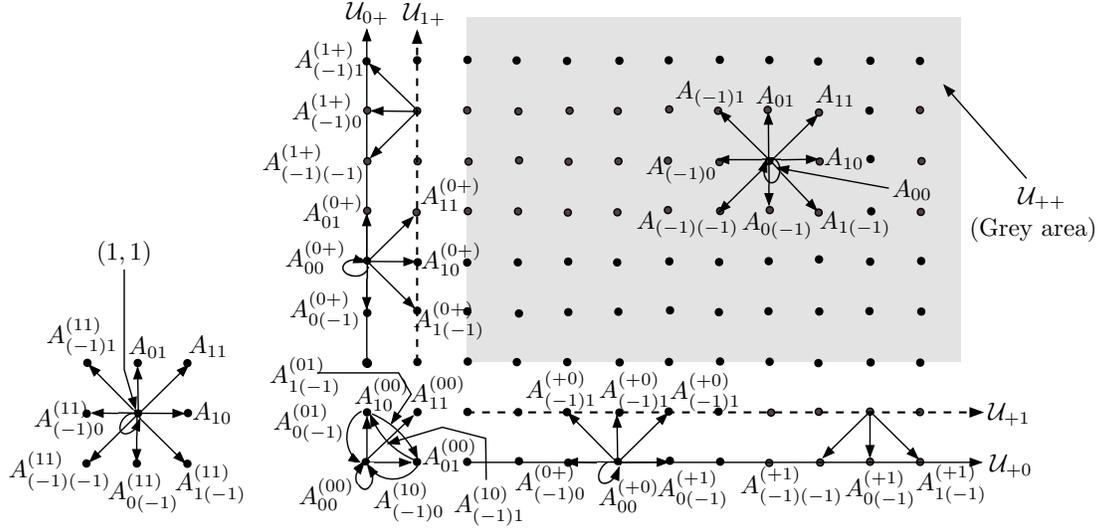}
	\caption{Regions $\sr{U}_{ss'}$ and transition probability matrices $A_{ij}$ and $A^{(ss')}_{ij}$}
	\label{fig:RRW1}
\end{figure}

Those assumptions on the transition probabilities are essentially the same as those introduced by \citet{Ozaw2013}, while there is some minor flexibility in our assumption that $A^{(11)}_{0(-1)}$ ($ A^{(11)}_{(-1)0}$) may be different from $A^{(+1)}_{0(-1)}$ ($A^{(1+)}_{(-1)0}$, respectively), which are identical in \cite{Ozaw2013}. Another difference is in that we have nine families of transition matrices while \citet{Ozaw2013} expresses them by four families, $A^{(s)}_{ij}$ for $s=0,1,2$ and $A_{ij}$. 

By the homogeneity and independence assumptions, we can define $Z_{n+1}$ in terms of $Z_{n}$ and independent increments as
\begin{align}
\label{eqn:L n}
  (\vc{L}_{n+1}, J_{n+1}) = \Big(\vc{L}_{n} + \sum_{s, s' \in \sr{A} } \vc{X}^{(ss')}_{n}(J_{n}) 1(\vc{L}_{n} \in \sr{U}_{ss'}), J_{n+1} \Big),
\end{align}
where $\vc{X}^{(ss')}_{n}(k)$ is the increment at time $n$ when the random walk component on $\sr{U}_{ss'}$ and the background state is $k$. By the modeling assumption, $\vc{X}^{(ss')}_{n}(k)$ is independent of $Z_{\ell}$ for $\ell \le n-1$ and $\vc{L}_{n}$ for given $s, s'$ and $k$.
 
The 2d-QBD process is a natural model for a two node queueing network under various situations including a Markovian arrival process and phase-type service time distributions. Its stationary distribution is a key characteristic for performance evaluation, but hard to get. This is even the case for a two dimensional reflecting random walk, which does not have background states (e.g., see \cite{Miya2011}). Thus, recent interest has been directed to the tail asymptotics of the stationary distribution.

\subsection{Markov additive kernel and stability}
\label{sect:Markov}

Recall that the 2d-QBD process is denoted by $\{(L_{1n}, L_{2n}, J_{n}); n=0,1,\ldots\}$. To define the 1-dimensional QBD process, for $i =1,2$, let $L^{(i)}_{n} = L_{in}$ and $\vc{J}^{(i)}_{n} = (L_{(3-i)n}, J_{n})$, then they represent level and background state at time $n$, respectively. Thus, $\{(L^{(i)}_{n}, \vc{J}^{(i)}_{n}); n \ge 0\}$ is a one-dimensional QBD process for $i=1,2$. Denote its transition matrix by $P^{(i)}$. For example, $P^{(1)}$ is given by
\begin{align*}
  P^{(1)} = \left (\begin{matrix}
N^{(1)}_{0} & N^{(1)}_{1} & 0 & \ldots & \ldots & \ldots & \cr
N^{(1)}_{-1} & Q^{(1)}_{0} & Q^{(1)}_{1} & 0 & 0 & \ldots  \cr
 0 &  Q^{(1)}_{-1} & Q^{(1)}_{0} & Q^{(1)}_{1} & 0 & \ddots \cr
0 & 0 & Q^{(1)}_{-1} & Q^{(1)}_{0} & Q^{(1)}_{1} & \ddots \cr
\vdots & \vdots & \ddots & \ddots & \ddots & \ddots
 \end{matrix} \right ),
\end{align*}
where, using $(-j)^{+} = \max(0,-j)$,
\begin{align*}
 & N^{(1)}_{j} = \left (\begin{matrix}
A^{((-j)^{+}0)}_{j0} & A^{((-j)^{+}0)}_{j1} & 0 & 0 & \cdots \cr
A^{((-j)^{+}1)}_{j(-1)} & A^{((-j)^{+}+)}_{j0} & A^{((-j)^{+}+)}_{j1} & 0 & \cdots \cr
0 & A^{((-j)^{+}+)}_{j(-1)} & A^{((-j)^{+}+)}_{j0} & A^{((-j)^{+}+)}_{j1} & \ddots \cr
\vdots & \vdots & \ddots & \ddots & \ddots
 \end{matrix} \right ),
\end{align*}
\begin{align*}
  & Q^{(1)}_{j} = \left (\begin{matrix}
A^{(+0)}_{j0} & A^{(+0)}_{j1} & 0 & 0 & \cdots \cr
A^{(+1)}_{j(-1)} & A_{j0} & A_{j1} & 0 & \cdots \cr
0 & A_{j(-1)} & A_{j0} & A_{j1} & \ddots \cr
\vdots & \vdots & \ddots & \ddots & \ddots
 \end{matrix} \right ), \qquad j=0,1, -1. 
\end{align*}

We next introduce the Markov additive process by removing the boundary at level $0$ of the one dimensional QBD process generated by $P^{(1)}$, and denote its transition probability matrix by $\ul{P}^{(1)}$. That is,
\begin{align*}
  \ul{P}^{(1)} = \left (\begin{matrix}
\ddots & \ddots & \ddots & \ddots & \ddots & \ddots & \ddots \cr
\ddots & Q^{(1)}_{-1} & Q^{(1)}_{0} & Q^{(1)}_{1} & 0 & 0 & \cdots  \cr
 \cdots &  0 & Q^{(1)}_{-1} & Q^{(1)}_{0} & Q^{(1)}_{1} & 0 & \ddots \cr
 \cdots & 0 & 0 & Q^{(1)}_{-1} & Q^{(1)}_{0} & Q^{(1)}_{1} & \ddots \cr
 \ddots & \ddots & \ddots & \ddots & \ddots & \ddots & \ddots
 \end{matrix} \right ).
\end{align*}
$\ul{P}^{(2)}$ and $Q^{(2)}_{i}$'s are similarly defined exchanging the coordinates. For $i=1,2$, let
\begin{align*}
 & Q^{(i)} = Q^{(i)}_{-1} + Q^{(i)}_{0} + Q^{(i)}_{1},
\end{align*}
then $Q^{(i)}$ is stochastic. Let $\vc{\nu}^{(i)} \equiv \{\vc{\nu}^{(i)}_{\ell}; \ell=0,1,\ldots\}$ be the left invariant positive vector of $Q^{(i)}$ when it exists, where $\vc{\nu}^{(i)}_{\ell}$ is $m_{0}$ and $m$-dimensional vectors for $\ell=0$ and $\ell \ge 1$, respectively. Define
\begin{align*}
   Q^{(i)}_{*}(\theta) = e^{-\theta} Q^{(i)}_{-1} + Q^{(i)}_{0} + e^{\theta} Q^{(i)}_{1}, \qquad i=1,2.
\end{align*}

In order to discuss the stability of the 2d-QBD process, we define the mean drifts $\mu^{(i)}_{i}$ for each $i =1,2$ as
\begin{align*}
  \mu^{(i)}_{i} = \vc{\nu}^{(i)} \frac d{d \theta} Q^{(i)}_{*}(\theta) |_{\theta=0} \vc{1},
\end{align*}
as long as  $Q^{(i)}$ is positive recurrent, where the derivative of a matrix function is taken entry-wise. Let $A = \sum_{j,k \in \dd{H}} A_{jk}$. Since $A$ is stochastic and finite dimensional, it has a stationary distribution. We denote it by the row vector $\vc{\nu}^{(+)}$. Define the mean drifts $\mu_{1}$ and $\mu_{2}$ as
\begin{align*}
  \mu_{1} = \vc{\nu}^{(+)} \sum_{k \in \dd{H}} ( - A_{(-1)k} + A_{1k}) \vc{1}, \qquad \mu_{2} = \vc{\nu}^{(+)} \sum_{j \in \dd{H}} ( - A_{j(-1)} + A_{j1}) \vc{1}.
\end{align*}
Note that, if $\mu_{i} < 0$, then $Q^{(i)}$ is positive recurrent because $\mu_{i}$ is the mean drift at off-boundary states of the QBD process generated by $Q^{(i)}$. We refer to the recent result due to \citet{Ozaw2013}.

\begin{lemma}[Theorem 5.1 and Remark 5.1 of \citet{Ozaw2012}]
\label{lem:2d-QBD stability 1}
  The 2d-QBD process $\{Z_{n}\}$ is positive recurrent if either one of the following three conditions holds.
\begin{mylist}{0}
\item [(i)] If $\mu_{1} < 0$ and $\mu_{2} < 0$, then $\mu^{(1)}_{1} < 0$ and $\mu^{(2)}_{2} < 0$.
\item [(ii)] If $\mu_{1} \ge 0$ and $\mu_{2} < 0$, then $\mu^{(1)}_{1} < 0$.
\item [(iii)] If $\mu_{1} < 0$ and $\mu_{2} \ge 0$, then $\mu^{(2)}_{2} < 0$.
\end{mylist}
On the other hand, if $\mu_{1} > 0$ and $\mu_{2} > 0$, then the 2d-QBD process is transient. Hence, if $(\mu_{1}, \mu_{2}) \ne \vc{0}$, then $\{Z_{n}\}$ is positive recurrent if and only if one of the conditions (i)--(iii) holds.
\end{lemma}
\begin{remark} {\rm
\label{rem:2d-QBD stability 1}
  The stability conditions of this lemma are exactly the same as those of the two dimensional reflecting random walk on the lattice quarter plane of \cite{FayoIasnMaly1999}, which is called a double QBD process in \cite{Miya2009} (see also \cite{KobaMiya2014}). This is not surprising because the stability is generally determined by the mean drifts of so called induced Markov chains, which are generated by removing one of the boundary faces. However, its proof requires careful mathematical arguments, which have been done by \citet{Ozaw2013}.
}\end{remark}

Throughout the paper, we assume that the 2d-QBD process has a stationary distribution, which is denoted by the row vector $\vc{\pi} \equiv \{\pi(\vc{z},k); (\vc{z},k) \in \sr{S}\}$. \lem{2d-QBD stability 1} can be used to verify this stability assumption. However, it is not so useful in application because the signs of $\mu^{(1)}_{1}$ and $\mu^{(2)}_{2}$ are hard to get. Thus, we will not use \lem{2d-QBD stability 1} in our arguments. We will return to this issue later.

\subsection{Tail asymptotics for the stationary distribution}
\label{sect:tail}

\citet{Ozaw2013} studies the tail asymptotics of the stationary distribution of the 2d-QBD process in coordinate directions, assuming stability and some additional assumptions. His arguments are based on the sufficiency part of \thr{characterization 1}. As discussed at the end of \sectn{existence superh}, this is intractable for applications. So far, we will consider the problem in a different way. In the first part of this section, we derive upper and lower bounds for the tail decay rates using relatively easy conditions. We then assume an extra condition similar to \ass{key assumption 1}, and derive the tail decay rate of the marginal stationary distribution in an arbitrary direction.

To describe the modeling primitives, we will use the following matrix moment generating functions.
\begin{align*}
 & A_{**}(\vc{\theta}) = \sum_{j,k \in \dd{H}} e^{-(j\theta_{1} + k\theta_{2})} A_{jk},\\
 & A_{*k}(\theta_{1}) = e^{-\theta_{1}} A_{(-1)k} + A_{0k} + e^{\theta_{1}} A_{1k}, \quad A_{+k}(\theta_{1}) = A_{0k} + e^{\theta_{1}} A_{1k}, \quad k \in \dd{H},\\
 & A^{(+1)}_{*(-1)}(\theta_{1}) = e^{-\theta_{1}} A^{(+1)}_{(-1)(-1)} + A^{(+1)}_{0(-1)} + e^{\theta_{1}} A^{(+1)}_{1(-1)}, \quad A^{(+1)}_{+k}(\theta_{1}) = A_{+k}(\theta_{1}), \quad k \in \dd{H}_{+},\\
 & A^{(+0)}_{*k}(\theta_{1}) = e^{-\theta_{1}} A^{(+0)}_{(-1)k} + A^{(+0)}_{0k} + e^{\theta_{1}} A^{(+0)}_{1k}, \quad A^{(+0)}_{+k}(\theta_{1}) = A^{(+0)}_{0k} + e^{\theta_{1}} A^{(+0)}_{1k}, \quad k \in \dd{H}_{+}.
\end{align*}
Similarly, $A_{j*}(\theta_{2}), A_{j+}(\theta_{2}), A^{(1+)}_{(-1)*}(\theta_{2}), A^{(1+)}_{j+}(\theta_{2}), A^{(0+)}_{j*}(\theta_{2}), A^{(0+)}_{j+}(\theta_{2})$ are defined. 
Thus, we have many matrix moment generating functions, but they are generated by the simple rule that subscripts $*$ and $+$ indicate taking the sums for indices in $\dd{H} \equiv \{0, \pm 1\}$ and $\{0,1\}$, respectively. 

Similar to $C_{0}$ of \eqn{C 0}, we define the $m \times m$ matrix generating functions:
\begin{align}
\label{eqn:A tilde 1}
 & C^{(1)}_{**}(\vc{\theta}) = A_{*+}(\vc{\theta}) + A^{(+1)}_{*(-1)}(\theta_{1}) (I - A^{(+0)}_{*0}(\theta_{1}))^{-1} A^{(+0)}_{*1}(\theta_{1}),\\
 \label{eqn:A tilde 2}
 & C^{(2)}_{**}(\vc{\theta}) = A_{+*}(\vc{\theta}) + A^{(1+)}_{(-1)*}(\theta_{2}) (I - A^{(0+)}_{0*}(\theta_{2}))^{-1} A^{(0+)}_{1*}(\theta_{2}),
\end{align}
where, for $i=1,2$, $c_{p}(A^{(i)}_{*0}(\theta_{i})) > 1$ is assumed as long as $C^{(i)}_{**}(\vc{\theta})$ is used.

Similar to $\Gamma^{(1d)}_{0+}$ and $\theta^{(A,C)}_{\max}$ of \eqn{theta min-max}, let, for $i=1,2$,
\begin{align*}
  \Gamma^{(2d)}_{i+} = \{ \vc{\theta} \in \dd{R}^{2}; \exists \vc{h} > \vc{0}, A_{**}(\vc{\theta}) \vc{h} \le \vc{h}, C^{(i)}_{**}(\vc{\theta}) \vc{h} \le \vc{h} \}.
\end{align*}
We further need the following notation.
\begin{align*}
  \Gamma^{(2d)}_{+} = \{ \vc{\theta} \in \dd{R}^{2}; \exists \vc{h} > \vc{0}, A_{**}(\vc{\theta}) \vc{h} \le \vc{h} \},\quad
  \Gamma^{(2d)}_{\max} = \{ \vc{\theta} \in \dd{R}^{2}; \exists \vc{\theta}' > \vc{\theta}, \vc{\theta}' \in \Gamma^{(2d)}_{+} \}.
\end{align*}
Recall that the Perron-Frobenius eigenvalue of $A_{**}(\vc{\theta})$ is denoted by $\gamma^{(A_{**})}(\vc{\theta})$, which is finite because $A_{**}(\vc{\theta})$ is a finite dimensional matrix. Obviously, we have
\begin{align}
\label{eqn:gamma ++}
  \Gamma^{(2d)}_{+} = \{ \vc{\theta} \in \dd{R}^{2}; \gamma^{(A_{**})}(\vc{\theta}) \le 1 \}.
\end{align}

We now define key points for $i = 1,2$.
\begin{align*}
  \vc{\theta}^{(i,\cp)} = \arg_{\vc{\theta} \in \dd{R}^{2}} \sup\{ \theta_{i} \ge 0; \vc{\theta} \in \Gamma^{(2d)}_{i+} \} ,\quad
  \vc{\theta}^{(i,\max)} = \arg_{\vc{\theta} \in \dd{R}^{2}} \sup\{ \theta_{i}; \vc{\theta} \in \Gamma^{(2d)}_{+} \}.
\end{align*}
Using these points, we define the vector $\vc{\tau}$ by
\begin{align}
\label{eqn:tau def 1}
  \tau_{1} = \sup\{\theta_{1} \in \dd{R}; \vc{\theta} \in \Gamma^{(2d)}_{1+}; \theta_{2} < \theta^{(2,\cp)}_{2} \},\quad
  \tau_{2} = \sup\{\theta_{2} \in \dd{R}; \vc{\theta} \in \Gamma^{(2d)}_{2+}; \theta_{1} < \theta^{(1,\cp)}_{1} \}.
\end{align}
Note that $\tau_{i}$ is finite because $\Gamma^{(2d)}_{i+}$ is a bounded set. It is notable that, in the definitions \eqn{tau def 1}, the condition that $A_{**}(\vc{\theta}) \vc{h} \le \vc{h}$ can be replaced by $A_{**}(\vc{\theta}) \vc{h} = \vc{h}$ or, equivalently, $\gamma^{A_{**}}(\vc{\theta}) = 1$.

For $i = 1,2$, define the function $\ol{\xi}^{(i)}(\theta_{3-i})$ for $\vc{\theta} = (\theta_{1}, \theta_{2})$ as
\begin{align*}
  \ol{\xi}^{(i)}(\theta_{3-i}) = \sup \{ \theta_{i} \in \dd{R}; \vc{\theta} \in \Gamma^{(2d)}_{i+}\} \quad \mbox{for} \quad \theta_{3-i} \in [0, \theta^{(i,\cp)}_{i}].
\end{align*}
Obviously, $\ol{\xi}^{(i)}(x)$ is a convex function because $\Gamma^{(2d)}_{i+}$ is a bounded set.

As in \cite{KobaMiya2012,KobaMiya2014}, it is convenient to introduce the following classifications for $\vc{\tau} \equiv (\tau_{1}, \tau_{2})$.
\begin{mylist}{5}
\item [(Category I)] $\theta^{(1,\cp)}_{2} < \theta^{(2,\cp)}_{2}$ and $\theta^{(2,\cp)}_{1} < \theta^{(1,\cp)}_{1}$, for which $\vc{\tau} = (\theta^{(1,\cp)}_{1}, \theta^{(2,\cp)}_{2})$.
\item [(Category II-1)] $\theta^{(1,\cp)}_{2} \ge \theta^{(2,\cp)}_{2}$ and $\theta^{(2,\cp)}_{1} < \theta^{(1,\cp)}_{1}$, for which $\vc{\tau} = (\ol{\xi}^{(1)}(\theta^{(2,\cp)}_{2}), \theta^{(2,\cp)}_{2})$.
\item [(Category II-2)] $\theta^{(1,\cp)}_{2} < \theta^{(2,\cp)}_{2}$ and $\theta^{(2,\cp)}_{1} \ge \theta^{(1,\cp)}_{1}$, for which $\vc{\tau} = (\theta^{(1,\cp)}_{1}, \ol{\xi}^{(2)}(\theta^{(1,\cp)}_{1}))$.
\end{mylist}
Since it is impossible that $\theta^{(1,\cp)}_{2} \ge \theta^{(2,\cp)}_{2}$ and $\theta^{(2,\cp)}_{1} \ge \theta^{(1,\cp)}_{1}$, these three categories completely cover the all cases (e.g., see Section 4 of \cite{Miya2009}). These categories are crucial in our arguments as we shall see \thr{lower bound 1} below.

We first derive upper bounds. Let $\varphi$ be the moment generating function of $\vc{L}$. Namely, $\varphi(\vc{\theta}) = \dd{E}(e^{\br{\vc{L}, \vc{\theta}}})$. Define its convergence domain as
\begin{align*}
  \sr{D} = \{ \vc{\theta} \in \dd{R}^{2}; \exists \vc{\theta}' > \vc{\theta}, \varphi(\vc{\theta}') < \infty \}.
\end{align*}
We prove the following lemma in \app{domain 0}. 
\begin{lemma}
\label{lem:domain 0}
Under the stability assumption,
\begin{align}
\label{eqn:domain 0}
   \Gamma^{(2d)}_{\vc{\tau}} \equiv \{ \vc{\theta} \in \Gamma^{(2d)}_{\max}; \vc{\theta} < \vc{\tau} \} \subset \sr{D}.
\end{align}
\end{lemma}
Using this lemma, the following upper bound is obtained.

\begin{theorem}
\label{thr:upper bound 1}
  Under the stability condition, we have, for each non-zero vector $\vc{c} \ge \vc{0}$,
\begin{align}
\label{eqn:upper bound 1}
  \lim_{x \to \infty} \frac 1x \log \dd{P}(\br{\vc{L}, \vc{c}} > x) \le - \sup \{u \ge 0; u \vc{c} \in \Gamma^{(2d)}_{\vc{\tau}}\}.
\end{align}
\end{theorem}

This theorem is proved  in \app{upper bound 1}. We next derive lower bounds. We first consider lower bounds concerning the random walk component in an arbitrary direction. For this, we consider the two dimensional random walk modulated by $\{A_{jk}; j,k \in \dd{H}\}$, which is denoted by $\{(\vc{Y}_{n}, J_{n}); n \ge 1\}$. Similar to Lemma 7 of \cite{KobaMiya2014}, we have the following fact, which is proved in \app{lower bound 2}.

\begin{lemma}
\label{lem:lower bound 2}
  For each non-zero vector $\vc{c} \ge \vc{0}$,  
\begin{align}
\label{eqn:lower bound 3}
  \liminf_{x \to \infty} \frac 1{x} \log \dd{P}(\vc{L} > x \vc{c} ) \ge - \sup\{ \br{\vc{\theta}, \vc{c}}; \vc{\theta} \in \Gamma^{(2d)}_{+} \},
\end{align}
  and therefore $\varphi(\vc{\theta})$ is infinite for $\vc{\theta} \not\in \ol{\Gamma}^{(2d)}_{\max}$, where $\ol{\Gamma}^{(2d)}_{\max}$ is the closure of $\Gamma^{(2d)}_{\max}$.
\end{lemma}

Note that the upper bound in \eqn{upper bound 1} is generally larger than the lower bound in \eqn{lower bound 3}. To get tighter lower bounds, we use the one dimensional QBD formulation. For this, we require assumptions similar to \ass{key assumption 1}.

\begin{assumption} {\rm
\label{ass:key assumption 2}
For each $\vc{\theta} \in \dd{R}^{2}$ satisfying that $\gamma^{(A_{**})}(\vc{\theta}) = 1$, for each $i=1,2$, there is an $m_{0}$-dimensional positive vector $\vc{h}^{(0i)}(\vc{\theta})$ and functions $c^{(i)}_{0}(\vc{\theta})$ and $c^{(i)}_{1}(\vc{\theta})$ such that either one of $c^{(i)}_{1}(\vc{\theta})$ or $c^{(i)}_{2}(\vc{\theta})$ equals one, and
\begin{align}
\label{eqn:key assumption 2a}
 & A^{(i)}_{*(i0)}(\theta_{i}) \vc{h}^{(0i)}(\vc{\theta}) + e^{\theta_{2}} A^{(i)}_{*(i1)}(\theta_{i}) \vc{h}^{(A_{**})}(\vc{\theta}) = c^{(i)}_{0}(\vc{\theta}) \vc{h}^{(0i)}(\vc{\theta}),\\
\label{eqn:key assumption 2b}
 & \lefteqn{e^{-\theta_{2}} A^{(i)}_{*(i(-1))}(\theta_{i}) \vc{h}^{(0i)}(\vc{\theta}) + A_{*(i0)}(\theta_{i}) \vc{h}^{(A_{**})}(\vc{\theta}) + e^{\theta_{2}} A_{*(i1)}(\theta_{i}) \vc{h}^{(A_{**})}(\vc{\theta})} \nonumber \\
 & \hspace{50ex} = c^{(i)}_{1}(\vc{\theta}) \vc{h}^{(A_{**})}(\vc{\theta}),
\end{align}
where $*(ik) = *k$ for $i=1$ and $*(ik) = k*$ for $i=2$. We recall that $\vc{h}^{(A_{**})}(\vc{\theta})$ is the Perron-Frobenius eigenvector of $A_{**}(\vc{\theta})$.
}\end{assumption}

\begin{theorem}
\label{thr:lower bound 1}
  Assume that the 2d-QBD process has a stationary distribution and \ass{key assumption 2}. Then, we have the following facts for each $i = 1,2$. For each $\ell \ge 0$ and either $k \in \sr{V}_{1}$ for $\ell=0$ or $k \in \sr{V}_{+}$ for $\ell \ge 1$,
\begin{align}
\label{eqn:lower bound 1}
  \lim_{n \to \infty} \frac 1n \log \dd{P}(L_{i} > n, L_{3-i} = \ell, J = k) = - \tau_{i}.
\end{align}
In particular, for Category I satisfying $\tau_{i} < \theta^{(i,\max)}_{i}$, there is a positive constant $c^{(i)}_{\ell k}$ such that
\begin{align}
\label{eqn:direction 1}
  \lim_{n \to \infty} e^{\tau_{i} n} \dd{P}(L_{i} > n, L_{3-i} = \ell, J = k) = c^{(i)}_{\ell k}.
\end{align}
Otherwise, for Category II-i satisfying $\tau_{i} < \theta^{(i,\cp)}_{i}$, there are positive constants $\ul{d}^{(i)}_{\ell k}$ and $\ol{d}^{(i)}_{\ell k}$ such that
\begin{align}
\label{eqn:direction 2}
 & \liminf_{n \to \infty} e^{\tau_{i} n} \dd{P}(L_{i} > n, L_{3-i} = \ell, J = k) \ge \ul{d}^{(i)}_{\ell k},\\
\label{eqn:direction 3}
 & \limsup_{n \to \infty} e^{\tau_{i} n} \dd{P}(L_{i} > n, L_{3-i} = \ell, J = k) \le \ol{d}^{(i)}_{\ell k}.
\end{align}
\end{theorem}

This theorem will be proved in \app{one}. Similar results without \ass{key assumption 2} are obtained as Theorem 4.1 in \cite{Ozaw2013}. However, the method assumes other assumptions such as Assumption 3.1 of \cite{Ozaw2013}. Furthermore, it requires a large amount of numerical work to compute $\tau_{i}$.

Combining Theorems \thrt{upper bound 1} and \thrt{lower bound 1} and \lem{lower bound 2}, we have the following tail asymptotics.

\begin{theorem}
\label{thr:decay rate 1}
  Under the assumptions of \thr{lower bound 1}, we have
\begin{align}
\label{eqn:domain 1}
  \sr{D} = \Gamma^{(2d)}_{\vc{\tau}} \equiv \{ \vc{\theta} \in \Gamma^{(2d)}_{\max}; \vc{\theta} < \vc{\tau} \}.
\end{align}
and, for each non-zero vector $\vc{c} \ge \vc{0}$,
\begin{align}
\label{eqn:decay rate 1}
  \lim_{x \to \infty} \frac 1x \log \dd{P}(\br{\vc{L}, \vc{c}} > x) = - \sup \{u \ge 0; u \vc{c} \in \sr{D}\}.
\end{align}
\end{theorem}

\begin{proof}
By \thr{upper bound 1}, we already have the upper bound of the tail probability for \eqn{decay rate 1}. To consider the lower bound, let
\begin{align*}
  u_{\vc{c}} = \sup \{u \ge 0; u \vc{c} \in \sr{D}\}, \qquad \vc{\theta}(\vc{c}) = u_{\vc{c}} \vc{c}.
\end{align*}
Note that $\vc{\theta}(\vc{c}) \le \vc{\tau}$ by \thr{upper bound 1}. We first assume that $\vc{\theta}_{\vc{c}} < \vc{\tau}$. Then, by \thr{upper bound 1}, $\vc{\theta}(\vc{c})\in \partial \Gamma^{(2d)}_{+}$, and therefore \lem{lower bound 2} leads to
\begin{align}
\label{eqn:lower bound 3a}
    \liminf_{x \to \infty} \frac 1x \log \dd{P}(\br{\vc{L}, \vc{c}} > x, J=k) \ge - u_{\vc{c}} = -\sup \{u \ge 0; u \vc{c} \in \sr{D}\}.
\end{align}
Assume $\vc{\theta}_{\vc{c}} = \vc{\tau}$. In this case, by \thr{upper bound 1}, we have that $[\vc{\theta}(\vc{c})]_{1} = \tau_{1}$ or $[\vc{\theta}(\vc{c})]_{2} = \tau_{2}$, equivalently, $c_{1} u_{\vc{c}} = \tau_{1}$ or $c_{2} u_{\vc{c}} = \tau_{2}$. Since these two cases are symmetric, we only consider the case for $c_{1} u_{\vc{c}} = \tau_{1}$. By \thr{lower bound 1}, we have, for each fixed $\ell$ and $k$,
\begin{align*}
  \liminf_{n \to \infty} \frac 1n \log \dd{P}(c_{1} L_{1} > n, L_{2} = \ell, J = k) \ge - \frac {\tau_{1}} {c_{1}} = - u_{\vc{c}}.
\end{align*}
Since $c_{1} L_{1} + c_{2} L_{2} > n$ for $L_{2} = \ell$ implies that $c_{1} L_{1} > n - c_{2} \ell$, this yields
\begin{align*}
  \liminf_{n \to \infty} \frac 1n \log \dd{P}(\br{\vc{L}, \vc{c}} > x, J=k) \ge - u_{\vc{c}} = - \sup \{u \ge 0; u \vc{c} \in \sr{D}\}.
\end{align*}
Thus, the limit supremum and the limit infimum are identical, and we get \eqn{decay rate 1}.
\end{proof}

\section{Two node generalized Jackson network}
\setnewcounter
\label{sect:two node}

In this section, we consider a continuous time Markov chain $\{(\vc{L}(t),J(t))\}$ whose embedded transitions under uniformization constitute a discrete-time 2d-QBD process. We refer it as a continuous-time 2d-QBD process. This process is convenient in queueing applications because they are often of continuous time. Since the stationary distribution is unchanged under uniformization, its tail asymptotics are also unchanged. Thus, it is routine to convert the asymptotic results obtained for the discrete-time 2d-QBD process to those for $\{(\vc{L}(t),J(t))\}$. We summarize them for convenience of application.

\subsection{Continuous time formulation of a 2d-QBD process}
\label{sect:continuous formulation}

As discussed above, we define a continuous time 2d-QBD process $\{(\vc{L}(t),J(t))\}$ by changing $P^{(1)}$ (or $P^{(2)}$) to a transition rate matrix. Denote it by $\tilde{P}^{(1)}$ (or $\tilde{P}^{(2)}$). That is, $\tilde{P}^{(i)}$ has the same block structure as that of $P^{(i)}$ while $\tilde{P}^{(i)} \vc{1} = \vc{0}$ and all its diagonal entries are not positive. In what follows, continuous time characteristics are indicated by tilde except for those concerning the stationary distribution because the stationary distribution is unchanged. Among them, it is notable that $I - A^{(i)}_{00}$ and $I - A_{00}$ are replaced by $-\tilde{A}^{(i)}_{00}$ and $-\tilde{A}_{00}$, respectively, while $A^{(i)}_{ij}$ and $A_{ij}$ are replaced by $\tilde{A}^{(i)}_{ij}$ and $\tilde{A}_{ij}$ for $(i,j) \ne (0,0)$. Similarly, $A^{(i)}_{jk}$, $A_{**}(\vc{\theta})$ and $C^{(i)}_{**}(\vc{\theta})$ are defined. For example, $I - A^{(1)}_{*0}(\theta_{1})$ is replaced by $- \tilde{A}^{(1)}_{*0}(\theta_{1})$, and therefore
\begin{align*}
  \tilde{C}^{(1)}_{**}(\vc{\theta}) = \tilde{A}_{*+}(\vc{\theta}) + \tilde{A}^{(1)}_{*(-1)}(\theta_{1}) (- \tilde{A}^{(1)}_{*0}(\theta_{1}))^{-1} \tilde{A}^{(1)}_{*1}(\theta_{1}),
\end{align*}
as long as $(- \tilde{A}^{(1)}_{*0}(\theta_{1}))^{-1}$ exists and is nonnegative. $\tilde{C}^{(2)}_{**}(\vc{\theta})$ is similarly defined.

Suppose that we start with the continuous time 2d-QBD process with primitive data $\tilde{A}^{(i)}_{jk}$. These data must satisfy
\begin{align*}
  \tilde{A}_{**}(\vc{0}) \vc{1} = \vc{0}, \qquad \tilde{C}^{(i)}_{**}(\vc{0}) \vc{1} = \vc{0}, \quad i=1,2,
\end{align*}
because of the continuous time settings. Since the condition for the existence of a superharmonic vector $\vc{h}$ for $A_{**}(\vc{\theta})$ is changed to $\tilde{A}_{**}(\vc{\theta}) \vc{h} \le \vc{0}$, we define the following sets.
\begin{align}
\label{eqn:Gamma 2d+}
 & \tilde{\Gamma}^{(2d)}_{+} = \{ \vc{\theta} \in \dd{R}^{2}; \exists \vc{h} > \vc{0}, \tilde{A}_{**}(\vc{\theta}) \vc{h} \le \vc{0} \},\\
\label{eqn:Gamma 2dmax}
 & \tilde{\Gamma}^{(2d)}_{\max} = \{ \vc{\theta} \in \dd{R}^{2}; \exists \vc{\theta}' > \vc{\theta}, \vc{\theta}' \in \tilde{\Gamma}^{(2d)}_{1}(\tilde{A}_{**}) \},\\
\label{eqn:Gamma 2di}
 & \tilde{\Gamma}^{(2d)}_{i+} = \{ \vc{\theta} \in \dd{R}^{2}; \exists \vc{h} > \vc{0}, \tilde{A}_{**}(\vc{\theta}) \vc{h} \le \vc{0}, \tilde{C}^{(i)}_{**}(\vc{\theta}) \vc{h} \le \vc{0}\}, \quad s=1,2.
\end{align}
The following auxiliary notation will be convenient.
\begin{align}
\label{eqn:Gamma 2ei}
  \tilde{\Gamma}^{(2e)}_{i+} = \{ \vc{\theta} \in \dd{R}^{2}; \exists \vc{h} > \vc{0}, \tilde{A}_{**}(\vc{\theta}) \vc{h} = \vc{0}, \tilde{C}^{(i)}_{**}(\vc{\theta}) \vc{h} \le \vc{0}\}, \quad i=1,2.
\end{align}
Using these notation, we define
\begin{align*}
  \tilde{\vc{\theta}}^{(i,\cp)} = \arg_{\vc{\theta} \in \dd{R}^{2}} \sup\{ \theta_{i} \ge 0; \vc{\theta} \in \tilde{\Gamma}^{(2d)}_{i+} \} ,\quad
  \tilde{\vc{\theta}}^{(i,\max)} = \arg_{\vc{\theta} \in \dd{R}^{2}} \sup\{ \theta_{i}; \vc{\theta} \in \tilde{\Gamma}^{(2d)}_{+} \}.
\end{align*}
and define the vector $\tilde{\vc{\tau}}$ by
\begin{align}
\label{eqn:c tau def 1}
  \tilde{\tau}_{1} = \sup\{\theta_{1} \in \dd{R}; \vc{\theta} \in \tilde{\Gamma}^{(2d)}_{1+}; \theta_{2} < \tilde{\theta}^{(2,\cp)}_{2} \},\quad
  \tilde{\tau}_{2} = \sup\{\theta_{2} \in \dd{R}; \vc{\theta} \in \tilde{\Gamma}^{(2d)}_{2+}; \theta_{1} < \tilde{\theta}^{(1,\cp)}_{1} \}.
\end{align}
\vspace{-4ex}
\begin{remark}
\label{rem:Gamma 2ei}
 In the definition \eqn{c tau def 1}, we can replace $\tilde{\Gamma}^{(2d)}_{i+}$ by $\tilde{\Gamma}^{(2e)}_{i+}$ because $\tilde{\Gamma}^{(2d)}_{+}$ and $\{ \vc{\theta} \in \dd{R}^{2}; \tilde{C}^{(i)}_{**}(\vc{\theta}) \vc{h} \le \vc{0} \}$ are closed convex sets.
\end{remark}

We also need
\begin{align}
\label{eqn:c tau max}
  \tilde{\Gamma}_{\tilde{\vc{\tau}}} = \{ \vc{\theta} \in \tilde{\Gamma}^{(2d)}_{\max}(\tilde{A}); \vc{\theta} < \tilde{\vc{\tau}} \}.
\end{align}
Let $\tilde{\gamma}^{\tilde{A}_{**}}(\vc{\theta})$ be the Perron-Frobenius eigenvalue of $\tilde{A}_{**}(\vc{\theta})$. A continuous time version of \ass{key assumption 2} is given by

\begin{assumption} {\rm
\label{ass:key assumption 3}
For each $\vc{\theta} \in \dd{R}^{2}$ satisfying that $\gamma^{(\tilde{A}_{**})}(\vc{\theta}) = 1$, for each $i=1,2$, there is an $m_{0}$-dimensional positive vector $\tilde{\vc{h}}^{(i0)}(\vc{\theta})$ and functions $\tilde{c}^{(i)}_{0}(\vc{\theta})$ and $\tilde{c}^{(i)}_{1}(\vc{\theta})$ such that one of $\tilde{c}^{(i)}_{0}(\vc{\theta})$ or $\tilde{c}^{(i)}_{1}(\vc{\theta})$ vanishes, and, for $i=1$,
\begin{align}
\label{eqn:key assumption 3a}
 & \tilde{A}^{(+0)}_{*0}(\theta_{1}) \tilde{\vc{h}}^{(10)}(\vc{\theta}) + e^{\theta_{2}} \tilde{A}^{(+0)}_{*1}(\theta_{1}) \tilde{\vc{h}}^{(A_{**})}(\vc{\theta}) = \tilde{c}^{(1)}_{0}(\vc{\theta}) \tilde{\vc{h}}^{(10)}(\vc{\theta}),\\
\label{eqn:key assumption 3b}
 & e^{-\theta_{2}} \tilde{A}^{(+1)}_{*(-1)}(\theta_{1}) \tilde{\vc{h}}^{(10)}(\vc{\theta}) + \tilde{A}_{*+}(\theta_{1}) \tilde{\vc{h}}^{(\tilde{A}_{**})}(\vc{\theta}) = \tilde{c}^{(1)}_{1}(\vc{\theta}) \tilde{\vc{h}}^{(\tilde{A}_{**})}(\vc{\theta}), \qquad
\end{align}
and, for $i=2$,
\begin{align}
\label{eqn:key assumption 3c}
 & \tilde{A}^{(0+)}_{0*}(\theta_{2}) \tilde{\vc{h}}^{(20)}(\vc{\theta}) + e^{\theta_{1}} \tilde{A}^{(0+)}_{1*}(\theta_{2}) \tilde{\vc{h}}^{(A_{**})}(\vc{\theta}) = \tilde{c}^{(2)}_{0}(\vc{\theta}) \tilde{\vc{h}}^{(20)}(\vc{\theta}),\\
\label{eqn:key assumption 3d}
 & e^{-\theta_{1}} \tilde{A}^{(1+)}_{(-1)*}(\theta_{2}) \tilde{\vc{h}}^{(20)}(\vc{\theta}) + \tilde{A}_{+*}(\theta_{2}) \tilde{\vc{h}}^{(\tilde{A}_{**})}(\vc{\theta}) = \tilde{c}^{(2)}_{1}(\vc{\theta}) \tilde{\vc{h}}^{(\tilde{A}_{**})}(\vc{\theta}). \qquad
\end{align}
We recall that $\tilde{\vc{h}}^{(\tilde{A}_{**})}(\vc{\theta})$ is the Perron-Frobenius eigenvector of $\tilde{A}_{**}(\vc{\theta})$.
}\end{assumption}

Define the domain for the stationary distribution of $\vc{L}$ as
\begin{align}
\label{eqn:domain c time}
  \sr{D} = \mbox{the interior of } \{\vc{\theta} \in \dd{R}^{2}; \dd{E}(e^{\br{\vc{\theta},\vc{L}}}) < \infty\},
\end{align}
where $\vc{L}$ is a random vector subject to the stationary distribution of $\vc{L}(t)$. It is easy to see that Theorems \thrt{upper bound 1} and \thrt{decay rate 1} can be combined and converted into the following continuous version.

\begin{theorem}
\label{thr:decay rate 2}
  For a continuous-time 2d-QBD process satisfying the irreducibility and stability conditions, $\tilde{\Gamma}_{\tilde{\vc{\tau}}} \subset \sr{D}$ and we have
\begin{align}
\label{eqn:decay rate 3}
  \lim_{x \to \infty} \frac 1x \log \dd{P}(\br{\vc{L}, \vc{c}} > x) \le - \sup \{u \ge 0; u\vc{c} \in \tilde{\Gamma}_{\tilde{\vc{\tau}}}\},
\end{align}
and this inequality becomes equality with $\sr{D} = \tilde{\Gamma}_{\tilde{\vc{\tau}}}$ if \ass{key assumption 3} is satisfied.
\end{theorem}

\subsection{Two node generalized Jackson network with MAP arrivals and PH-service time distributions}
\label{sect:Jackson}

 As an example of the 2d-QBD process, we consider a two node generalized Jackson network with a MAP arrival process and phase type service time distributions. Obviously, this model can be formulated as a 2d-QBD process. We are interested to see how exogenous arrival processes and service time distributions influence the decay rates. This question has been partially answered for the tail decay rates of the marginal distributions of tandem queues with stationary or renewal inputs (e.g. see \cite{BertPascTsit1998,GaneAnan1996}). They basically use the technique for sample path large deviations, and no joint distributions has been studied for queue lengths at multiple nodes. For Markov modulated arrivals and more general network topologies, there is seminal work by Takahashi and his colleagues \cite{FujiTaka1996,FujiTakaMaki1998,KatoMakiTaka2004,KatoMakiTaka2008}. They started with numerical examinations and finally arrived at upper bounds for the stationary tail probabilities for the present generalized Jackson network in \cite{KatoMakiTaka2008}. The author \cite{Miya2003} conjectured the tail decay rates of the stationary distribution for a $d$-node generalized Jackson network with $d \ge 2$ and renewal arrivals.

Thus, the question has not yet been satisfactorily answered particularly for a network with feedback routes. This motivates us to study the present decay rate problem. As we will see, the answer is relatively simple, and naturally generalizes the tandem queue case. However, first we have to introduce yet more notation to describe the generalized Jackson network. This network has two nodes, which are numbered as $1$ and $2$. We make the following modeling assumptions.
\begin{itemize}
\item [(\sect{two node}a)] A customer which completes service at node $i$ goes to node $j$ with probability $r_{ij}$ or leaves the network with probability $1 - r_{ij}$ for $(i,j) = (1,2)$ or $(2,1)$, where $r_{12} + r_{21} > 0$ and $r_{12} r_{21} < 1$, which exclude obvious cases. This routing of customers is assumed to be independent of everything else.
\item [(\sect{two node}b)] Exogenous customers arrive at node $i$ subject to the Markovian arrival process with generator ${T}_{i}+{U}_{i}$, where ${U}_{i}$ generates arrivals. Here, ${T}_{i}$ and ${U}_{i}$ are finite square matrices of the same size for each $i=1,2$.
\item [(\sect{two node}c)] Node $i$ has a single server, whose service times are independently and identically distributed subject to a phase type distribution with $(\vc{\beta}_{i}, {S}_{i})$, where $\vc{\beta}_{i}$ is the row vector representing the initial phase distribution and ${S}_{i}$ is a transition rate matrix for internal state transitions. Here, ${S}_{i}$ is a finite square matrix, and $\vc{\beta}_{i}$ has the same dimension as that of ${S}_{i}$ for each $i=1,2$.
\end{itemize}

Let ${D}_{i} = (-{S}_{i} \vc{1}) \vc{\beta}_{i}$, then ${S}_{i}+{D}_{i}$ is a generator for a continuous time Markov chain which generates completion of service times with rate ${D}_{i}$. Since the service time distribution at node $i$ has the phase type distribution with $(\vc{\beta}_{i}, {S}_{i})$, its moment generating function $g_{i}$ of  is given by
\begin{align}
\label{eqn:g i}
  g_{i}(\theta) = \br{\vc{\beta}_{i}, (-\theta I_{i+2} - {S}_{i})^{-1} (-{S}_{i} \vc{1})}, \qquad i=1,2,
\end{align}
as long as $\theta I_{i+2} + {S}_{i}$ is non-singular (e.g., see \cite{LatoRama1999} in which the Laplace transform is used instead of the moment generating function).  Clearly, $g_{i}(\theta)$ is a increasing function of $\theta$ from $(-\infty, \theta_{0i})$ to $(0, \infty)$, where $- \theta_{0i}$ is the Perron-Frobenius eigenvalue of $S_{i}$.

Let $L_{i}(t)$, $J_{ia}(t)$ and $J_{ib}(t)$ be the number of customers at node $i$, the background state for arrivals and the phase of service in progress, respectively, at time $t$, where $J_{ib}(t)$ is undefined if there is no customer in node $i$ at time $t$. Then, it is not hard to see that $\{(\vc{L}(t), \vc{J}(t)); t \ge 0\}$ is a continuous-time Markov chain and considered as a 2d-QBD process, where $\vc{L}(t) = (L_{1}(t), L_{2}(t))$ and $\vc{J}(t) = (J_{1a}(t), J_{2a}(t), J_{1b}(t), J_{2b}(t))$, where $J_{ib}(t)$ is removed from the components of $\vc{J}(t)$ if it is undefined.

We first note the stability condition for this 2d-QBD process. Since, for node $i$, the mean exogenous arrival rate $\lambda_{i}$ and the mean service rate $\mu_{i}$ are given by
\begin{align*}
  \lambda_{i} = \br{\vc{\nu}_{i}, U_{i} \vc{1}_{i}} \qquad \mu_{i} = \br{\vc{\beta}_{i}, (-S_{i})^{-1} \vc{1}_{i}},
\end{align*}
where $\vc{\nu}_{i}$ is the stationary distribution of the Markov chain with generator $T_{i} + U_{i}$, it is well known that the stability condition is given by
\begin{align}
\label{eqn:stability of GJ}
  \rho_{i} \equiv \frac {\lambda_{i} + \lambda_{3-i} r_{(3-i)i}} {(1-r_{12} r_{21}) \mu_{i}} < 1, \qquad i=1,2.
\end{align}
We assume this condition throughout in \sectn{Jackson}.

We next introduce point processes to count arriving and departing customers from each node. By $N^{(a)}_{i}(t)$, we denote the number of exogenous arriving customers at node $i$ during the time interval $[0,t]$. Then, it follows from (\sect{two node}b) (also the comment above \eqn{stability of GJ}) that
\begin{align*}
  \dd{E}(e^{\theta N^{(a)}_{i}(t)}1(J(t)=k)|J(0)=j) = \left[\exp(t({T}_{i} + e^{\theta} {U}_{i}))\right]_{jk}.
\end{align*}
We define a time-average cumulant moment generating function $\gamma^{(ia)}$ as
\begin{align}
\label{eqn:gamma ia}
  \gamma^{(ia)}(\theta) = \lim_{t \to \infty} \frac 1t \log \dd{E}(e^{\theta N^{(a)}_{i}(t)}), \qquad i=1,2.
\end{align}
It is not hard to see that $\gamma^{(ia)}(\theta)$ is the Perron-Frobenius eigenvalue of ${T}_{i} + e^{\theta} {U}_{i}$.

By $N^{(d)}_{i}(t)$, we denote the number of departing customers from node $i$ during the time interval $[0,t]$ when the server at node $i$ is always busy in this time interval. Let $\Phi_{i}(n)$ be the number of customers who are routed to node $3-i$ among $n$ customers departing from node $i$. Obviously, it follows from (\sect{two node}a) that $\Phi_{i}(n)$ is independent of $N^{(d)}_{i}(t)$, and has the Bernoulli distribution with parameter $(n, r_{i(3-i)})$. Then,
\begin{align*}
  \lefteqn{\dd{E}(e^{-\theta_{i} N^{(d)}_{i}(t) + \theta_{3-i} \Phi_{i}(N^{(d)}_{i}(t))}1(J(t)=k)|J(0)=j)} \hspace{15ex}\\
  & = \left[\exp(t({S}_{i} + e^{-\theta_{i}} (r_{i0} + e^{\theta_{3-i}} r_{i(3-i)}) {D}_{i})\right]_{jk},
\end{align*}
where $r_{i0} = 1 - r_{i(3-i)}$. Similar to $\gamma^{(ia)}$, we define a time-average cumulant moment generating function $\gamma^{(id)}$ by
\begin{align*}
  \gamma^{(id)}(\vc{\theta}) = \lim_{t \to \infty} \frac 1t \log \dd{E}(e^{-\theta_{i} N^{(d)}_{i}(t) + \theta_{3-i} \Phi_{i}(N^{(d)}_{i}(t))}), \qquad \vc{\theta} = (\theta_{1}, \theta_{2}), \; i=1,2.
\end{align*}
One can see that $\gamma^{(id)}(\vc{\theta})$ is the Perron-Frobenius eigenvalue of ${S}_{i} + e^{-\theta_{i}} (r_{i0} + e^{\theta_{3-i}} r_{i(3-i)}) {D}_{i}$.

One may expect that the decay rates for the generalized Jackson network are completely determined by the cumulants $\gamma^{(1a)}, \gamma^{(2a)}, \gamma^{(1d)}, \gamma^{(2d)}$ since their conjugates are known to be rate functions for the Cram\'{e}r type of large deviations. We will show that this is indeed the case. Let
\begin{align*}
 & \gamma^{(+)}(\vc{\theta}) = \gamma^{(1a)}(\theta_{1}) + \gamma^{(2a)}(\theta_{2}) + \gamma^{(1d)}(\vc{\theta}) + \gamma^{(2d)}(\vc{\theta}),\\
 & \gamma^{(i)}(\vc{\theta}) = \gamma^{(1a)}(\theta_{1}) + \gamma^{(2a)}(\theta_{2}) + \gamma^{(id)}(\vc{\theta}), \qquad i=1,2.
\end{align*}
We then have the following result.

\begin{theorem}
\label{thr:decay rate Jackson}
  For the generalized Jackson network satisfying conditions (\sect{two node}a) (\sect{two node}b) and (\sect{two node}c), if the stability condition \eqn{stability of GJ} holds, then \ass{key assumption 3} is satisfied, and we have
\begin{align}
\label{eqn:G-Jackson domain 1}
& \tilde{\Gamma}^{2d}_{+} = \{ \vc{\theta} \in \dd{R}^{2}; \gamma^{(+)}(\vc{\theta}) \le 0 \}, \\
\label{eqn:G-Jackson domain 2}
& \tilde{\Gamma}^{2e}_{i+} = \{ \vc{\theta} \in \dd{R}^{2}; \gamma^{(+)}(\vc{\theta}) = 0, \gamma^{(i)}(\vc{\theta}) \le 0 \}, \quad i=1,2. \quad
\end{align}
Define $\tilde{\tau}$ and $\tilde{\Gamma}_{\tilde{\tau}}$ by \eqn{c tau def 1} and \eqn{c tau max}, then the domain $\sr{D}$ for $\vc{L}$ is given by $\tilde{\Gamma}_{\tilde{\vc{\tau}}}$ and, for non-zero vector $\vc{c} \ge \vc{0}$,
\begin{align}
\label{eqn:decay rate Jackson 1}
  \lim_{x \to \infty} \frac 1x \log \dd{P}(\br{\vc{L}, \vc{c}} > x) = - \sup \{u \ge 0; u\vc{c} \in \tilde{\Gamma}_{\tilde{\vc{\tau}}}\},
\end{align}
where $\vc{L}$ is a random vector subject to the stationary distribution of $\vc{L}(t)$.
\end{theorem}

\begin{remark} {\rm
\label{rem:decay rate Jackson 1}
As we will show at the end of \sectn{proof of Jackson}, the condition that $\gamma^{(i)}(\vc{\theta}) \le 0$ in \eqn{G-Jackson domain 2} can be replaced by $e^{-\theta_{3-i}} (r_{(3-i)0} + e^{\theta_{i}} r_{(3-i)i}) \ge 1$.
}\end{remark}

\begin{remark} {\rm
\label{rem:decay rate Jackson 2}
For $\vc{c} = (1,0), (0,1)$, \citet{KatoMakiTaka2004} obtained the right-hand side of \eqn{decay rate Jackson 1} as an upper bound for its left-hand side (see Theorem 4.1 there). Namely, they derived the inequality \eqn{decay rate 3}, which is conjectured to be tight in \cite{Miya2003}. \thr{decay rate Jackson} shows that those upper bounds are indeed tight. Based on the results in \cite{KatoMakiTaka2004}, \citet{KatoMakiTaka2008} derived upper bounds for the decay rate of the probability $\dd{P}(\vc{L} = n \vc{c} + \vc{d})$ for positive vectors $\vc{c}, \vc{d}$ with integer entries as $n \to \infty$, and numerically examined their tightness. This asymptotic is different from that in \eqn{decay rate Jackson 1}, so we can not confirm its tightness by \eqn{decay rate Jackson 1}, but conjecture it to be true since similar asymptotics are known for a two dimensional semimartingale reflecting Brownian motion (see \cite{AvraDaiHase2001,DaiMiya2013}).
}\end{remark}

  See \fig{domain Jackson} to see how the domain looks like.
\begin{figure}[h]
 	\centering
	\includegraphics[height=6.7cm]{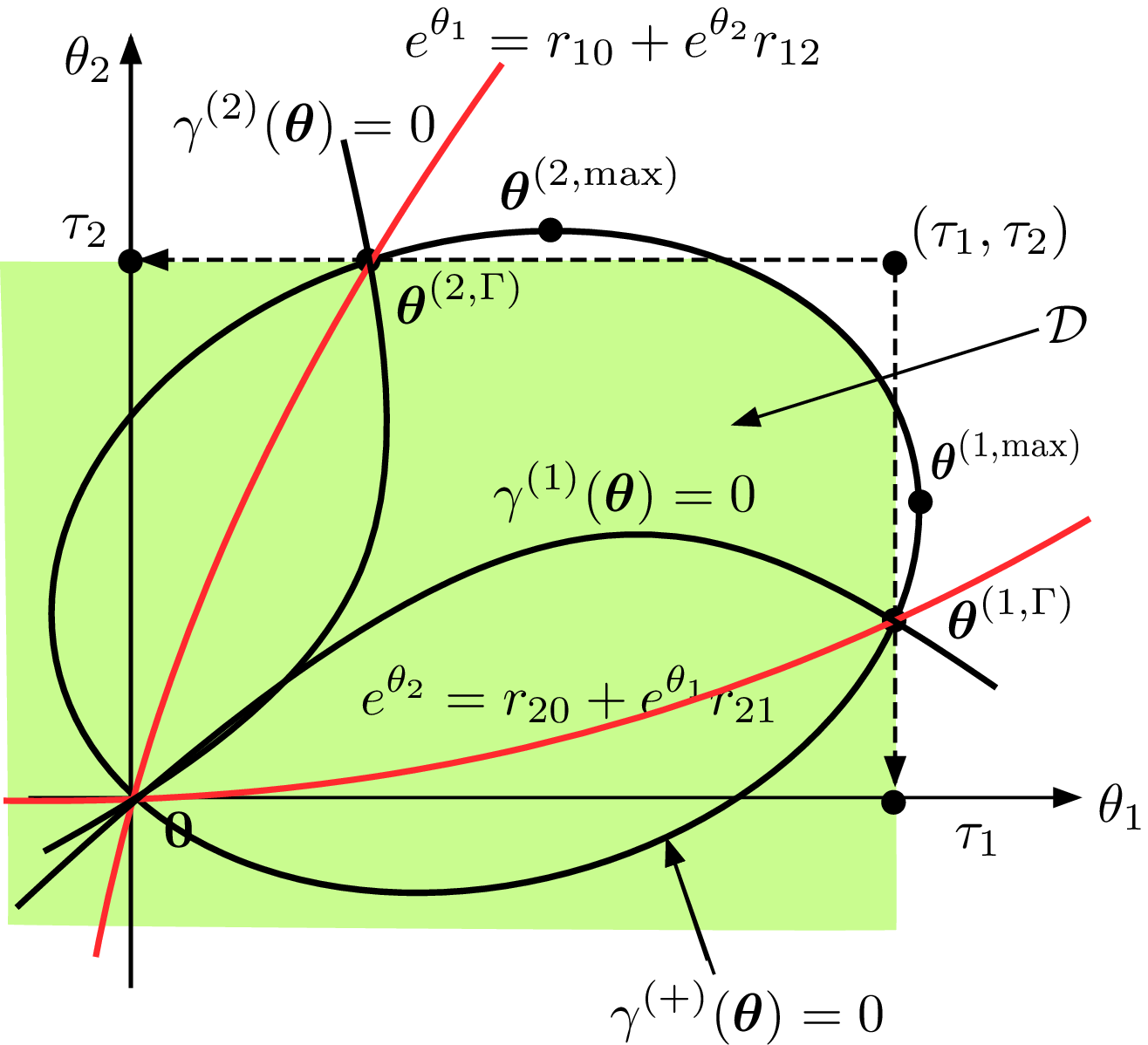} $\;$
	\includegraphics[height=6.7cm]{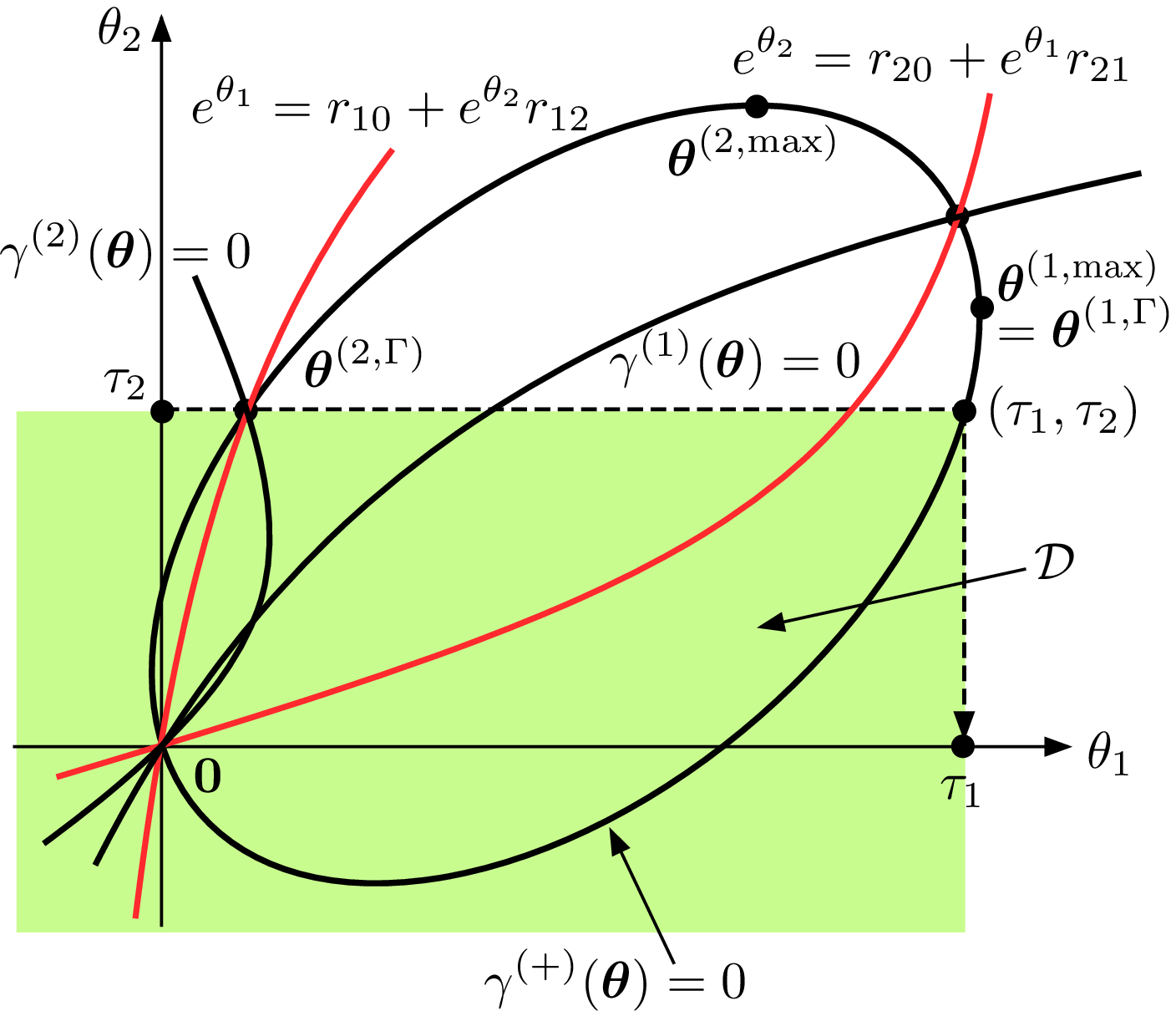} 
	\caption{The domain $\sr{D}$ for the two node generalized Jackson network}
	\label{fig:domain Jackson}
\end{figure}

\subsection{Primitive data and matrix moment generation functions}
\label{sect:matrix moment}

In this section, we describe transition rate matrices and their moment generating functions in terms of the primitive data, $T_{i}, U_{i}, S_{i}, \vc{\beta}_{i}$, of the generalized Jackson network, and prove \eqn{G-Jackson domain 1} and \eqn{G-Jackson domain 2}. They will be used to prove \thr{decay rate Jackson} in the next subsection.

To specify those matrices for the generalized Jackson network, we will use the Kronecker product $\otimes$ and sum $\oplus$, respectively, where $\oplus$ is defined for square matrices $A$ and $B$ as
\begin{align*}
  A \oplus B = A \otimes I_{2} + I_{1} \otimes B,
\end{align*}
where $I_{1}$ and $I_{2}$ are the identity matrices with the same sizes as $A$ and $B$, respectively. From this definition, it is easy to see that. if $A$ and $B$ have right eigenvectors $\vc{h}_{A}$ and $\vc{h}_{A}$ with eigenvalues $\gamma_{A}$ and $\gamma_{B}$, respectively, then
\begin{align}
\label{eqn:oplus eigenvector}
  (A \oplus B) (\vc{h}_{A} \otimes \vc{h}_{B}) = (\gamma_{A} + \gamma_{B}) (\vc{h}_{A} \otimes \vc{h}_{B}).
\end{align}
We also will use this computation.

For transitions around the origin, we let
\begin{align*}
 & \tilde{A}^{(0)}_{00} = {T}_{1} \oplus {T}_{2}, \quad \tilde{A}^{(0)}_{10} = {U}_{1} \otimes I_{2} \otimes \vc{\beta}_{1}, \quad \tilde{A}^{(0)}_{01} = I_{1} \otimes {U}_{2} \otimes \vc{\beta}_{2},\\
 & \tilde{A}^{(0)}_{(-1)0} = I_{1} \otimes I_{2} \otimes {D}_{1} \vc{1}, \quad \tilde{A}^{(0)}_{0(-1)} = I_{1} \otimes I_{2} \otimes {D}_{2} \vc{1},
\end{align*}
where other $\tilde{A}^{(0)}_{ij}$'s not specified above are all null matrices. This convention for null matrices is used for all transition matrices. Around $\sr{U}_{+0} \cup \sr{U}_{+1}$, that is, the 1st coordinate half axis except for the origin,
\begin{align*}
 &  \tilde{A}^{(1)}_{(-1)0} = I_{1} \otimes I_{2} \otimes (r_{10} {D}_{1}) \quad \tilde{A}^{(1)}_{00} = {T}_{1} \oplus {T}_{2} \oplus {S}_{1}, \quad \tilde{A}^{(1)}_{10} = {U}_{1} \otimes I_{2} \otimes I_{3},\\
 &  \tilde{A}^{(1)}_{(-1)1} = I_{1} \otimes I_{2} \otimes (r_{12} {D}_{1}) \otimes \vc{\beta}_{2} \quad \tilde{A}^{(1)}_{01} = I_{1} \otimes {U}_{2} \otimes I_{3} \otimes \vc{\beta}_{2},\\
 &  \tilde{A}^{(1)}_{0(-1)} = I_{1} \otimes I_{2} \otimes I_{3} \otimes (r_{20} {D}_{2} \vc{1}),\quad \tilde{A}^{(1)}_{1(-1)} = I_{1} \otimes I_{2} \otimes I_{3} \otimes (r_{21} {D}_{2} \vc{1}) . \end{align*}
Similarly, around $\sr{U}_{0+} \cup \sr{U}_{1+}$, that is, the 2nd coordinate half axis except for the origin,
\begin{align*}
 & \tilde{A}^{(2)}_{0(-1)} = I_{1} \otimes I_{2} \otimes (r_{20} {D}_{2}),\quad \tilde{A}^{(2)}_{00} = {T}_{1} \oplus {T}_{2} \oplus {S}_{2}, \quad \tilde{A}^{(2)}_{01} = I_{1} \otimes {U}_{2} \otimes I_{4}, \\
 & \tilde{A}^{(2)}_{1(-1)} = I_{1} \otimes I_{2} \otimes \vc{\beta}_{1} \otimes (r_{21} {D}_{2}), \quad \tilde{A}^{(2)}_{10} = {U}_{1} \otimes I_{2} \otimes \vc{\beta}_{1} \otimes I_{4},\\
 & \tilde{A}^{(2)}_{(-1)0} = I_{1} \otimes I_{2} \otimes (r_{10} {D}_{1} \vc{1}) \otimes I,\quad \tilde{A}^{(2)}_{(-1)1} = I_{1} \otimes I_{2} \otimes (r_{12} {D}_{1} \vc{1}) \otimes I_{4} . 
\end{align*}
For transitions within $\sr{U}_{+}$, that is, the interior,
\begin{align*}
 & \tilde{A}_{00} = {T}_{1} \oplus {T}_{2} \oplus {S}_{1} \oplus {S}_{2}, \quad \tilde{A}_{10} = {U}_{1} \otimes I_{2} \otimes I_{2} \otimes I_{3}, \quad \tilde{A}_{01} = I_{1} \otimes {U}_{2} \otimes I_{3} \otimes I_{4}, \\
 & \tilde{A}_{(-1)0} = I_{1} \otimes I_{2} \otimes (r_{10} {D}_{1}) \otimes  I_{4},\quad \tilde{A}_{(-1)1} = I_{1} \otimes I_{2} \otimes (r_{12} {D}_{1}) \otimes I_{4},\\
 & \tilde{A}_{0(-1)} = I_{1}  \otimes I_{2} \otimes I_{3} \otimes (r_{20} {D}_{2}), \quad \tilde{A}_{1(-1)} = I_{1}  \otimes I_{2} \otimes I_{3} \otimes (r_{21} {D}_{2}).
\end{align*}
Thus, we have
\begin{align*}
  \lefteqn{\tilde{A}_{**}(\vc{\theta}) = ({T}_{1} + e^{\theta_{1}} {U}_{1}) \oplus ({T}_{2} + e^{\theta_{2}} {U}_{2})} \hspace{8ex}\\
  & \oplus ({S}_{1} + (e^{-\theta_{1}} r_{10} + e^{-\theta_{1}+\theta_{2}} r_{12}) {D}_{1}) \oplus ({S}_{2} + (e^{-\theta_{2}} r_{20} + e^{\theta_{1}-\theta_{2}} r_{21}) {D}_{2}).
\end{align*}

Recall that $\gamma^{(ia)}(\theta_{i})$ is the Perron-Frobenius eigenvalue of ${T}_{i} + e^{\theta_{i}} {U}_{i}$. We denote its eigenvector by $\vc{h}^{(ia)}(\theta_{i})$. Similarly, we denote the Perron-Frobenius eigenvalues and vectors of ${S}_{1} + (e^{-\theta_{1}} r_{10} + e^{-\theta_{1}+\theta_{2}} r_{12}) {D}_{1}$ and ${S}_{2} + (e^{-\theta_{2}} r_{20} + e^{\theta_{1}-\theta_{2}} r_{21}) {D}_{2}$ by $\gamma^{(1d)}(\vc{\theta})$ and $\gamma^{(2d)}(\vc{\theta})$, and $\vc{h}^{(1d)}(\vc{\theta})$ and $\vc{h}^{(1d)}(\vc{\theta})$, respectively. That is, they satisfy
\begin{align}
\label{eqn:T1}
 & ({T}_{1} + e^{\theta_{1}} {U}_{1}) \vc{h}^{(1a)}(\theta_{1}) = \gamma^{(1a)}(\theta_{1}) \vc{h}^{(1a)}(\theta_{1}),\\
\label{eqn:T2}
 & ({T}_{2} + e^{\theta_{2}} {U}_{2}) \vc{h}^{(2a)}(\theta_{2}) = \gamma^{(2a)}(\theta_{2}) \vc{h}^{(2a)}(\theta_{2}),\\
\label{eqn:S1}
 & ({S}_{1} + (e^{-\theta_{1}} r_{10} + e^{-\theta_{1}+\theta_{2}} r_{12}) {D}_{1}) \vc{h}^{(1d)}(\vc{\theta}) = \gamma^{(1d)}(\vc{\theta}) \vc{h}^{(1d)}(\vc{\theta}),\\
\label{eqn:S2}
 & ({S}_{2} + (e^{-\theta_{2}} r_{20} + e^{\theta_{1}-\theta_{2}} r_{21}) {D}_{2}) \vc{h}^{(2d)}(\vc{\theta}) = \gamma^{(2d)}(\vc{\theta}) \vc{h}^{(2d)}(\vc{\theta}).
\end{align}

Thus, recalling $\gamma^{(+)}(\vc{\theta})$ and letting
\begin{align*}
  \vc{h}^{(+)}(\vc{\theta}) = \vc{h}^{(1a)}(\theta_{1}) \otimes \vc{h}^{(2a)}(\theta_{2}) \otimes \vc{h}^{(1d)}(\vc{\theta}) \otimes \vc{h}^{(2d)}(\vc{\theta}),
\end{align*}
we have, by repeatedly applying \eqn{oplus eigenvector},
\begin{align*}
  \tilde{A}_{**}(\vc{\theta}) \vc{h}^{(+)}(\vc{\theta}) = \gamma^{(+)}(\vc{\theta}) \vc{h}^{(+)}(\vc{\theta}) .
\end{align*}
Hence, recalling the definition \eqn{Gamma 2d+} of $\tilde{\Gamma}^{(2d)}$, we have \eqn{G-Jackson domain 1}.

We next note that $\gamma^{(id)}$ can also be obtained from the moment generating function $g_{i}$ of service time distribution at node $i$. For $i=1,2$, let
\begin{align*}
  t_{i}(\vc{\theta}) = e^{-\theta_{i}} r_{i0} + e^{-\theta_{i}+\theta_{3-i}} r_{i(3-i)},
\end{align*}
then it follows from \eqn{S1} and \eqn{S2} that
\begin{align*}
  t_{i}(\vc{\theta}) \br{\vc{\beta}_{i}, \vc{h}^{(id)}(\vc{\theta})} (-S_{i}\vc{1}) = (\gamma^{(id)}(\vc{\theta}) I - S_{i}) \vc{h}^{(id)}(\vc{\theta}),
\end{align*}
since $D_{i} \vc{h}^{(id)}(\vc{\theta}) = \br{\vc{\beta}_{i}, \vc{h}^{(id)}(\vc{\theta})} (-S_{i}\vc{1})$. Hence, premultiplying $(\gamma^{(id)}(\vc{\theta}) I - S_{i})^{-1}$, we have
\begin{align*}
  \vc{h}^{(id)}(\vc{\theta}) = t_{i}(\vc{\theta}) \br{\vc{\beta}_{i}, \vc{h}^{(id)}(\vc{\theta})} (\gamma^{(id)}(\vc{\theta}) I - S_{i})^{-1} (- S_{i}) \vc{1},
\end{align*}
Let us normalize $\vc{h}^{(id)}(\vc{\theta})$ in such a way that
\begin{align}
\label{eqn:normalize h}
  \br{\vc{\beta}_{i}, \vc{h}^{(id)}(\vc{\theta})} = t_{i}(\vc{\theta})^{-1},
\end{align}
then we have the following facts since $g_{i}$ is nondecreasing.

\begin{lemma}
\label{lem:service time 1}
For $i=1,2$, (a) under the normalization \eqn{normalize h},
\begin{align}
\label{eqn:h id1}
  \vc{h}^{(id)}(\vc{\theta}) = (\gamma^{(id)}(\vc{\theta}) I - S_{i})^{-1} (- S_{i}) \vc{1},
\end{align}
and therefore $g_{i}(-\gamma^{(id)}(\vc{\theta})) = t_{i}(\vc{\theta})^{-1}$, which yields $\gamma^{(id)}(\vc{\theta}) = - g_{i}^{-1}(t_{i}(\vc{\theta})^{-1})$,\\
(b) $\gamma^{(id)}(\vc{\theta}) \ge 0$ if and only if $t_{i}(\vc{\theta}) \ge 1$, which is equivalent to
\begin{align}
\label{eqn:routing condition}
  r_{i0} + e^{\theta_{3-i}} r_{i(3-i)} \ge e^{\theta_{i}}.
\end{align}
(c) If $U_{i} = (-T_{i} \vc{1}) \vc{\alpha}_{i}$ for probability vectors $\vc{\alpha}_{i}$, that is, the arrival process at node $i$ is the renewal process with interarrival distribution determined by the moment generating function:
\begin{align*}
  f_{i}(\theta_{i}) = \br{\vc{\alpha}_{i}, (-\theta_{i} I - T_{i})^{-1} (-T_{i} \vc{1})},
\end{align*}
then $\gamma^{(ia)}(\theta_{i}) = - f_{i}^{-1}(e^{-\theta_{i}})$.
\end{lemma}

\begin{remark} {\rm
\label{rem:service time 1}
  (a) and (c) are known (see, e.g., Proposition 2 of \cite{Taka1981} and Lemma 4.1 of \cite{FujiTakaMaki1998}).
}\end{remark}

\subsection{Proof of \thr{decay rate Jackson}}
\label{sect:proof of Jackson}

We first verify \ass{key assumption 3} for $i=1$ and $\tilde{c}^{(1)}_{1}(\vc{\theta}) = 0$, Namely, for $\vc{\theta} \in \dd{R}^{2}$ satisfying that $\gamma^{(+)}(\vc{\theta}) = 0$, that is,
\begin{align}
\label{eqn:g+ 0}
  \gamma^{(1a)}(\theta_{1}) + \gamma^{(2a)}(\theta_{2}) + \gamma^{(1d)}(\vc{\theta}) +\gamma^{(2d)}(\vc{\theta}) = 0 ,
\end{align}
we show that there are some $\tilde{c}^{(1)}_{0}(\vc{\theta})$ and $\vc{h}^{(01)}(\vc{\theta}) > \vc{0}$ such that
\begin{align}
\label{eqn:G-Jackson S1}
 & \tilde{A}^{(1)}_{*0}(\theta_{1}) \vc{h}^{(01)}(\vc{\theta}) + e^{\theta_{2}} \tilde{A}^{(1)}_{*1}(\theta_{1}) \vc{h}^{(+)}(\vc{\theta}) = \tilde{c}^{(1)}_{0}(\vc{\theta}) \vc{h}^{(01)}(\vc{\theta}),\\
\label{eqn:G-Jackson S2}
 & e^{-\theta_{2}} \tilde{A}^{(1)}_{*(-1)}(\theta_{1}) \vc{h}^{(01)}(\vc{\theta}) + \tilde{A}_{*0}(\theta_{1}) \vc{h}^{(+)}(\vc{\theta}) + e^{\theta_{2}} \tilde{A}_{*1}(\theta_{1}) \vc{h}^{(+)}(\vc{\theta}) = \vc{0}, \qquad
\end{align}
where
\begin{align*}
 & \tilde{A}_{*+}(\vc{\theta}) = ({T}_{1} + e^{\theta_{1}} {U}_{1}) \oplus ({T}_{2} + e^{\theta_{2}} {U}_{2}) \oplus ({S}_{1} + (e^{-\theta_{1}} r_{10} + e^{-\theta_{1}+\theta_{2}} r_{12}) {D}_{1}) \oplus {S}_{2},\\
 & \tilde{A}_{*(-1)}(\theta_{1}) = I_{1} \otimes I_{2} \otimes I_{3} \otimes ((r_{20} + r_{21} e^{\theta_{1}}) {D}_{2}),\\
 & \tilde{A}^{(1)}_{*(-1)}(\theta_{1}) = I_{1} \otimes I_{2} \otimes I_{3} \otimes ((r_{20} + r_{21} e^{\theta_{1}}) {D}_{2} \vc{1}),\\
 & \tilde{A}^{(1)}_{*0}(\theta_{1}) = ({T}_{1} + e^{\theta_{1}} {U}_{1}) \oplus {T}_{2} \oplus ({S}_{1} + r_{10} e^{-\theta_{1}} {D}_{1}),\\
 & \tilde{A}^{(1)}_{*1}(\theta_{1}) = I_{1} \otimes ({U}_{2} \oplus (r_{12} e^{-\theta_{1}} {D}_{1})) \otimes \vc{\beta}_{2}.
\end{align*}
We further require the non-singularity condition:
\begin{align}
\label{eqn:non-singular 1}
  \tilde{A}^{(1)}_{*0}(\theta_{1}) \vc{h}^{(01)}(\vc{\theta}) < \vc{0}.
\end{align}
From \eqn{G-Jackson S1}, this holds if $\tilde{c}_{0}^{(1)}(\vc{\theta}) \le 0$.

Since $\tilde{A}_{**}(\vc{\theta}) \vc{h}^{(+)}(\vc{\theta}) = \vc{0}$ by \eqn{g+ 0}, \eqn{G-Jackson S2} is equivalent to
\begin{align}
\label{eqn:G-Jackson S3}
 & \tilde{A}^{(1)}_{*(-1)}(\theta_{1}) \vc{h}^{(01)}(\vc{\theta}) - \tilde{A}_{*(-1)}(\theta_{1}) \vc{h}^{(+)}(\vc{\theta}) = \vc{0}.
\end{align}
Note that $ \tilde{A}_{*(-1)}(\theta_{1})$ and $ \tilde{A}^{(1)}_{*(-1)}(\theta_{1})$ have a similar form, so we let
\begin{align*}
  \vc{h}^{(1)}(\vc{\theta}) = \vc{h}^{(1a)}(\theta_{1}) \otimes \vc{h}^{(2a)}(\theta_{2}) \otimes \vc{h}^{(1d)}(\vc{\theta}),
\end{align*}
and guess that, for some scalar $a(\vc{\theta})$,
\begin{align*}
  \vc{h}^{(01)}(\vc{\theta}) = a(\vc{\theta}) \vc{h}^{(1)}(\vc{\theta}).
\end{align*}

We first verify \eqn{G-Jackson S1}. Since 
\begin{align*}
  \lefteqn{\tilde{A}^{(1)}_{*0}(\theta_{1}) \vc{h}^{(01)}(\vc{\theta}) = a(\vc{\theta}) \big( \gamma^{(1a)}(\theta_{1}) \vc{h}^{(1)}(\vc{\theta}) + \vc{h}^{(1a)}(\theta_{1}) \otimes (T_{2} \vc{h}^{(2a)}) \otimes \vc{h}^{(1d)}(\vc{\theta})}\\
  & \hspace{25ex} + \vc{h}^{(1a)}(\theta_{1}) \otimes \vc{h}^{(2a)}(\theta_{2}) \otimes ({S}_{1} + r_{10} e^{-\theta_{1}} {D}_{1}) \vc{h}^{(1d)}(\vc{\theta}) \big),\\
  \lefteqn{e^{\theta_{2}} \tilde{A}^{(1)}_{*1}(\theta_{1}) \vc{h}^{(+)}(\vc{\theta})  = \vc{h}^{(1a)}(\theta_{1}) \otimes \left(e^{\theta_{2}} U_{2} \vc{h}^{(2a)}(\theta_{2}) \otimes \vc{h}^{(1d)}(\vc{\theta})\right.}\\
  & \hspace{25ex} + \left. \vc{h}^{(2a)}(\theta_{2}) \otimes (r_{12} e^{-\theta_{1}+\theta_{2}} D_{1} \vc{h}^{(1d)}(\vc{\theta}))\right) \br{\vc{\beta}_{2}, \vc{h}^{(2d)}(\vc{\theta})}, 
\end{align*}
we choose $a(\vc{\theta}) = \br{\vc{\beta}_{2}, \vc{h}^{(2d)}(\vc{\theta})}$, which is $t_{2}(\vc{\theta})^{-1}$ by \eqn{normalize h}, then
\begin{align*}
  \tilde{A}^{(1)}_{*0}(\theta_{1}) \vc{h}^{(01)}(\vc{\theta}) + e^{\theta_{2}} \tilde{A}^{(1)}_{*1}(\theta_{1}) \vc{h}^{(+)}(\vc{\theta})
 & = \big( \gamma^{(1a)}(\theta_{1}) + \gamma^{(2a)}(\theta_{2}) + \gamma^{(1d)}(\vc{\theta}) \big) \vc{h}^{(01)}(\vc{\theta}).
\end{align*}
Hence, we have \eqn{G-Jackson S1} with
\begin{align*}
  \tilde{c}^{(1)}_{0}(\vc{\theta}) = \gamma^{(1a)}(\theta_{1}) + \gamma^{(2a)}(\theta_{2}) + \gamma^{(1d)}(\vc{\theta}) (\equiv \gamma^{(1)}(\vc{\theta})).
\end{align*}

We next consider \eqn{G-Jackson S3}. Recall that $D_{2} = (-S_{2}\vc{1}) \vc{\beta}_{2}$. Since
\begin{align*}
  & e^{-\theta_{2}} \tilde{A}^{(1)}_{*(-1)}(\theta_{1}) \vc{h}^{(01)}(\vc{\theta}) = \vc{h}^{(01)}(\vc{\theta}) \otimes ((r_{20} e^{-\theta_{2}} + r_{21} e^{\theta_{1} - \theta_{2}}) \br{\vc{\beta}_{2}, \vc{h}^{(2d)}(\vc{\theta})} {D}_{2} \vc{1})),\\
  & e^{-\theta_{2}} \tilde{A}_{*(-1)}(\theta_{1}) \vc{h}^{(+)}(\vc{\theta}) = \vc{h}^{(01)}(\vc{\theta}) \otimes ((r_{20} e^{-\theta_{2}} + r_{21} e^{\theta_{1} - \theta_{2}}) {D}_{2} \vc{h}^{(2d)}(\vc{\theta})),\\
 & D_{2} \vc{h}^{(2d)}(\vc{\theta}) = (-S_{2}\vc{1}) \br{\vc{\beta}_{2}, \vc{h}^{(2d)}(\vc{\theta})} = \br{\vc{\beta}_{2}, \vc{h}^{(2d)}(\vc{\theta})} (-S_{2}\vc{1}) = \br{\vc{\beta}_{2}, \vc{h}^{(2d)}(\vc{\theta})} D_{2} \vc{1},
\end{align*}
we have \eqn{G-Jackson S3}. Thus, we have verified \ass{key assumption 3}, and therefore $\vc{\theta} \in \tilde{\Gamma}^{(2e)}_{1+}$ is equivalent to $\gamma^{(+)}(\vc{\theta}) = 0$ and $\gamma^{(1)}(\vc{\theta}) \le 0$. Because arguments are symmetric for nodes 1 and 2, we can get similar results for node 2. Thus, \thr{decay rate Jackson} follows from \thr{decay rate 2} because of \rem{Gamma 2ei}.

We finally note that, for $\vc{\theta} \in \dd{R}^{2}$ satisfying $\gamma^{(+)}(\vc{\theta}) = 0$, $\gamma^{(3-i)}(\vc{\theta}) \le 0$ is equivalent to $\gamma^{(id)}(\vc{\theta}) \ge 0$, which further is equivalent to \eqn{routing condition}. Hence, we have verified the claim in \rem{decay rate Jackson 1}.

\section{Concluding remarks}
\setnewcounter
\label{sect:concluding}

We have studied the existence of a superharmonic vector for a nonnegative matrix with QBD block structure. We saw how this existence is useful for studying the tail asymptotics of the stationary distribution of a Markov modulated two dimensional reflecting random walk, called the 2d-QBD process. We have assumed that all blocks of the nonnegative matrix are finite dimensional. This is a crucial assumption, but we need to remove it for studying a higher dimensional reflecting random walk. This is a challenging problem. Probably, further structure is needed for the background process. For example, we may assume that each block matrix has again QBD block structure, which is satisfied by a reflecting random walk in any number of dimensions with Markov modulation. We think research in this direction would be useful.

Another issue is about the tail asymptotics for a generalized Jackson network. We have considered the two node case. In this case, the tail decay rates are determined by time average cumulant moment generating functions, $\gamma^{(a)}_{i}$ and $\gamma^{(d)}_{i}$ by \thr{decay rate Jackson}. This suggests that more general arrival processes and more general routing mechanisms may lead to the decay rates in the same way. Some related issues have been recently considered for a single server queue in Section 2.4 of \cite{Miya2015}, but the network case has not yet been studied. So, it is also an open problem.

In a similar fashion, we may be able to consider a generalized Jackson network with more than two nodes. To make the problem specific, let us consider the $k$ node cases for $k \ge 2$. Let $K = \{1,2,\ldots, k\}$, and let
\begin{align*}
  \gamma^{(+)}(\vc{\theta}) = \sum_{j=1}^{k} (\gamma^{(ja)}(\theta_{j}) + \gamma^{(jd)}(\vc{\theta})).
\end{align*}
Then, the sets similar to $\tilde{\Gamma}^{2e}_{i+}$ for $k=2$ may be indexed by a non-empty subset $A$ of $K$, and given by
\begin{align*}
  \tilde{\Gamma}^{ke}_{A+} = \left\{ \vc{\theta} \in \dd{R}^{k}; \gamma^{(+)}(\vc{\theta}) = 0, \gamma^{(i)}(\vc{\theta}) \ge 0, \forall i \in K \setminus A \right\}.
\end{align*}
These together with $\tilde{\Gamma}^{kd}_{+} = \left\{ \vc{\theta} \in \dd{R}^{k}; \gamma^{(+)}(\vc{\theta}) \le 0 \right\}$ would play the same role as in the two dimensional case. That is, they would characterize the tail decay rates of the stationary distribution. We may generate those sets from
\begin{align*}
  \tilde{\Gamma}^{ke}_{i+} = \left\{ \vc{\theta} \in \dd{R}^{k}; \gamma^{(+)}(\vc{\theta}) = 0, \gamma^{(i)}(\vc{\theta}) \ge 0 \right\}, \qquad i \in K.
\end{align*}
Thus, the characterization may be much simpler than that for a general $k$ dimensional random walk with Markov modulation. However, we do not know how to derive the decay rates from them for $k \ge 3$ except for tandem type models under some simple situations (e.g., see \cite{BertPascTsit1998,Chan1995}). This remains as a very challenging problem (e.g., see Section 6 of \cite{Miya2011}).

We finally remark on the continuity of the decay rate for a sequence of the two node generalized Jackson networks which weakly converges to the two dimensional SRBM in heavy traffic. Under suitable scaling and appropriate conditions, such convergence is known not only for their processes but for their stationary distributions (see, e.g, \cite{BudhLee2009,GamaZeev2006}). Since the tail decay rates are known for this SRBM (see \cite{DaiMiya2011}), we can check whether the decay rate also converges to that of the SRBM. This topic is considered in \cite{Miya2015}.

\bigskip

\appendix

\noindent{\bf \Large Appendix}

\section{\large Proof of \lem{Gamma+}}
\label{app:Gamma+}

  (a) For sufficiency, we assume that $\Gamma^{(1d)}_{+} \ne \emptyset$, that is, there is a $\theta \in \Gamma^{(1d)}_{+}$. For this $\theta$, let $\vc{y} = (\vc{y}_{1}, \vc{y}_{2}, \ldots)^{\rs{t}}$ for $\vc{y}_{n} = e^{\theta n} \vc{h}^{A_{*}}(\theta)$, then it is easy to see that $K_{+} \vc{y} \le \vc{y}$, and therefore $c_{p}(K_{+}) \ge 1$. For necessity, we use the same idea as in the proof of Theorem 3.1 of \cite{KobaMiyaZhao2010}. Assume the contrary that $\Gamma^{(1d)}_{+} = \emptyset$, which is equivalent to $\min_{\theta} \gamma^{(A_{*})}(\theta) > 1$, when $c_{p}(K_{+}) \ge 1$ holds, and leads a contradiction. By this supposition and the convexity of $\gamma^{(A_{*})}(\theta)$, there is a $\theta_{0}$ such that $\gamma^{(A_{*})}(\theta_{0}) > 1$ and $(\gamma^{(A_{*})})'(\theta_{0}) = 0$. 

We next define the stochastic matrix $\widehat{P}^{(\theta_{0})}$ whose $(k, \ell)$ block matrix is given by
\begin{align}
\label{eqn:q}
  \widehat{P}^{(\theta_{0})}_{k \ell} = \left\{\begin{array}{ll}
  [\gamma^{(A_{*})}(\theta_{0})]^{-1} \widehat{A}^{(\theta_{0})}_{k - \ell}, \qquad & k \ge 1, \ell \ge 0, \; |k-\ell| \le 1, \vspace{1ex}\\
  I, & k=\ell=0, \vspace{1ex}\\
  0, & \mbox{otherwise}.
  \end{array} \right.
\end{align}
Since $\gamma^{(A_{*})}(\theta_{0}) > 1 > 0$, this stochastic matrix $\widehat{P}^{(\theta_{0})}$ is well defined. The Markov chain with this transition matrix is a Markov modulated random walk on $\dd{Z}_{+}$ with an absorbing state at block $0$, where $[\gamma^{(A_{*})}(\theta_{0})]^{-1} \widehat{A}^{(\theta_{0})}$ is the transition probability matrix of the background process as long as the random walk part is away from the origin. Denote its stationary distribution by $\widehat{\vc{\pi}}^{(\theta_{0})}$. That is,
\begin{align*}
  \widehat{\vc{\pi}}^{(\theta_{0})} \widehat{A}^{(\theta_{0})} = \gamma^{(A_{*})}(\theta_{0}) \widehat{\vc{\pi}}^{(\theta_{0})},
\end{align*}
  which is equivalent to
\begin{align*}
  \widehat{\vc{\pi}}^{(\theta_{0})} \Delta^{-1}_{\vc{h}^{(A_{*})}(\theta_{0})} A_{*}(\theta_{0}) = \gamma^{(A_{*})}(\theta_{0}) \widehat{\vc{\pi}}^{(\theta_{0})} \Delta^{-1}_{\vc{h}^{(A_{*})}(\theta_{0})}.
\end{align*}
  Taking the derivatives of $A_{*}(\theta) \vc{h}^{(A_{*})}(\theta) = \gamma^{(A_{*})}(\theta) \vc{h}^{(A_{*})}(\theta)$ at $\theta = \theta_{0}$, we have
\begin{align*}
  A_{*}'(\theta_{0}) \vc{h}^{(A_{*})}(\theta_{0}) + A_{*}(\theta_{0}) (\vc{h}^{(A_{*})})'(\theta_{0}) = \gamma^{(A_{*})}(\theta_{0}) (\vc{h}^{(A_{*})})'(\theta_{0}) + (\gamma^{(A_{*})})'(\theta_{0}) \vc{h}^{(A_{*})}(\theta_{0}).
\end{align*}
Multiplying by $\widehat{\vc{\pi}}^{(\theta_{0})} \Delta^{-1}_{\vc{h}^{(A_{*})}(\theta_{0})}$ from the left, we have
\begin{align}
\label{eqn:drift 1}
  \widehat{\vc{\pi}}^{(\theta_{0})} \Delta^{-1}_{\vc{h}^{(A_{*})}(\theta_{0})} (- e^{-\theta_{0}} A_{-1} + e^{\theta_{0}} A_{1}) \Delta_{\vc{h}^{(A_{*})}(\theta_{0})} \vc{1} = (\gamma^{(A_{*})})'(\theta_{0}).
\end{align}
The left side of this equation is the mean drift of the Markov modulated random walk. Since $(\gamma^{\rt{A}_{*}})'(\theta_{0}) = 0$, this drift vanishes, and therefore the random walk hits one level below with probability one.

Since we have assumed that $c_{p}(K_{+}) \ge 1$, $K_{+}$ has a superharmonic $\vc{y}^{+}$. Let $\vc{y}^{+}= (\vc{y}^{+}_{0}, \vc{y}^{+}_{1}, \ldots)^{\rs{t}}$ be a superharmonic vector of $K_{+}$, and let
\begin{align}
\label{eqn:tilde y 1}
  \widehat{\vc{y}}^{(\theta_{0})}_{n} = e^{-\theta_{0} n} \Delta_{\vc{h}^{(A_{*})}(\theta_{0})}^{-1} \vc{y}^{+}_{n}, \qquad n \ge 0.
\end{align}
We then rewrite \eqn{y +} as
\begin{align}
\label{eqn:harmonic tilde 1}
  \gamma^{(A_{*})}(\theta_{0}) \sum_{\ell=0, \pm 1} P^{(\theta_{0})}_{n (n+\ell)} \widehat{\vc{y}}^{(\theta_{0})}_{n+\ell} \le \widehat{\vc{y}}^{(\theta_{0})}_{n}, \qquad n \ge 1.
\end{align}
Let $f_{(n,i)0}^{(\ell)}(\theta_{0})$ be the probability that the Markov chain with transition matrix $\widehat{P}^{(\theta_{0})}$ is absorbed at block $0$ at time $\ell$ given that it starts at state $(n,i)$, and denote the vector whose $i$-th entry is $f_{(n,i)0}^{(\ell)}(\theta_{0})$ by $\vc{f}^{(\ell)}_{n0}$. Define its generating function as
\begin{align}
\label{eqn:f*}
    \vc{f}_{n0}^{(*)}(v,\theta_{0}) = \sum_{\ell=1}^{\infty} v^{\ell} \vc{f}_{n0}^{(\ell)}(\theta_{0}).
\end{align}

Assume that $\widehat{\vc{y}}_{1}(\theta_{0}) \ge \vc{1}$, which is equivalent to
\begin{align}
\label{eqn:bound h}
  \vc{y}^{+}_{1} \ge \vc{h}^{(A_{*})}(\theta_{0}).
\end{align}
  We can always take $\vc{h}^{(A_{*})}(\theta_{0})$ satisfying this condition because the vectors are finite dimensional and constant multiplication does not change the eigenvalue. Since $\widehat{\vc{y}}^{(\theta_{0})}$ is $\gamma(\theta_{0})$-superharmonic by \eqn{harmonic tilde 1}, it follows from the right-invariant version of Lemma 4.1 of Vere-Jones~\cite{Vere1967} that
\begin{align}
\label{eqn:y tilde}
  \widehat{\vc{y}}^{(\theta_{0})}_{n} \ge \vc{f}_{n0}^{(*)}(\gamma^{(A_{*})}(\theta_{0}), \theta_{0}) , \qquad n \ge 1.
\end{align}
However, the random walk is null recurrent. Hence, $\vc{f}_{n0}^{*}(1; \theta_{0}) = \vc{1}$. This implies that $\vc{f}_{n0}^{*}(\gamma^{(A_{*})}(\theta_{0}); \theta_{0}) = \vc{\infty}$ because $(\vc{f}_{n0}^{(*)})'(1, \theta_{0}) = \vc{\infty}$ and $\gamma^{(A_{*})}(\theta_{0}) > 1$, which implies that $\widehat{\vc{y}}^{(\theta_{0})}_{n} = \infty$. This and \eqn{bound h} conclude that $\vc{y}^{+} = \infty$, which contradicts the fact that $\vc{y}^{+}$ is superharmonic for $K_{+}$. Thus, we must have that $\min_{\theta} \gamma^{(A_{*})}(\theta) \le 1$.\\
(b) It follows from (a) that $c_{p}(u K_{+}) \ge 1$ if and only if $u \gamma^{A_{*}}(\theta) \le 1$ for some $\theta \in \dd{R}$. By \eqn{cp harmonic 1}, $c_{p}(K_{+}) \ge u$ if and only if $c_{p}(uK_{+}) \ge 1$. Hence,
\begin{align*}
  c_{p}(K_{+}) = \sup \{u \ge 0; \exists \theta \in \dd{R}, u \gamma^{A_{*}}(\theta) \le 1\} = (\min_{\theta} \gamma^{(A_{*})}(\theta))^{-1}.
\end{align*}
This proves (b). We remark that the finiteness of $m$ is crucial for \eqn{bound h} to hold.

\section{\large Proof of \lem{Gamma convex 1}}
\label{app:Gamma convex 1}

  Since $\Gamma^{(1d)}_{0+}$ is a subset of $\Gamma^{(1d)}_{+} \cap \Gamma^{(1d)}_{0}$, it is bounded. For the convexity, we apply the same method that was used to prove Lemma 3.7 of \cite{MiyaZwar2012}. For $\theta_{1}, \theta_{2} \in \Gamma^{(1d)}_{0+}$, there exist positive vectors $\vc{h}^{(1)}(\theta)$ and $\vc{h}^{(2)}(\theta)$ such that, for $i=1,2$,
\begin{align*}
  A_{*}(\theta_{i}) \vc{h}^{(i)}(\theta) \le \vc{h}^{(i)}(\theta), \qquad C_{*}(\theta_{i}) \vc{h}^{(i)}(\theta) \le \vc{h}^{(i)}(\theta) .
\end{align*}
Choose an arbitrary number $\lambda \in (0,1)$. Let $\vc{g}$ be the vector whose $j$-th entry $g_{j}$ is given by
\begin{align*}
  g_{j} = (h^{(1)}_{j})^{\lambda} (h^{(2)}_{j})^{1-\lambda}, \qquad j = 1,2,\ldots,m.
\end{align*}
Then, using H\"{o}lder's inequality similarly to the proof of Lemma 3.7 of \cite{MiyaZwar2012}, we can show that
\begin{align*}
  A_{*}(\lambda \theta_{1} + (1-\lambda)\theta_{2}) \vc{g} \le \vc{g}, \qquad C_{*}(\lambda \theta_{1} + (1-\lambda)\theta_{2}) \vc{g} \le \vc{g} .
\end{align*}
This proves that $\lambda \theta_{1} + (1-\lambda)\theta_{2} \in \Gamma^{(1d)}_{0+}$. Thus, $\Gamma^{(1d)}_{0+}$ is a convex set, and therefore it is a finite interval.

It remains to prove that $\Gamma^{(1d)}_{0+}$ is a closed set. To see this, let $\theta_{n}$ be an increasing sequence converging to $\theta_{\max}$. Then, we can find $\vc{h}_{n}$ for each $\theta_{n}$ such that \eqn{A 1} and \eqn{C 1} hold for $\vc{h} = \vc{h}_{n}$ and $\theta = \theta_{n}$ and it is normalized so that $\vc{h}_{n}^{\rs{t}} \vc{1} = 1$, where $\vc{1}$ is the column vector whose entries are all units. Since $\vc{h}_{n}$ is normalized, we can further find a subsequence of $\{\vc{h}_{n}\}$ which converges to some finite $\vc{h}_{\infty} \ge 0$ as $n \to \infty$. Since $\theta_{n}$ converges to $\theta_{\max}$ as $n \to \infty$, we have \eqn{A 1} and \eqn{C 1} for $\vc{h}_{\infty}$ and $\theta_{\max}$, which in turn imply that $\vc{h}_{\infty} > 0$ by the irreducibility of $A_{*}(\theta)$. Hence, $\theta_{\max} \in \Gamma^{(1d)}_{0+}$. Similarly, we can prove $\theta_{\min} \in \Gamma^{(1d)}_{0+}$. Thus, $\Gamma^{(1d)}_{0+} = [\theta_{\min}, \theta_{\max}]$.

\section{\large A counter example}
\label{app:counter example}

We produce an example such that $A_{1} G_{-} \ne e^{\theta} A_{1}$ for any $\theta \ne 0$ for $m=2$. For $p,q,r,s >0$ such that $p+q+r < 1$, $2p + q < 1$ and $s < 1$, define two dimensional matrices $A_{i}$ as
\begin{align*}
  A_{-1} = \left(\begin{array}{cc}r & 0 \\s & 0\end{array}\right), \quad
  A_{0} = \left(\begin{array}{cc}0 & 1-(p+q+r) \\s & 1-s\end{array}\right), \quad
  A_{1} = \left(\begin{array}{cc}p & q \\0 & 0\end{array}\right). \quad
\end{align*}
Since $A \equiv A_{-1} + A_{0} + A_{1}$ has the stationary measure $(s, 1-(p+r))$, the Markov additive process with kernel $\{A_{i}; i=0, \pm 1\}$ has a negative drift by the condition that $2p + q < 1$. Hence $G_{-}$ must be stochastic. Furthermore, the background state must be $1$ after the level is one down because the second column of $A_{-1}$ vanishes. Hence,
\begin{align*}
  G_{-} = \left(\begin{array}{cc}1 & 0 \\1 & 0\end{array}\right),
\end{align*}
and therefore
\begin{align*}
  A_{1} G_{-} = \left(\begin{array}{cc}p+q & 0 \\0 & 0\end{array}\right) \ne \left(\begin{array}{cc} e^{\theta} p & e^{\theta} q \\0 & 0\end{array}\right) = e^{\theta} A_{1}.
\end{align*}

\section{\large Proofs for the upper bounds}
\setnewcounter
\label{app:domain}

In this section, we prove \lem{domain 0} and \thr{upper bound 1}. To this end, we formulate the 2d-QBD process $\{Z_{n}\}$ as a Markov modulated reflecting random walk on the quarter lattice plane, and consider the stationary equation for this random walk using moment generating functions. Similarly to the one dimensional QBD processes in \sectn{nonnegative}, we first derive a canonical form for the stationary equations. This canonical form simplifies transitions around the boundary similar to the QBD case. 

\subsection{The stationary equation and inequality in canonical form}
\label{app:stationary equation}

Assume that $\{Z_{n}\}$ has the stationary distribution $\pi$. Let 
\begin{align*}
 & \varphi^{(w)}_{k}(\vc{\theta}) = \dd{E}(e^{\br{\vc{\theta},\vc{L}}}; \vc{L} \in \sr{U}_{w}, J=k), \qquad k \in \sr{V}_{+}, {w} = +, ++,\\
 & \varphi^{(w)}_{k}(\theta_{1}) = \dd{E}(e^{\theta_{1} L^{(1)}}; \vc{L} \in \sr{U}_{w}, J=k), \qquad k \in \sr{V}_{1}, w = 1, +0, +1,\\
 & \varphi^{(w)}_{k}(\theta_{2}) = \dd{E}(e^{\theta_{2} L^{(2)}}; \vc{L} \in \sr{U}_{w}, J=k), \qquad k \in \sr{V}_{2}, w = 2, 0+, 1+,
\end{align*}
where $Z \equiv (\vc{L}, J)$ is a random vector subject to $\pi$. We denote the vectors whose $k$-th entry is $\varphi^{(w)}_{k}(\vc{\theta})$ and $\varphi^{(w)}_{k}(\theta_{\ell})$ respectively by $\vc{\varphi}^{(w)}(\vc{\theta})$ and $\vc{\varphi}^{(w)}(\theta_{\ell})$. Similarly, $\vc{\pi}(i,j)$ denotes the vectors for the stationary probabilities $\pi(i,j,k)$.

\begin{lemma} {\rm
\label{lem:MGF se 1}
If $\varphi(\vc{\theta})$ is finite, then
\begin{align}
\label{eqn:se 6v}
  \lefteqn{\vc{\varphi}^{(++)}(\vc{\theta}) (I - {A}_{**}(\vc{\theta}))} \hspace{3ex}\nonumber\\
  & + e^{\theta_{2}} \vc{\varphi}^{(+1)} (\theta_{1}) \big(I - C^{(1)}_{**}(\vc{\theta}) \big) + e^{\theta_{1}} \vc{\varphi}^{(1+)} (\theta_{2}) \big(I -C^{(2)}_{**}(\vc{\theta}) \big) + \vc{\psi}^{(0)}(\vc{\theta}) = 0, \hspace{5ex}
\end{align}
  where
\begin{align*}
    \lefteqn{\vc{\psi}^{(0)}(\vc{\theta}) = e^{\theta_{1} + \theta_{2}} \vc{\pi}(1,1) (I - {A}^{(0)}_{++}(\vc{\theta}))} \hspace{4ex}\\
  & \quad -  e^{\theta_{1} + \theta_{2}} \big(\vc{\pi}(1,0) {A}^{(1)}_{+1}(\theta_{1}) + e^{\theta_{2}} \vc{\pi}(0,1) {A}^{(2)}_{1+}(\theta_{2}) + \vc{\pi}(\vc{0}) {A}^{(0)}_{11} \big)\\
  & \quad - e^{\theta_{2}} \vc{\psi}^{(1)}(\theta_{1}) (I - {A}^{(1)}_{*0}(\theta_{1}))^{-1} {A}^{(1)}_{*1}(\theta_{1}) - e^{\theta_{1}} \vc{\psi}^{(2)}(\theta_{2}) (I - {A}^{(2)}_{0*}(\theta_{2}))^{-1} {A}^{(2)}_{1*}(\theta_{2}) ,
\end{align*}
in which $\vc{\psi}^{(1)}(\theta_{1})$ and $\vc{\psi}^{(2)}(\theta_{2})$ are defined as
\begin{align*}
 & \vc{\psi}^{(1)}(\theta_{1}) = e^{\theta_{1}} \big(\vc{\pi}(1,1) {A}^{(1)}_{+(-1)}(\theta_{1}) + \vc{\pi}(1,0) ({A}^{(1)}_{+0}(\theta_{1}) - I) +\vc{\pi}(0,1) {A}^{(0)}_{1(-1)} + \vc{\pi}(\vc{0}) {A}^{(0)}_{10}\big),\\
 & \vc{\psi}^{(2)}(\theta_{2}) = e^{\theta_{2}} \big(\vc{\pi}(1,1) {A}^{(2)}_{(-1)+}(\theta_{2}) + \vc{\pi}(0,1) ({A}^{(2)}_{0+}(\theta_{2}) - I) +\vc{\pi}(1,0) {A}^{(0)}_{(-1)1} + \vc{\pi}(\vc{0}) {A}^{(0)}_{01}\big).
\end{align*}
}\end{lemma}

\begin{remark} {\rm
\label{rem:MGF se 1}
  \eqn{se 6v} reduces the stationary equations to those for the 2d-QBD whose random walk component is on $\sr{U}_{+}$. Obviously, all the complexities are pushed into $C^{(i)}_{**}(\vc{\theta})$ and $\vc{\psi}^{(0)}(\vc{\theta})$.
}\end{remark}

\begin{proof}
Assume that $Z_{0}$ has the stationary distribution $\pi$, then $Z_{n+1} \equiv (\vc{L}_{n+1}, J_{n+1})$ and $Z_{n} \equiv (\vc{L}_{n}, J_{n})$ have the same distribution $\pi$. Hence, recalling that $\dd{H} = \{0, 1, -1\}$ and taking the moment generating functions of \eqn{L n} for $J_{n} = k \in \sr{V}_{+}$, we have
\begin{align}
\label{eqn:se 1}
  \lefteqn{\varphi^{(+)}_{k}(\vc{\theta}) = \sum_{k' \in \sr{V}_{+}} \Big( \sum_{i,j \in \dd{H}} e^{i\theta_{1}+j\theta_{2}}\varphi^{(++)}_{k'}(\vc{\theta}) [{A}_{ij}]_{k'k} + \sum_{i,j \in \dd{H}_{+}} e^{i\theta_{1}+j\theta_{2}} e^{\theta_{1}+\theta_{2}} \pi(1,1,k') [{A}_{ij}]_{k'k}} \hspace{5ex} \nonumber\\
  & \qquad + \sum_{i \in \dd{H},j \in \dd{H}_{+}} e^{i\theta_{1}+j\theta_{2}} e^{\theta_{2}} \varphi^{(+1)}_{k'}(e^{\theta_{1}}) [{A}_{ij}]_{k'k} + \sum_{i \in \dd{H}_{+}, j \in \dd{H}} e^{i\theta_{1}+j\theta_{2}} e^{\theta_{1}} \varphi^{(1+)}_{k'}(e^{\theta_{2}}) [{A}_{ij}]_{k'k} \Big) \nonumber\\
  & \quad + \sum_{k' \in \sr{V}_{1}} \Big(\sum_{i \in \dd{H}} e^{i\theta_{1}+\theta_{2}} \varphi^{(+0)}_{k'}(\theta_{1}) [{A}^{(1)}_{i1}]_{k'k} + \sum_{i \in \dd{H}_{+}} e^{i\theta_{1}+\theta_{2}} e^{\theta_{1}} \pi(1,0,k') [{A}^{(1)}_{i1}]_{k'k} \Big) \nonumber\\
  & \quad  + \sum_{k' \in \sr{V}_{2}} \Big(\sum_{ j \in \dd{H}} e^{\theta_{1}+j\theta_{2}} \varphi^{(0+)}_{k'}(e^{\theta_{2}}) [{A}^{(2)}_{1j}]_{k'k} + \sum_{j \in \dd{H}_{+}} e^{\theta_{1} + j\theta_{2}} e^{\theta_{2}} \pi(0,1,k') [{A}^{(2)}_{1j}]_{k'k}\Big) \nonumber\\
  & \quad + \sum_{k' \in \sr{V}_{0}} e^{\theta_{1} + \theta_{2}} \pi(0,0,k') [{A}^{(0)}_{11}]_{k'k},
\end{align}
as long as $\varphi^{(+)}_{k}(\vc{\theta})$ and $\varphi^{(w)}_{k}(\vc{\theta})$ for $w=1,2$ exist and are finite for all $k$. Similarly, it follows from \eqn{L n} that, for $k \in \sr{V}_{1}$,
\begin{align}
\label{eqn:se 2}
  \lefteqn{\varphi^{(+0)}_{k}(\theta_{1}) + e^{\theta_{1}} \pi(1,0,k)}\nonumber\\
  & = \sum_{k' \in \sr{V}_{+}} \Big(\sum_{i \in \dd{H}} e^{i\theta_{1}-\theta_{2}} e^{\theta_{2}} \varphi^{(+1)}_{k'}(\theta_{1}) [{A}_{*(i(-1))}^{(1)}]_{k'k} + \sum_{i \in \dd{H}_{+}} e^{i\theta_{1}-\theta_{2}}e^{\theta_{1}+\theta_{2}} \pi(1,1,k') [{A}^{(1)}_{*(i(-1))}]_{k'k}\Big)\nonumber\\
  & \qquad + \sum_{k' \in \sr{V}_{1}} \Big( \sum_{i \in \dd{H}} e^{i\theta_{1}} \varphi^{(+0)}_{k'}(\theta_{1}) [{A}^{(1)}_{i0}]_{k'k} + \sum_{i \in \dd{H}_{+}} e^{\theta_{i}} e^{\theta_{1}} \pi(1,0,k') [{A}^{(1)}_{i0}]_{k'k} \Big) \nonumber\\
  & \qquad + \sum_{k' \in \sr{V}_{0}} \Big( e^{\theta_{1} - \theta_{2}} e^{\theta_{2}} \pi(0,1,k') [{A}^{(0)}_{1(-1)}]_{k'k} + e^{\theta_{1}} \pi(0,0,k') [{A}^{(0)}_{10}]_{k'k}\Big),
\end{align}
and $\varphi^{(0+)}_{k}(\theta_{2})$ for $k \in \sr{V}_{2}$ is symmetric to $\varphi^{(+0)}_{k}(\theta_{1})$  for $k \in \sr{V}_{1}$.

Recalling the matrix notation, ${A}_{+j}(\theta_{1})$, ${A}_{i+}(\theta_{2})$, ${A}^{(1)}_{+j}(\theta_{1})$, ${A}^{(2)}_{i+}(\theta_{2})$ and the vector notation $\vc{\varphi}^{(w)}(\vc{\theta}) $ for $w=+, ++$ and $\vc{\varphi}^{(w')}(\theta_{\ell})$ for $w' = 1, +0, +1$ and $\ell=1$ and for $w' = 2, 0+, 1+$ and $\ell=2$, the stationary equation \eqn{se 1} can be written as
\begin{align}
\label{eqn:se 1v}
  \vc{\varphi}^{(+)}(\vc{\theta}) = \, & \vc{\varphi}^{(++)}(\vc{\theta}) {A}_{**}(\vc{\theta}) + e^{\theta_{1} + \theta_{2}} \vc{\pi}(1,1) {A}_{++}(\vc{\theta}) \nonumber\\
  & + e^{\theta_{2}} (\vc{\varphi}^{(+1)}(\theta_{1}) {A}_{*+}(\vc{\theta}) + \vc{\varphi}^{(+0)}(\theta_{1}) {A}^{(1)}_{*1}(\theta_{1})) \nonumber\\
  & + e^{\theta_{1}} (\vc{\varphi}^{(1+)}(\theta_{2}) {A}_{+*}(\vc{\theta}) + \vc{\varphi}^{(0+)}(\theta_{2}) {A}^{(2)}_{1*}(\theta_{2})) \nonumber\\
  & + e^{\theta_{1}} \vc{\pi}(1,0) {A}^{(1)}_{+1}(\theta_{1}) + e^{\theta_{2}} \vc{\pi}(0,1) {A}^{(2)}_{1+}(\theta_{2}) + e^{\theta_{1} + \theta_{2}} \vc{\pi}(\vc{0}) {A}^{(0)}_{11},
\end{align}
as long as $\vc{\varphi}(\vc{\theta})$ is finite, where $\vc{\varphi}(\vc{\theta})$ is the $\sr{V}_{+}$-dimensional vector whose $k$-th entry is $\varphi_{k}(\vc{\theta})$. Similarly, \eqn{se 2} yields
\begin{align}
\label{eqn:se 2v}
  \vc{\varphi}^{(+0)}(\theta_{1}) + e^{\theta_{1}} \vc{\pi}(1,0) = & \vc{\varphi}^{(+1)}(\theta_{1}) {A}_{*(-1)}^{(1)}(\theta_{1}) + \vc{\varphi}^{(+0)}(\theta_{1}) {A}^{(1)}_{*0}(\theta_{1}) + e^{\theta_{1}} \vc{\pi}(1,1) {A}^{(1)}_{+(-1)}(\theta_{1}) \nonumber\\
  & + e^{\theta_{1}} \big(\vc{\pi}(1,0) {A}^{(1)}_{+0}(\theta_{1}) + \vc{\pi}(0,1) {A}^{(0)}_{1(-1)} + \vc{\pi}(\vc{0}) {A}^{(0)}_{10}\big), \qquad
\end{align}
and by symmetry,
\begin{align}
\label{eqn:se 3v}
  \vc{\varphi}^{(0+)}(\theta_{2}) + e^{\theta_{2}} \vc{\pi}(0,1) = & \vc{\varphi}^{(1+)}(\theta_{2}) {A}_{(-1)*}^{(2)}(\theta_{2}) + \vc{\varphi}^{(0+)} (\vc{\theta}) {A}^{(2)}_{0*}(\theta_{2}) + e^{\theta_{2}} \vc{\pi}(1,1) {A}^{(2)}_{(-1)+}(\theta_{2}) \nonumber\\
  & + e^{\theta_{2}} \big( \vc{\pi}(0,1) {A}^{(2)}_{0+}(\theta_{2}) + \vc{\pi}(1,0) {A}^{(0)}_{(-1)1} + \vc{\pi}(\vc{0}) {A}^{(0)}_{01} \big), \qquad
\end{align}
and 
\begin{align}
\label{eqn:se 4v}
  \vc{\pi}(\vc{0}) = \vc{\pi}(\vc{0}) {A}^{(0)}_{00} + \vc{\pi}(0,1) {A}^{(0)}_{0(-1)} + \vc{\pi}(1,0) {A}^{(0)}_{(-1)0} + \vc{\pi}(1,1) {A}^{(0)}_{(-1)(-1)}.
\end{align}
  Obviously, the equations \eqn{se 1v}--\eqn{se 4v} constitute the full set of the stationary equations, and therefore they uniquely determine the stationary distribution $\pi$ because of the irreducibility.

Assume that $I - {A}^{(1)}_{*0}(\theta_{1})$ and $I - {A}^{(2)}_{0*}(\theta_{2})$ are invertible and recall the definitions of $\vc{\psi}^{(1)}(\theta_{1})$ and $\vc{\psi}^{(2)}(\theta_{2})$, then we can get, from \eqn{se 2v} and \eqn{se 3v},
\begin{align}
\label{eqn:phi 1}
 & \vc{\varphi}^{(+0)}(\theta_{1}) = \big(\vc{\varphi}^{(+1)}(\theta_{1}) {A}_{*(-1)}^{(1)}(\theta_{1}) + \vc{\psi}^{(1)}(\theta_{1}) \big) (I - {A}^{(1)}_{*0}(\theta_{1}))^{-1},\\
\label{eqn:phi 2}
 & \vc{\varphi}^{(0+)}(\theta_{2}) = \big(\vc{\varphi}^{(1+)}(\theta_{2}) {A}_{(-1)*}^{(2)}(\theta_{2}) + \vc{\psi}^{(2)}(\theta_{2}) \big) (I - {A}^{(2)}_{0*}(\theta_{2}))^{-1}.
\end{align}

Substituting these $\vc{\varphi}^{(+0)}(\theta_{1})$ and $\vc{\varphi}^{(0+)}(\theta_{2})$ into \eqn{se 1v} and using the vector version of \eqn{se 1}: 
\begin{align*}
  \vc{\varphi}^{(+)}(\vc{\theta}) = \vc{\varphi}^{(++)}(\vc{\theta}) + e^{\theta_{2}} \vc{\varphi}^{(+1)}(\theta_{1}) + e^{\theta_{1}} \vc{\varphi}^{(1+)}(\theta_{2}) + e^{\theta_{1} + \theta_{2}} \vc{\pi}(1,1),
\end{align*}
we have
\begin{align*}
  & \vc{\varphi}^{(++)}(\vc{\theta}) (I - {A}_{**}(\vc{\theta})) \nonumber\\
  & + e^{\theta_{2}} \vc{\varphi}^{(+1)} (\theta_{1}) \big(I - ({A}_{*+}(\vc{\theta}) + {A}^{(1)}_{*(-1)}(\theta_{1}) (I - {A}^{(1)}_{*0}(\theta_{1}))^{-1} {A}^{(1)}_{*1}(\theta_{1})) \big) \nonumber\\
  & + e^{\theta_{1}} \vc{\varphi}^{(1+)} (\theta_{2}) \big(I - ({A}_{+*}(\vc{\theta}) + {A}^{(1)}_{(-1)*}(\theta_{2}) (I - {A}^{(2)}_{0*}(\theta_{2}))^{-1} {A}^{(2)}_{1*}(\theta_{2})) \big) + \vc{\psi}^{(0)}(\vc{\theta}) = 0.
\end{align*}
  Recalling the definitions of $\tilde{C}^{(i)}(\vc{\theta})$, this yields  \eqn{se 6v}.
\end{proof}

\subsection{Proof of \lem{domain 0}}
\label{app:domain 0}

In \lem{MGF se 1}, we have assumed that the moment generating functions for the stationary distribution are finite. We can not use this finiteness to prove \lem{domain 0}. Nevertheless, \lem {MGF se 1} is useful in the proof of \lem{domain 0}. This is because we will use its inequality version under some extra conditions in a similar way to Lemma 4 of \citet{KobaMiya2014}. A key idea is the following lemma.

\begin{lemma}
\label{lem:domain 1}
  Assume that $\vc{\theta} \in \dd{R}^{2}$ satisfies one of the following conditions.
\begin{mylist}{0}
\item [(a)] $\vc{\theta} \in \Gamma^{(2d)}_{+}$ and $|\vc{\varphi}^{w}(\vc{\theta})| < \infty$ for $w = +1, 1+$,
\item [(b)] $\vc{\theta} \in \Gamma^{(2d)}_{1+}$ and $|\vc{\varphi}^{(1+)} (\theta_{2})| < \infty$,
\item [(c)] $\vc{\theta} \in \Gamma^{(2d)}_{2+}$ and $|\vc{\varphi}^{(+1)} (\theta_{1})| < \infty$,
\end{mylist}
where $|\vc{a}| = \sum_{i} |a_{i}|$ for vector $\vc{a}$ whose $i$-th entry is $a_{i}$. Then,
\begin{align}
\label{eqn:stationary inequality 1}
  \lefteqn{\vc{\varphi}^{(++)}(\vc{\theta}) (I - {A}_{**}(\vc{\theta}))} \hspace{1ex}\nonumber\\
  & + e^{\theta_{2}} \vc{\varphi}^{(+1)} (\theta_{1}) \big(I - C^{(1)}_{**}(\vc{\theta}) \big) + e^{\theta_{1}} \vc{\varphi}^{(1+)} (\theta_{2}) \big(I -C^{(2)}_{**}(\vc{\theta}) \big) + \vc{\psi}^{(0)}(\vc{\theta}) \le 0, \hspace{5ex}
\end{align}
and therefore $\vc{\theta} \in \sr{D}$.
\end{lemma}

This lemma is slightly different from Lemma 4 of \cite{KobaMiya2014} because we here have background states. However, we can apply the exactly same arguments to derive \eqn{stationary inequality 1} from the one step transition relation \eqn{L n} for each fixed background state under the stationary distribution. Hence, $A _{**}(\vc{\theta}) \vc{h} < \vc{h}$ ($C^{(i)} _{**}(\vc{\theta}) \vc{h} < \vc{h}$) and
\begin{align*}
  \vc{\varphi}^{(++)}(\vc{\theta}) (I - A _{**}(\vc{\theta})) \vc{h} < \infty, \qquad (\vc{\varphi}^{(w(i))}(\vc{\theta}) (I - C^{(i)} _{**}(\vc{\theta})) \vc{h} < \infty),
\end{align*}
where $w(1) = +1$ and $w(2) = 1+$, implies that $|\vc{\varphi}^{(++)}(\vc{\theta})| < \infty$ ($|\vc{\varphi}^{(i+)}(\vc{\theta})| < \infty$, respectively). This completes the proof of \lem{domain 1}.

Similar to the proof of Theorem 1 of \cite{KobaMiya2014} (see Section 4.3 there), it is not hard to see that \lem{domain 1} yields \lem{domain 0}.

\subsection{The proof of \thr{upper bound 1}}
\label{app:upper bound 1}

For each $u, x > 0$, we have, for $u \vc{c} \in \Gamma_{\tau}^{(2d)}$,
\begin{align*}
  e^{ux} \dd{P}(\br{\vc{L}, \vc{c}} > x) \le \dd{E}(e^{\br{\vc{L}, u\vc{c}}} 1(\br{\vc{L}, \vc{c}} > x)) \le \varphi(u\vc{c}).
\end{align*}
Taking logarithm of both sides of this inequality, we get
\begin{align*}
  u + \frac 1x \log \dd{P}(\br{\vc{L}, \vc{c}} > x) \le \frac 1x \log \varphi(u\vc{c}).
\end{align*}
This yields
\begin{align*}
  \limsup_{x \to \infty} \frac 1x \log \dd{P}(\br{\vc{L}, \vc{c}} > x) \le - u
\end{align*}
as long as $u \vc{c} \in \Gamma_{\tau}^{(2d)}$, and therefore
\begin{align*}
  \lefteqn{\limsup_{x \to \infty} \frac 1x \log \dd{P}(\br{\vc{L}, \vc{c}} > x, J=k)} \nonumber \hspace{10ex}\\
  & \le  \limsup_{x \to \infty} \frac 1x \log \dd{P}(\br{\vc{L}, \vc{c}} > x) \le - \sup \{u \ge 0; u \vc{c} \in \Gamma_{\tau}^{(2d)}\}.
\end{align*}

\section{\large The proof of \lem{lower bound 2}}
\setnewcounter
\label{app:lower bound 2}

Similar to Lemma 4.2 of \cite{KobaMiya2014}, we can apply the permutation arguments in Lemma 5.6 of \cite{BoroMogu2001} twice. For this, we use a Markov modulated two dimensional random walk $\{(\vc{Y}_{n}, J_{n})\}$, whose increments $\vc{X}_{n+1} \equiv \vc{Y}_{n+1} - \vc{Y}_{n}$ have the following conditional distribution.
\begin{align*}
  \dd{P}(\vc{X}_{n+1} = \vc{u}, J_{n+1} = j|J_{n} = i) = [A_{\vc{u}}]_{ij}, \qquad \vc{u} \in \dd{H}^{2}, i,j \in \sr{V}_{+}.
\end{align*}
We here recall that $\dd{H} = \{0, \pm 1\}$. For each $n \ge 1$, we permute the Markov modulated random walk $\{(\vc{Y}_{\ell}, J_{\ell}), \ell=0, 1, \ldots, n\}$ starting with $\vc{Y}_{0} = \vc{0}$, and define $\{(\vc{Y}^{(m)}_{\ell}, J^{(m)}_{\ell}); \ell = 0, 1,\ldots, n\}$ for $m = 1,2,\ldots, n$ as 
\begin{align*}
 & \vc{Y}^{(m)}_{0} = \vc{0}, \vc{Y}^{(m)}_{1} = \vc{X}_{m+1}, \vc{Y}^{(m)}_{2} =  \vc{X}_{m+1} + \vc{X}_{m+2}, \ldots, \vc{Y}^{(m)}_{n-m} = \vc{X}_{m+1} + \ldots + \vc{X}_{n},\\
 & \vc{Y}^{(m)}_{n-m+1} = \vc{X}_{m+1} + \ldots + \vc{X}_{n} + \vc{X}_{1},\\
 & \ldots \\
 & \vc{Y}^{(m)}_{n} = \vc{X}_{m+1} + \ldots + \vc{X}_{n} + \vc{X}_{1} + \vc{X}_{2} + \ldots + \vc{X}_{m},\\
 & J^{(m)}_{\ell} = \left\{\begin{array}{ll}
  J_{m+\ell}, \quad & \ell=0, 1, \ldots, n-m, \\
  J_{\ell - (n-m)}, \quad & \ell=n-m+1, \ldots, n.
  \end{array} \right.
\end{align*}
Obviously, $\{(\vc{Y}^{(m)}_{\ell}, J^{(m)}_{\ell}); \ell = 0, 1,\ldots, n\}$ and $J^{(m)}_{0} = J_{m} = J^{(m)}_{n}$ for $m = 1,2,\ldots, n$ have the same joint distribution for all $m$ under the probability measure in which $\{J_{n}\}$ is stationary. We denote this probability measure by $\dd{P}_{\nu_{+}}$, where $\nu_{+}$ is the stationary distribution of the background process $\{J_{n}\}$. We next consider the following event for $n \ge 1$, $1 \le m \le n$, $\vc{x} > \vc{0}$, $j \in \sr{V}_{+}$ and $B \in \sr{B}(\dd{R}_{+}^{2})$.
\begin{align*}
 & E_{+}(n, m, B) = \{ \min_{1 \le \ell \le n} Y^{(m)}_{\ell 1} > 0, \min_{1 \le \ell \le n} Y^{(m)}_{\ell 2} > 0, \vc{Y}^{(m)}_{n} \in B, J_{0} =J_{n} \},\\
 & E_{2}(n, m, B) = \{ \min_{1 \le \ell \le n} Y^{(m)}_{\ell 2} > 0, \vc{Y}^{(m)}_{n} \in B, J_{0} =J_{n} \}.
\end{align*}
Then, we have
\begin{align*}
 & \cup_{m=1}^{n} E_{+}(n,m,B) \supset E_{2}(n,M,B), \quad \mbox{ for some } M \in \{1,2, \ldots, n\},\\
 & \cup_{m'=1}^{n} E_{2}(n,m',B) \supset \{ \vc{Y}_{n} \in B, J_{0} =J_{n} \},
\end{align*}
where $M$ may be chosen so that $Y^{(m)}_{\ell 1}$ for $m=0,1,\ldots, n$ attains the minimum at $m=M$.

Since $E_{+}(n,m,B)$ has the same probability for any $m$ under $\dd{P}_{\nu_{+}}$ and similarly $E_{2}(n,m',B)$ does so, we have
\begin{align}
\label{eqn:lower bound 2}
  \lefteqn{\dd{P}_{\nu_{+}}( \min_{1 \le \ell \le n} Y_{\ell 1} > 0, \min_{1 \le \ell \le n} Y_{\ell 2} > 0, \vc{Y}_{n} \in B, J_{n} = J_{0})} \hspace{10ex} \nonumber \\
  & \ge \frac 1n \dd{P}_{\nu_{+}}( \min_{1 \le \ell \le n} Y_{\ell 2} > 0, \vc{Y}_{n} \in B, J_{n} = J_{0} ) \nonumber\\
  & \ge \frac 1{n^{2}} \dd{P}_{\nu_{+}}( \vc{Y}_{n} \in B, J_{n} = J_{0}, \vc{Y}_{0} = \vc{0}).
\end{align}
  We next note the Markov modulated version of the well known Cram\'{e}r's theorem (see Theorem 1 of \cite{NeyNumm1987a}). For this, define the Fenchel-Legendre transform of $\log \gamma^{(A_{**})}(\vc{\theta})$ as
\begin{align*}
  \Lambda^{*}(\vc{x}) = \sup_{\vc{\theta} \in \dd{R}^{2}}\{\br{\vc{\theta},\vc{x}} - \log \gamma^{(A_{**})}(\vc{\theta}) \},
\end{align*}
then we have, for any open set $G$ in $\dd{R}^{2}$,
\begin{align}
\label{eqn:Cramer 1}
  \liminf_{n \to \infty} \frac 1{n} \log \dd{P}(\vc{Y}_{n} \in n G, J_{n} = j | J_{0} = i) \ge - \Lambda^{*}(\vc{z}), \qquad i,j \in \sr{V}_{+},  \vc{z} \in G.
\end{align}
  
  Let $\sr{S}_{++} = \sr{U}_{++} \times \sr{V}_{+}$, and let $\sigma_{0} = \inf\{ \ell \ge 1; \vc{L}_{\ell} \in \sr{S} \setminus \sr{S}_{++} \}$. Since the random walk $\{(\vc{Y}_{\ell}, J_{\ell})\}$ is stochastically identical with $\{(\vc{L}_{\ell}, J_{\ell})\}$ as long as they are in $\sr{S}_{++}$, we have, for $\vc{y} \in \dd{Z}_{+}^{2}$ and $G$ such that $G + \vc{z} \subset G$ for each $\vc{z} \in \dd{Z}_{+}^{2}$,
\begin{align}
\label{eqn:lower bound 4}
  \lefteqn{\dd{P}_{\nu_{+}}( \vc{L}_{n} \in n G, \sigma_{0} > n, J_{n} = J_{0} | \vc{L}_{0} = \vc{y} )} \hspace{10ex} \nonumber\\
  & = \dd{P}_{\nu_{+}}( \vc{Y}_{n} \in n G - \vc{y}, \sigma_{0} > n, J_{n} = J_{0} ) \nonumber\\
  & \ge \frac 1{n^{2}} \dd{P}_{\nu_{+}}( \vc{Y}_{n} \in n G, J_{n} = J_{0} ). 
\end{align}

Recall that $\nu \equiv \{\nu(\vc{z}, j); (\vc{z},i) \in \sr{S}\}$ denotes the stationary distribution. For $\vc{z}_{0} = (2,2)$, let
\begin{align*}
  d = \min_{i \in \sr{V}_{+}} \nu(\vc{z}_{0}, i),
\end{align*}
then $d > 0$ since $\{(\vc{L}_{\ell}, J_{\ell})\}$ is irreducible and $\sr{V}_{+}$ is a finite set. Denote the normalized distribution of $\pi$ restricted on $\sr{S} \setminus \sr{S}_{++}$ by $\pi_{0}$, and denote the probability measure for $\{(\vc{L}_{\ell}, J_{\ell})\}$ with the initial distribution $\pi_{0}$ by $\dd{P}_{\pi_{0}}$. Let
\begin{align*}
  G = \{ \vc{\theta} \in \dd{R}^{2}; \vc{\theta} > \vc{c} \},
\end{align*}
which satisfies the requirement of \eqn{lower bound 4}. Then, it follows from the occupation measure representation of the stationary distribution and \eqn{lower bound 2} with $B = G$ that, for any $m, n \ge 1$, $j \in \sr{V}_{+}$ and $\vc{z}_{0} \equiv (2,2) \in \sr{S}_{++}$,
\begin{align*}
  \dd{P}( \vc{L} & \in n G) = \frac 1{E_{\pi}(\sigma_{0})} \sum_{\ell = 1}^{\infty} \dd{P}_{\pi_{0}}( \vc{L}_{\ell} \in n G, \sigma_{0} > \ell) \\
  & \ge \frac 1{E_{\pi}(\sigma_{0})} \dd{P}_{\pi_{0}}( \vc{L}_{m} \in n G, J_{m} = J_{0}, \sigma_{0} > m, \vc{L}_{0} = \vc{z}_{0} ) \\
  & \ge \frac 1{E_{\pi}(\sigma_{0})} \sum_{i \in \sr{V}_{+}} \dd{P}_{\nu_{+}}( \vc{L}_{m} - \vc{L}_{0} \in n G, J_{m} = J_{0}, \sigma_{0} > m| \vc{L}_{0} = \vc{z}_{0}, J_{0} = i) \pi_{0}(\vc{z}_{0}, i) \\
  & \ge \frac d{E_{\pi}(\sigma_{0})} \sum_{i \in \sr{V}_{+}} \dd{P}_{\nu_{+}}( \vc{L}_{m} - \vc{L}_{0} \in n G, J_{m} = J_{0}, \sigma_{0} > m, J_{0} = i| \vc{L}_{0} = \vc{z}_{0}) \\
  & \ge \frac d{E_{\pi}(\sigma_{0})} \dd{P}_{\nu_{+}}( \vc{L}_{m} \in n G, J_{m} = J_{0}, \sigma_{0} > m | \vc{L}_{0} = \vc{z}_{0}) \\
  & = \frac d{E_{\pi}(\sigma_{0})} \dd{P}_{\nu_{+}}( \vc{Y}_{m} \in n G, J_{m} = J_{0}, \sigma_{0} > m)\\
  & \ge \frac d{m^{2} E_{\pi}(\sigma_{0})} \dd{P}_{\nu_{+}}( \vc{Y}_{m} \in n G, J_{m} = J_{0})\\
  & \ge \frac d{m^{2} E(\sigma_{0})} \dd{P}( \vc{Y}_{m} \in n G, J_{m} = j| J_{0} = j) \nu_{+}(j),
\end{align*}
where we have used the facts that the distribution of $\{(\vc{L}_{\ell}, J_{\ell})\}$ is unchanged under the conditional probability measures $\dd{P}_{\pi_{0}}$ and $\dd{P}_{\nu_{+}}$ given $(\vc{L}_{0}, J_{0})$, and similarly $\{\vc{Y}_{\ell}\}$ is unchanged for $\dd{P}_{\nu_{0}}$ and $\dd{P}$ given $J_{0}$.

  Since $\vc{x} \in n G$ is equivalent to $\vc{x} > \vc{c}$, taking logarithms for both sides of the above inequality and letting $m, n \to \infty$ in such a way that $n/m \to t$ for each fixed $t > 0$, \eqn{Cramer 1} yields
\begin{align*}
  \liminf_{n \to \infty} \frac 1{n} \log \dd{P}( \vc{L}> n \vc{c} ) &\ge  \lim_{n \to \infty} \frac m{n} \frac 1m \log \dd{P}\left( \left. \vc{Y}_{m} > m \frac nm \vc{c}, J_{m} = j \right| J_{0} = j \right) \\
  &\ge - \frac 1t \Lambda^{*}(t\vc{c}).
\end{align*}
Since $t> 0$ can be arbitrary, this implies that
\begin{align*}
  \liminf_{x \to \infty} \frac 1{x} \log \dd{P}( \vc{L} > x \vc{c} ) \ge - \inf_{ t > 0} \frac 1t \Lambda(t\vc{c}) = - \sup\{ \br{\vc{\theta}, \vc{c}}; \vc{\theta} \in \Gamma^{(2d)}_{+} \},
\end{align*}
  where the last equality is obtained from Theorem 1 of \cite{BoroMogu1996a} (see also Theorem 13.5 of \cite{Rock1970}).
  
  It remains to prove that $\vc{\theta} \not\in \ol{\Gamma}_{\max}$ implies $\varphi(\vc{\theta}) = \infty$, but its proof is exactly the same as that of Lemma 4.2 of \cite{KobaMiya2014} except for a minor modification. So, we omit it.

\section{\large One dimensional QBD and lower bounds}
\setnewcounter
\label{app:one}

In this section, we prove \thr{lower bound 1}. For this, we apply the Markov additive approach given in Section 5.5 of \cite{Miya2011}. This approach is also taken by \citet{Ozaw2013}, which is essentially the same as that of \citet{Miya2009}. We first formulate the 2d-QBD process as a 1-dimensional QBD process with infinitely many background states, taking one of the half coordinate axis of the lattice quarter plane as level. There are two such QBD processes. Since they are symmetric, we mainly consider the case that the first coordinate is taken as level. Our arguments are parallel to those of \citet{Ozaw2013}, but answers are more tractable because of \thr{super-h 2}.

\subsection{Convergence parameter of the rate matrix}
\label{sect:convergence}

We first consider the convergence parameters of the so called rate matrix $R^{(s)}$ of the one dimensional QBD process $\{(L^{(s)}_{n}, \vc{J}^{(s)}_{n})\}$ for $s =1,2$. This $R^{(s)}$ is defined as the minimal nonnegative solution of the matrix quadratic equation:
\begin{align*}
  R^{(s)} = (R^{(s)})^{2} Q_{-1} + R^{(s)} Q_{0} + Q_{1}.
\end{align*}
Since arguments are symmetric for $s =1$ and $s=2$, we will mainly consider the case for $s=1$. As is well known, the stationary distribution of $P^{(1)}$ is given by
\begin{align}
\label{eqn:geometric s}
  \vc{\pi}^{(1)}_{n} = \vc{\pi}^{(1)}_{1} (R^{(1)})^{n-1}, \qquad n \ge 1,
\end{align}
where $\vc{\pi}^{(1)}_{n} = \{\pi^{(1)}(n,j,k); k \in \sr{V}_{1} \mbox{ for } j=0, k \in \sr{V}_{+} \mbox{ for } j\ge 1 \}$. Then, we can see that the reciprocal of the convergence parameter $c_{p}(R^{(1)})$ gives a lower bound for the decay rate of $\pi^{(1)}(n,j,k)$ for each fixed $j, k$ (e.g., see \cite{Miya2011} for details).

As shown in \cite{Miya2011}, this convergence parameter problem can be reduced to find the right (or left) sub-invariant vector of the matrix moment generating function $Q^{(1)}_{*}(\theta_{1})$ by the Wiener-Hopf factorization for the Markov additive process with transition matrix $\ul{P}^{(1)}$.

Recall that
\begin{align*}
  Q^{(1)}_{*}(\theta_{1}) = \left (\begin{matrix}
A^{(1)}_{*0}(\theta_{1}) & A^{(1)}_{*1}(\theta_{1}) & 0 & 0 & 0 & \cdots \cr
A^{(1)}_{*(-1)}(\theta_{1}) & A_{*0}(\theta_{1}) & A_{*1}(\theta_{1}) & 0 & 0 & \cdots  \cr
 0 & A_{*(-1)}(\theta_{1}) & A_{*0}(\theta_{1}) & A_{*1}(\theta_{1}) & 0 & \cdots \cr
 0 & 0 & A_{*(-1)}(\theta_{1}) & A_{*0}(\theta_{1}) & A_{*1}(\theta_{1}) & \cdots \cr
 \vdots & \vdots & \ddots & \ddots & \ddots & \cdots 
 \end{matrix} \right ).
\end{align*}
Let
\begin{align}
\label{eqn:C *0}
  C^{(1)}_{*0}(\theta_{1}) = A^{(1)}_{*(-1)}(\theta_{1}) (I - A^{(1)}_{*0}(\theta_{1}))^{-1} A^{(1)}_{*1}(\theta_{1}) + A_{*0}(\theta_{1}),
\end{align}
and define the canonical form of $Q_{*}^{(1)}(\theta_{1})$ as
\begin{align*}
  \ol{Q}_{*}^{(1)}(\theta_{1}) = \left (\begin{matrix}
C^{(1)}_{*0}(\theta_{1}) & A_{*1}(\theta_{1}) & 0 & 0 & 0 & \cdots \cr
A_{*(-1)}(\theta_{1}) & A_{*0}(\theta_{1}) & A_{*1}(\theta_{1}) & 0 & 0 & \cdots  \cr
 0 & A_{*(-1)}(\theta_{1}) & A_{*0}(\theta_{1}) & A_{*1}(\theta_{1}) & 0 & \cdots \cr
 0 & 0 & A_{*(-1)}(\theta_{1}) & A_{*0}(\theta_{1}) & A_{*1}(\theta_{1}) & \cdots \cr
 \vdots & \vdots & \ddots & \ddots & \ddots & \cdots 
 \end{matrix} \right ).
\end{align*}
Similarly, $\ol{Q}_{*}^{(2)}(\theta_{2})$ is defined. It is easy to see that $\ol{Q}_{*}^{(s)}(0)$ is stochastic for $s=1,2$.

Thus, $\ol{Q}_{*}^{(s)}(\theta_{s})$ is a nonnegative matrix with QBD block structure, and therefore we can apply \thr{super-h 2}. For this, we note the following fact.

\begin{lemma} 
\label{lem:Gamma convex 2}
  For $s=1,2$, $\Gamma^{(2d)}_{s+}$ is a nonempty and bounded convex subset of $\dd{R}^{2}$.
\end{lemma}

This lemma is proved similarly to \lem{Gamma convex 1} using the fact that $\vc{0} \in \Gamma^{(2d)}_{s+}$ for $s=1,2$. The following result is immediate from \thr{super-h 2}.

\begin{lemma}
\label{lem:2d-QBD super-h}
  Under the assumptions of \thr{lower bound 1}, $Q^{(1)}_{*}(\theta_{1})$, equivalently, $\ol{Q}^{(1)}_{*}(\theta_{1})$, has a superharmonic vector for each $\theta_{1} \in \dd{R}$ if and only if the following two conditions hold.
\begin{itemize}
\item [(i)] $c_{p}(A^{(1)}_{*0}(\theta_{1})) > 1$.
\item [(ii)] There exists a $\theta_{2} \in \dd{R}$ such that $\vc{\theta} \equiv (\theta_{1}, \theta_{2}) \in \Gamma^{(2d)}_{1+}$, equivalently, $\vc{\theta} \in \Gamma^{(2d)}_{1e}$.
\end{itemize}
By symmetry, a similar characterization holds for $Q^{(2)}_{*}(\theta_{2})$.
\end{lemma}

It follows from this lemma and  the Wiener-Hopf factorization that, for $s = 1,2$,
\begin{align}
\label{eqn:cp s 1}
  \log c_{p}(R^{(s)}) = \sup\{ \theta_{s} \ge 0; c_{p}(Q^{(s)}_{*}(\theta_{s})) \ge 1\} = \theta^{(s,\cp)}_{s},
\end{align}
as long as $c_{p}(A^{(s)}_{*0}(\theta^{(s,\cp)}_{s})) > 1$. We are now ready to accomplish our main task.

\subsection{The proof of \thr{lower bound 1}}
\label{sect:proof}

From \eqn{geometric s}, \eqn{cp s 1} and Caucy-Hadamard inequality (e.g., see Theorem 14.8 of Volume I of \cite{Mark1977}), we have the following lower bound.
\begin{align}
\label{eqn:lower bound 5}
  \liminf_{n \to \infty} \frac 1n \log \dd{P}(L_{s} > n, L_{3-s} = \ell, J = k) \ge - \theta^{(s,\cp)}_{s}, \qquad s=1,2.
\end{align}
By \lem{lower bound 2}, this lower bound is tight if $\theta^{(s,\cp)}_{s} = \theta^{(s,\max)}_{s}$ because $\vc{\theta}^{(s,\max)} \in \ol{\Gamma}^{{(2d)}}_{+}$. Thus, it remains to consider the case that $\theta^{(s,\cp)}_{s} < \theta^{(s,\max)}_{s}$. In this case, it follows from \thr{super-h 2} and \lem{t-invariant m 1} that $Q^{(s)}(\theta^{(s,\cp)}_{s})$ is 1-positive, which is equivalent to the fact that $e^{\theta^{(s,\cp)}_{s}} R^{(s)}$ is 1-positive by the Wiener-Hopf factorization. We consider Categories (I) and (II-1), separately, for $s=1$. This is sufficient for the proof because Category (II-2) is symmetric to Category (II-1).

Assume that the 2d-QBD process is in Category (I) and that $\theta^{(1,\cp)}_{1} < \theta^{(1,\max)}_{1}$. In this case $\tau_{s} = \theta^{(s,\cp)}_{s}$ for $s=1,2$. Hence, \eqn{lower bound 5} implies \eqn{lower bound 1}. To prove \eqn{direction 1}, we apply Theorem 4.1 of \cite{MiyaZhao2004} (see also Theorem 2.1 of \cite{LiMiyaZhao2007} or Proposition 3.1 of \cite{Miya2009}). For this, we consider the left and right nonnegative invariant vectors of $Q_{*}^{(1)}(\theta^{(1,\cp}_{1})$, which is a nonnegative matrix with QBD structure and unit convergence parameter.

Since $\varphi(\tau_{1}-\epsilon,0) < \infty$ for any $\epsilon > 0$, we have, similarly to the proof of \thr{upper bound 1},
\begin{align*}
  \limsup_{n \to \infty} \frac 1n \log \dd{P}(L_{1} > n) \le - \tau_{1} = \theta^{(1,\cp)}_{1}.
\end{align*}

 then We now consider the matrix geometric form of the stationary distribution:
\begin{align*}
  \vc{\pi}^{(1)}_{n} \vc{u} &= \vc{\pi}^{(1)}_{1} (R^{(1)})^{n-1} \vc{u}\\
  &= \sum_{k, i}  \sum_{\ell, j} \frac {\vc{\pi}^{(1)}_{1}(k,i)} {x_{ki}} x_{ki} (R^{(1)})^{n-1}_{(ki)(\ell j)} u_{\ell j}\\
  &= e^{-\alpha (n-1)} \sum_{k, i}  \sum_{\ell, j}\frac {\vc{\pi}^{(1)}_{1}(k,i)} {x_{ki}} x_{ki} (e^{\alpha} R^{(1)})^{n-1}_{(ki)(\ell j)} \frac {1} {x_{\ell j}} x_{\ell j} u_{\ell j}\\
  &= e^{-\alpha (n-1)}  \sum_{\ell, j} x_{\ell j} u_{\ell j} \sum_{k, i} (\tilde{G}^{(1+,\alpha)})^{n-1}_{(\ell j)(ki)} \frac {\vc{\pi}^{(1)}_{1}(k,i)} {x_{ki}} \\
\end{align*}

\subsection*{Acknowledgements}
The author is grateful to an anonymous referee for helpful suggestions to improve exposition. This work is supported by Japan Society for the Promotion of Science under grant No.\ 24310115. A part of this work was presented at a workshop of Sigmetrics 2014 (see its abstract \cite{Miya2014a}).

\def\cprime{$'$} \def\cprime{$'$} \def\cprime{$'$} \def\cprime{$'$}
  \def\cprime{$'$} \def\cprime{$'$} \def\cprime{$'$}


\end{document}